\documentclass[12pt]{amsart}

\usepackage[utf8]{inputenc}
\usepackage[T1]{fontenc}
\usepackage[ngerman, english]{babel}

\usepackage{graphicx}
\usepackage{amssymb, amsmath, amsthm, mathtools}
\usepackage{setspace}
\usepackage{color, soul}
\usepackage[linktocpage]{hyperref}
\usepackage{fancyhdr}
\usepackage{xr} 
\usepackage{stmaryrd} 
\usepackage{enumerate}   
\usepackage[toc,page]{appendix}  
 
\usepackage{enumerate}
\usepackage{scalerel,stackengine}

\usepackage{cleveref}

\usepackage{contour}
\contourlength{1pt}
\usepackage{blindtext} 
\usepackage{pict2e}

\usepackage{cancel} 

\usepackage{environ} 

\NewEnviron{elaboration}{
\par
\begin{tikzpicture}
\node[rectangle,minimum width=0.9\textwidth] (m) {\begin{minipage}{0.85\textwidth}\BODY\end{minipage}};
\draw[dashed] (m.south west) rectangle (m.north east);
\end{tikzpicture}
}

\newtheorem{lemma}{Lemma}[section]
\newtheorem{proposition}[lemma]{Proposition}
\newtheorem{theorem}[lemma]{Theorem}
\newtheorem{claim}[lemma]{Claim}

\newtheorem{problem}[lemma]{Problem}

\newtheorem{remark}[lemma]{Remark}

\newtheorem{corollary}[lemma]{Corollary}
\newtheorem{definition}[lemma]{Definition}

\newtheorem{thm}{Theorem}

\newtheorem{conjecture}[thm]{Conjecture}

\renewcommand{\t}{\tilde}
\newcommand{\ch}{\check}

\newcommand{\ol}{\overline}
\newcommand{\cc}{\circ}

\newcommand{\E}{\mathbb{E}}

\newcommand{\s}{\mathbb{S}}
\newcommand{\C}{\mathbb{C}}
\newcommand{\D}{\mathbb{D}}
\newcommand{\R}{\mathbb{R}}

\newcommand{\Z}{\mathbb{Z}}
\renewcommand{\SS}{\mathbb{S}}

\newcommand{\col}{\colon}

\newcommand{\infi}{\infty}
\newcommand{\diam}{\operatorname{diam}}

\newcommand{\HHH}{\mathcal{H}}

\newcommand{\MCG}{{\rm MCG}}

\newcommand{\ep}{\epsilon}

\newcommand{\tc}{\textcolor}

\definecolor{britishracinggreen}{rgb}{0.0, 0.26, 0.15}
\definecolor{ao(english)}{rgb}{0.0, 0.5, 0.0}
\definecolor{applegreen}{rgb}{0.55, 0.71, 0.0}
\definecolor{ao(english)}{rgb}{0.0, 0.5, 0.0}
\definecolor{byzantine}{rgb}{0.74, 0.2, 0.64}
\definecolor{amber(sae/ece)}{rgb}{1.0, 0.49, 0.0}

\date{\today}

\usepackage[all,cmtip]{xy}
\usepackage{overpic}
 \usepackage{float}
 \usepackage{fancybox}
\usepackage[usenames,dvipsnames]{xcolor} 
\usepackage[mathscr]{euscript}
\usepackage{tikz}
\usetikzlibrary{shapes.geometric, arrows}
\usetikzlibrary{fit,calc,positioning}

\newcommand{\TT}{\mathsf{T}}

\newcommand{\QF}{\mathsf{QF}}

\newcommand{\LLL}{\mathcal{L}}
\newcommand{\RRR}{\mathcal{R}}
\newcommand{\PPP}{\mathcal{P}}
\newcommand{\FFF}{\mathcal{F}}

\newcommand\rul{\bgroup\markoverwith{\textcolor{red}{\rule[-0.5ex]{2pt}{0.4pt}}}\ULon}

\newcommand{\PSL}{{\rm PSL}}

\newcommand{\CP}{\mathbb{C}{\rm  P}}

\newcommand{\Mod}{\operatorname{Mod}}
\newcommand{\Hol}{\operatorname{Hol}}

\newcommand{\Fix}{\operatorname{Fix}}

\newcommand{\Ep}{\operatorname{Ep}}

\newcommand{\hol}{\operatorname{hol}}
\newcommand{\length}{\operatorname{length}}
\newcommand{\Area}{\operatorname{Area}}
\newcommand{\Axis}{\operatorname{Axis}}

\renewcommand{\H}{\mathbb{H}}

\newcommand{\del}{\delta}
\newcommand{\kap}{\kappa}
\newcommand{\alp}{\alpha}

\newcommand{\gam}{\gamma}

\newcommand{\PP}{\mathsf{P}}

\DeclareRobustCommand{\rchi}{{\mathpalette\irchi\relax}}
\newcommand{\irchi}[2]{\raisebox{\depth}{$#1\chi$}}

\newcommand{\tr}{\operatorname{ tr}}
\newcommand{\dev}{\operatorname{dev}}
\renewcommand{\Im}{\operatorname{Im} }
\newcommand{\Stab}{\operatorname{Stab}}

\newcommand{\bdr}{\partial}

\newcommand{\ML}{\mathsf{ML }}

\newcommand{\sub}{\subset}
\newcommand{\til}[1]{\tilde{#1}} 
\newcommand{\ti}{\tilde}

\newcommand{\minus}{\setminus}

  \newcommand{\Label}[1]{\label{#1}\textcolor{green}{\tiny #1} }
 \renewcommand{\Label}[1]{\label{#1}}

\newcommand{\Qed}[1]{\nopagebreak[4]{\tiny \hfill
\fbox{\ref{#1}} \linebreak }\pagebreak[2]}

\newcommand\rwh[1]{\arraycolsep=0pt\relax%
\begin{array}{c}
\stretchto{
  \scaleto{
    \scalerel*[\widthof{\ensuremath{#1}}]{\kern-.5pt\bigwedge\kern-.5pt}
    {\rule[-\textheight/2]{1ex}{\textheight}} 
  }{\textheight} %
}{0.5ex}\\       
#1\\                
\rule{-1ex}{0ex}
\end{array}
}

\stackMath
\newcommand\reallywidehat[1]{%
\savestack{\tmpbox}{\stretchto{%
  \scaleto{%
    \scalerel*[\widthof{\ensuremath{#1}}]{\kern-.6pt\bigwedge\kern-.6pt}%
    {\rule[-\textheight/2]{1ex}{\textheight}}
  }{\textheight}%
}{0.5ex}}%
\stackon[1pt]{#1}{\tmpbox}%
}
\definecolor{dblue}{cmyk}{1, 0, 0, .7}

\setstcolor{red} 

 \definecolor{calpolypomonagreen}{rgb}{0.12, 0.3, 0.17}
 
 \definecolor{darkbyzantium}{rgb}{0.6, 0.2, 0.3}
 \definecolor{azure}{rgb}{0.0, 0.5, 1.0}
 \definecolor{cittingcolor}{cmyk}{60,0,10,0}
  
 \hypersetup{pdfborder=0 0 .3}
	\hypersetup{
    colorlinks=true,
   linkcolor=darkbyzantium,
    filecolor=magenta,      
    urlcolor=NavyBlue,
    citecolor = calpolypomonagreen,
    pdftitle={Overleaf Example},
    pdfpagemode=FullScreen,
    }

\usepackage{fancyhdr}
\usepackage[yyyymmdd]{datetime}
 \usepackage[many]{tcolorbox}    
\newtcolorbox{mybox}[1]{%
    tikznode boxed title,
    enhanced,
    arc=0mm,
    interior style={white},
    attach boxed title to top center= {yshift=-\tcboxedtitleheight/2},
    fonttitle=\bfseries,
    colbacktitle=white,coltitle=black,
    boxed title style={size=normal,colframe=white,boxrule=0pt},
    title={#1}}

\begin{document}
 \title[\today]{Neck-pinching of $\CP^1$-structures  in  the $\PSL_2\C$-character variety}

\author{Shinpei Baba}
\address{Osaka University 
}

\email{sb.sci@osaka-u.ac.jp}

\maketitle
\begin{abstract}
We characterize a certain neck-pinching degeneration of (marked) $\CP^1$-structures on a closed oriented surface $S$ of genus at least two.
In a more general setting, we take a path of  $\CP^1$-structures $C_t ~(t \geq 0)$ on $S$ which leaves every compact subset in its deformation space, such that the holonomy of $C_t$ converges in the $\PSL_2\C$-character variety as $t \to \infi$. 
Then it is well known that the complex structure $X_t$ of $C_t$ also leaves every compact subset in the Teichmüller space of $S$. 
In this paper, under an additional assumption that $X_t$ is pinched along a loop $m$ on $S$, we describe the limit of $C_t$ from different perspectives:  namely, in terms of the developing maps,  holomorphic quadratic differentials, and pleated surfaces.

The holonomy representations of $\CP^1$-structures on $S$ are known to be non-elementary (i.e. strongly irreducible and unbounded). 
We also give a rather exotic example of such a path  $C_t$ whose limit holonomy is the trivial representation.
\end{abstract}

\setcounter{tocdepth}{1}
\tableofcontents

\section{Introduction}
Let $S$ be a (connected) closed oriented surface of genus at least two, throughout this paper.
For a  (marked) $\CP^1$-structure $C$ on $S$,  the holonomy of $C$ is a homomorphism $\pi_1(S) \to \PSL_2\C$ uniquely determined  up to conjugation by $\PSL_2\C$; see \S \ref{sCPS}. 
This correspondence yields the {\it holonomy map}
\begin{equation*}
\Hol \col \PP \to \rchi,
\end{equation*}
where $\PP \,(\cong \R^{12g -12})$ is the deformation space of all $\CP^1$-structures on $S$ and $\chi$ is the $\PSL_2\C$-character variety of $S$. 
Note that there are many $\CP^1$-structures whose holonomy is not discrete. 

Hejhal  \cite{Hejhal-75} proved that $\Hol$ is a local homeomorphism (moreover, it is a local biholomorphic map \cite{Hubbard-81}, \cite{Earle-81}). 
 However, it is not a covering map onto its image (\cite{Hejhal-75}).
Thus it is a natural question to ask how the path-lifting property fails: 
\begin{problem}(Kapovich \cite[Problem 1]{Kapovich-95}, see also \cite[Problem 12.5.1]{Gallo-Kapovich-Marden}.)\Label{problem}
Let $C_t   ~(t > 0)$ be a path of $\CP^1$-structures on $S$ such that 
\begin{enumerate}
\item $C_t$ leaves every compact subset in $\PP$ at $t \to \infi$, and \Label{iCDiverges}
\item the holonomy $\eta_t \in \rchi$ of $C_t$ converges to $\eta_\infi \in \rchi$ as $t \to \infi$. \Label{iRhoConverges}
\end{enumerate}
 What is the asymptotic behavior of $C_t$? 
\end{problem}
In this paper,  we give various limiting behaviors to answer Question \ref{problem} in the ``neck-pinching'' case. 

\subsection{Pinching loops on Riemann surfaces}
For each $t \geq 0$, let $X_t$ denote the complex structure on $S$ induced by $C_t$. 
Then, by the work of Kapovich (\cite{Kapovich-95}, see also \cite{Gallo-Kapovich-Marden, Dumas18HolonomyLimit} ),  the conditions (\ref{iCDiverges}) and (\ref{iRhoConverges}) imply that  $X_t$ must also leave every compact subset in the  Teichm\"uller space $\TT$  (see Corollary \ref{DivergenceInT}).

We focus on the following basic type of degeneration of $X_t$.
Given a path $X_t \in \TT ~ $, $X_t$ is {\it pinched along a loop} $m$ if 
\begin{itemize}
\item $\length_{X_t} m \to 0$, and
\item  if an essential loop $\ell$ in $S \minus m$ is {\it not} homotopic to $m$, then $\length_{X_t} \ell$ is bounded between two positive numbers for all $t \geq 0$. 
\end{itemize}
Here $``\length_{X_t}"$ is either the extremal length of $X_t$ or the hyperbolic length of the uniformization of $X_t$. 
(In the  augmented Teichmüller space, this definition of pinching is equivalent to saying that $X_t$ accumulates to a compact subset of the boundary stratum corresponding to $m$ being pinched.)

{\it A multiloop} is a union of disjoint finitely many essential simple closed curves.
Then, similarly, we say that $X_t$ is {\it pinched along a multiloop} $M$ on $S$, if, 
\begin{itemize}
\item  for each loop $m$ of $M$, $\length_{X_t} m \to 0$ as $t \to \infi$, and 
\item for each loop $\ell$ in $S \minus M$ not homotopic to a loop of $M$,  $\length_{X_t} \ell$ is bounded between two positive numbers for all $t \geq 0$.
\end{itemize}

   The {\it quasi-Fuchsian representation} $\pi_1(S) \to \PSL_2\C$ is a discrete faithful representation whose limit set is a Jordan curve in $\CP^1$, the quasi-Fuchsian Space $\QF$ is an open subset of the character variety $\rchi$. 
  There is no path $C_t$ in Problem \ref{problem}, whose limit holonomy $\eta_\infi$ is in $\QF$.    
On the other hand, a dense subset of the boundary of $\QF$ consists of holonomy representations of $\CP^1$-structures  pinched along loops (\cite{McMullen91}), and it has been quite important to study such degeneration for the study of Kleinian groups.

 \subsection{Asymptotic behaviors} 
One of our main results is that 
$\tr \eta_\infi(m)$ must be $\pm 2$. 
In other words,  the holonomy along $m$ at $t = \infi$ corresponds to either (i) a parabolic element (which is not the identity) or (ii) the identity of $\PSL_2\C$.
We will describe,  in both Cases (i) and (ii),  the asymptotic behavior of $C_t$  from three different perspectives of $\CP^1$-structures:
 \begin{enumerate}[(A)]
 \item A holomorphic quadratic differential on a marked Riemann surface homeomorphic to $S$ ({\it Schwarzian parameters}).
 \item  A hyperbolic structure on $S$ and a measured lamination, which induces an equivariant pleated surface $\H^2 \to \H^3$ ({\it Thurston parameters}).
  \item A developing map $f\col \til{S} \to \CP^1$ and a holonomy representation $\rho \col \pi_1(S) \to \PSL_2\C$. ({\it Developing pair})
\end{enumerate}

The {\it residue} of a meromorphic quadratic differential $q$ at a pole is the integral of $\pm \sqrt{q}$ around the pole, which is well-defined up to sign (see  \cite{GuptaWolf19}).
Given a pole of order two, letting $r$ be its residue, $q$ is expressed as  $r^2 / z^{-2}  dz^2$ for an appropriate parametrization in a neighborhood of the pole (see \cite[Theorem 6.3]{Strebel84QuadraticDifferentials}). 

Let $X$ be a nodal Riemann surface, and let $\mathring{X}$ be the smooth part of  $X$.
Then the {\it normalization $\ol{X}$ of $X$} is the smooth Riemann surface together  with a continuous map $\xi\col \ol{X} \to X$ such that $\xi$ is a biholomorphic in $\xi^{-1}(\mathring{X})$ and for each node $p$ of $X$, $\xi^{-1}(p)$ consists of exactly two points. 
    A {\it regular quadratic differential} on $X$ is a meromorphic quadratic differential $\bar{q}$ on $\ol{X}$ such that
    \begin{itemize}
    \item every pole of $\bar{q}$ has an order at most two and it maps to a node of $X$, and
    \item if $z_1, z_2$ on $\ol{Z}$ map to the same node on $X$, then the residue around $z_1$ is equal to that of $z_2$
\end{itemize}
(see \cite{Bers73Spaces_of_degenerating_Riemann_surfaces} \cite{Loftin_Zhang(19)}).

For Perspective (A),
 the path $C_t$ corresponds to  a path of pairs $(X_t, q_t),~ t \geq 0$ in Schwarzian coordinates, where $X_t$ is a marked Riemann surface homeomorphic to $S$ and $q_t$ is a holomorphic quadratic differential $q_t$ on $X_t$ for all $t \geq 0$.

\begin{thm}\Label{SchwarzianLimit}
\begin{itemize}
\item
Suppose that $X_t$ is pinched along a loop $m$. 
Then, exactly one of the following holds:
\begin{enumerate}[(i)]
\item $X_t$ converges to a nodal Riemann surface $X_\infi$ with a single node, and $q_t$ converges to a regular quadratic differential on $X_\infi$ such that the node is at worst a pole of order one (Theorem \ref{LimitiDifferentialParabolicCusp}.)\Label{iStrongParabolicConvergence}
\item For every diverging sequence $0 \leq t_1 < t_2 < \dots$, up to a subsequence, $X_{t_i}$ converges to a nodal Riemann surface $X_\infi$ with a single node and $q_{t_i}$ converges to a regular quadratic differential $q_\infi$ on $X_\infi$ such that the residue of each pole is a   non-zero integral multiple of  $\sqrt{2}\pi $.\Label{iParabolicDiffferential}  (Theorem \ref{Residue}.)
\end{enumerate}
\item   
Suppose that $X_t$ is pinched along a multiloop $M$ consisting of $n$ loops. 
Then, for every diverging $t_1 < t_2 < \dots$, there is a subsequence such that $X_{t_i}$ converges a nodal Riemann surface $X_\infi$ with $n$ nodes and $q_t$ converges to a meromorphic quadratic differential $q_\infi$  on $X_\infi$   such that each node of $X_\infi$ is, at most,  a pole of order two.
(Corollary \ref{DifferentialConvertesMultiloop}.)
\end{itemize}
\end{thm}

The convergence of the holomorphic quadratic differential in Theorem  \ref{SchwarzianLimit} is normal convergence, and in particular, the $\CP^1$-structure $C_t$  converges to the $\CP^1$-structure corresponding to $(X_\infi, q_\infi)$ minus the node, uniformly on every compact subset.

The space of homomorphisms $\pi_1(S) \to \PSL_2\C$ is called the {\it representation variety}, and
the {\it character variety} $\rchi$ is the GIT-quotient of the representation variety (see \S\ref{sPathLiftingProperty}).  
In order to obtain an equivariant object as a limit of $C_t$,  we pick a (continuous) lift $\rho_t \col \pi_1(S) \to \PSL_2\C$ of $\eta_t \in \rchi$, such that $\rho_t$ converges,  as $t \to \infi$,  to a homomorphism $\rho_\infi\col \pi_1(S) \to \PSL_2\C$ which maps to $\eta_\infi$.
In fact, we prove the existence of such a lift in Proposition \ref{LiftingHolonomyPath},  since it is not obvious when $\eta_\infi$ is an elementary representation. 

Note that for every discrete faithful representations $\pi_1(S) \to \PSL_2\C$, there is a unique equivariant continuous map $\bdr_\infi \pi_1(
S) \cong \s^1 \to \CP^1$ called the Cannon-Thurston map (\cite{Mj14}).
 This map is closely related to the question which we consider, by identifying the ideal boundary of $\til{S}$ with $\s^1$. 
 
Let $N$ be a regular neighborhood of the loop $m$ in $S$.
For $ t \geq 0$, 
let $C_t \cong (\tau_t, L_t)$ be Thurston parameters, where $\tau_t$ is a path of marked hyperbolic structures on $S$ and $L_t$ is a path of measured laminations on $S$ (\S \ref{sThurston}). 
Fixing a marking $\iota_t\col S \to \tau$ in its isotopy class,   $(\tau_t, L_t)$ yields to a $\rho_t$-equivariant pleated surface $\beta_t \col \til{S} \cong \H^2 \to \H^3$, which changes continuously in $t \geq 0$.
Then, in fact, $\beta_t$ converges to a continuous equivariant map: 
\begin{thm}\Label{LimitInPleatedSurface}
Suppose that  $X_t$ is pinched along a loop $m$. 
Then, by taking an appropriate path of markings $\iota_t\col S \to \tau_t~ (t \geq 0)$ , exactly one of the following holds: 
\begin{enumerate}[(i)]
\item $\rho_\infi(m) \in \PSL_2\C$ is a parabolic element, and $\beta_t\col\ti{S} \to \H^3$ converges to a $\rho_\infi$-equivariant continuous map  $\beta_\infi \col \til{S}  \to \H^3 \cup \CP^1$ uniformly on compact subsets, such that $\beta^{-1}_\infi(\CP^1)$ is a $\pi_1(S)$-invariant multicurve on $\ti{S}$ which is $\pi_1(S)$-equivariantly homotopic to $\phi^{-1}(m)$, where $\phi\col \ti{S} \to S$ is the universal covering map.  (Theorem \ref{LimitPleatedSurfaceParabolic}).  \Label{iDevParaboclicLimit}
\item $\rho_\infi(m)$ is the identity in $\PSL_2\C$, and, for every sequence $0 \leq t_1 < t_2 < \dots$ diverging to $\infty$,  up to a subsequence,  $\beta_{t_i}\col \ti{S} \to \H^3$ converges to a $\rho_\infi$-equivariant continuous map $\beta_\infi \col \til{S} \to \H^3 \cup \CP^1$ such that $\beta_\infi^{-1} (\CP^1)$ descends either to the loop $m$ or to a subsurface isotopic to one or two components of $S\minus N$ (\S\ref{smIPleatedSurface}.)
 \Label{iDevIdentityLimit}
\end{enumerate}
\end{thm}

Let $f_t\col \til{S} \to \CP^1$ be the developing map of $C_t$ which is a $\rho_t$-equivariant local homeomorphism. 
As $C_t$ changes continuously in $t$,  we may assume that $f_t$ also changes continuously in $t \geq 0$. 
Such a family $(f_t)$ is unique up to a path of isotopies $S \to S$ in $t \geq 0$ homotopic to the identity.

Pick a regular neighborhood  $N$of $m$. 
Pick a component $\ti{N}$ of $\phi^{-1}(N)$. 
 By abuse of notation, we regard the loop $m$ also as the element of $\pi_1(S)$ which preserves $\ti{N}$.  
We show that the developing map $f_t$ converges in the complement of $\phi^{-1}(N)$, and the asymptotic behavior on $\bdr \phi^{-1}(N)$ is well controlled by  the holonomy $\rho_t(m)$. 
Hyperbolic structures are in particular $\CP^1$-structures. 
If a hyperbolic surface has a cusp, it has a neighborhood obtained by quotienting a horodisk in $\H^2$ by the cyclic group generated by a parabolic holonomy around the puncture. 
\begin{thm}\Label{LImitInDevelopingMap} 
Suppose that  $X_t$ is pinched along a loop $m$. 
Then, by an appropriate isotopy of $S$ in $t \geq 0$ homotopic to the identity,  exactly one of (i) and (ii) holds.    
\begin{enumerate}[(i)]
\item 
\begin{itemize}
\item $\rho_\infi(m)$ is parabolic;
\item  the cusps of $C_\infi$ have horodisk quotient neighborhoods;
\item    $f_t\col \ti{S} \to \CP^1$  converges to  a $\rho_\infi$-equivariant continuous map $f_\infi\col \til{S} \to \CP^1$ uniformly on compact subsets, and moreover, there is a multiloop $M$ which is a union of finitely many parallel copies of $m$ such that $f_\infi$ is a local homeomorphism exactly on $\til{S} \minus \phi^{-1}(M)$, 
and   $f_\infi$ takes each component $\ti{m}$ of $\phi^{-1}(M)$  to its corresponding parabolic fixed point (Theorem \ref{ParabolicDevConverges}).      
\end{itemize}

\item $\rho_\infi(m) = I$, and
for every diverging sequence $t_1 < t_2 < \dots$, up to a subsequence,
\begin{itemize}
\item  the restriction of $f_{t_i}$ to $\til{S} \minus \phi^{-1}(N)$ converges to a $\rho_\infi$-equivariant continuous map $f_\infi \col \til{S} \minus \phi^{-1}(N) \to \CP^1$, and
\item  $\Axis(\rho_{t_i}(m))$ converges to a geodesic in $\H^3$ or a point in $\CP^1$ so that $f_\infi$ takes the boundary components of $\ti{N}$ onto the ideal points (in $\CP^1$) of  $\lim_{i \to \infi} \Axis(\rho_{t_i}(m))$ (Theorem \ref{TrivialNeck}), where $\Axis(\rho_{t_i}(m))$ is the convex hull of the fixed point on $\CP^1$ (\Cref{axis}). 
\end{itemize}
\end{enumerate} 
\end{thm}
\begin{remark}
If a general $\CP^1$-structure has a cusp with parabolic peripheral holonomy, there is its cusp neighborhood isomorphic to either a horodisk quotient or a grafting of a horodisk quotient. (See Proposition \ref{CuspClassification}.) 

A (2$\pi$-)grafting is a cut-and-paste operation of a $\CP^1$-structure, and it yields a new $\CP^1$-structure with the same holonomy, by inserting an appropriate cylinder along an (admissible) loop  (\cite{Goldman-87}, see also  \cite{Kapovich-01, Baba20ThurstonParameter}).
Let $n$ be the number of parallel copies of $m$ constituting $M$ in (i). 
Then there is another diverging path $C_t'$ of $\CP^1$-structure on $S$ with holonomy $\rho_t$ and a path of admissible loops $m_t'$ on $C_t'$ for $t \gg 0$ such that $C_t$ is obtained by $2\pi (n -1)$-grafting of $C_t'$.
\end{remark}

In fact,  Cases (i) and (ii) in Theorem \ref{SchwarzianLimit}, Theorem \ref{LimitInPleatedSurface}, and Theorem \ref{LImitInDevelopingMap} correspond. 
In particular, the Type (i) degeneration occurs on the boundary of the quasi-Fuchsian space, by pinching a loop on a Bers slice.

On the other hand, Type (ii) degeneration is new indeed. 
In particular, $\eta_t$ must be a non-discrete representation for all sufficiently large $t > 0$, possibly except at $t = \infty$ (Theorem \ref{EventuallyNondiscrete}).
Notice that if the peripheral loop of a cusp of a $\CP^1$-structure has trivial holonomy, then the $\CP^1$-structure can be deformed without changing its holonomy (of the entire surface), by moving the cusp  (c.f. Theorem \ref{HolImmersion}). 
Then, since $\rho_\infi(m) = I$, therefore it is necessary to take a subsequence. 
In \S \ref{sExoticDegeneration}, we give examples of Type (ii) degenerations.

Next, we explain a certain uniform bound of $C_t$, which yields the convergence of $C_t$ away from the pinched loop $m$.
This uniform bound holds for a more general path $C_t$ with a multiloop being pinched.
 The integration of $\sqrt{q_t}$ along paths on $X_t$ yields a singular Euclidean structure  $E_t$ on $X_t$ such that  a zero of order $d$ of $q_t$ is the singular point of cone angle $(d/2 + 1)\pi$ of $E_t$ (see for example, \cite{FarbMarglit12, Strebel84QuadraticDifferentials}).  
Recall that the upper injectivity radius of $E_t$ is the supremum of the injectivity radii over all points in $E_t$  (as $E_t$ is compact, it is indeed maximum).
    \fontsize{13pt}{12pt}\selectfont
\begin{thm}(Theorem \ref{UpperInjectivityRadiusBound})\Label{BoundingUpperInjectivityRadius}
Suppose that $X_t$ is pinched along a multiloop. 
Then the {\it upper injectivity radius} of $E_t$ for all $t \geq 0$ is bounded from above. 
\end{thm}

It is a classical theorem that the holonomy map $\Hol$ is a local homeomorphism for the closed surface $S$.  
In the limit of $C_t$, we have a $\CP^1$-structure with cusps, such that cusp points are at most poles of order two in the Schwarzian coordinates.
The holonomy theorem is proved for such $\CP^1$-surfaces cusps by Luo (\cite{LuoFeng93MonodromyGroupsOfProjectiveStructure}) if punctures have non-trivial peripheral holonomy.
In this paper, we prove a more general holonomy theorem (Theorem \ref{HolImmersion}) for the developing pairs of $\CP^1$-structures allowing trivial holonomy around punctures.
We apply this holonomy theorem for the convergence on $C_t$ in every thick part as $t \to \infty$. 
This holonomy theorem is given by appropriately enlarging the character variety, and this enlargement is a certain ramification of the framed representation space introduced by Fock and Goncharov (\cite{FockGoncharov06}).
 (For recent developments on $\CP^1$-structure corresponding to higher order poles, see \cite{GuptaMj21, Allegretti_Bridgeland_20}.)

Gallo, Kapovich, and Marden algebraically characterized the image of $\Hol$; in particular, it is almost onto one of the two components of the character variety $\rchi$ (\cite{Gallo-Kapovich-Marden}).
To be more precise, $\rho \col \pi_1(S) \to \PSL_2\C \in \Im \Hol$ if and only if $\Im \rho$ is non-elementary and $\rho$ lifts to a homomorphism from $\pi_1(S)$ into ${\rm SL}(2, \C)$. 
As an example of  Type (ii) degeneration, we construct a path $C_t$ whose holonomy limits to an elementary representation, or even to the trivial representation in the representation variety (\S \ref{sExoticDegeneration}).
 
 If the holonomy of a $\CP^1$-structure around a puncture is trivial, as stated above, the $\CP^1$-structure can be deformed around the puncture without changing the holonomy of the entire surface.  
 A non-elementary subgroup of $\PSL_2\C$ has a non-trivial stabilizer,  a similar difficulty occurs when the limit holonomy of a component of $S \minus m$ is elementary.  
As a result of such flexibility, we have rather exotic degenerations described in Case (ii) of \Cref{LimitInPleatedSurface} and \Cref{LImitInDevelopingMap}.

One may certainly hope that some of the results extend to a more general setting of Problem \ref{problem}. 
In particular,  Theorem \ref{BoundingUpperInjectivityRadius} may hold in general:  
\begin{conjecture}
In the setting of Problem \ref{problem} (without the neck-pinching assumption),   
let $E_t$ be the singular Euclidean structure on $X_t$ given by the Schwarzian parameters of $C_t$. 
Then the upper injectivity radius of $E_t$ is bounded from above uniformly in $t \geq 0$. 
\end{conjecture}
    \fontsize{13pt}{12pt}\selectfont
  Recall that $\rho_t\, (t \geq 0)$ is a topological path in the character variety $\chi$ which converges to $\rho_\infty$ as $t \to \infty$ without any regularity assumption.
    It is plausible that Cases (\ref{iParabolicDiffferential}) in \Cref{SchwarzianLimit}, \Cref{LimitInPleatedSurface} and \Cref{LImitInDevelopingMap} do not occur if $\rho_t$ has a one-side derivative at $t = \infty$ (in the ambient affine space of $\chi$).
\begin{conjecture}\Label{cOnlyParabolic}
Suppose that $X_t$ is pinched along a loop $m$. 
If the path $\rho_t$ is tangential at $t = \infty$, Then $\eta_\infty(m) \in \PSL_2\C$ is a parabolic element (not equal to the identity $I$).  
\end{conjecture}

\subsection{Outline of this paper}
In \S\ref{sPreliminaries}, we recall $\CP^1$-structures, the Schwarzian parameters, Thurston parameters, and the Epstein surfaces for $\CP^1$-structures. 
In \S \ref{sPathLiftingProperty}, we prove a lifting property of paths in the character variety to paths in the representation variety. 
In \S \ref{sEpsteinSurface}, we give some estimates of the Epstein surfaces,  based on Dumas' work \cite{Dumas18HolonomyLimit}. 
In \S \ref{sHolonomySurfaceWithPuncture}, we prove a holonomy theorem for the space of developing pairs of  $\CP^1$-structures on surfaces with punctures, where punctures are at most poles of order two. 
In \S \ref{sInjetivityRadius}, we show that there is an upper bound for the upper injectivity radius of $E_t$ for all $t \geq 0$. 

In \S \ref{sConvergenceOfThickPart}, we show that $C_t$ converges on every thick part as $t \to \infi$, so that $C_t$ converges to a $\CP^1$-structure on a surface with two punctures homeomorphic to $S \minus m$. 
In  \S \ref{sDegeneration}, we state our main theorems and prove some properties of developing maps of a surface with punctures. 
The limit holonomy around $m$ can only be parabolic or the identity. 
This will be shown,  in \S \ref{sHyperbolicNeck} and \S \ref{sEllipticNeck}.
In \S \ref{sParabolicNeck}, we determine the asymptotic behavior of $C_t$ when $\rho_\infi(m)$ is parabolic.
In \S \ref{mTrivialHolonomy}, we give the asymptotic behavior of $C_t$ when $\rho_\infi(m) = I$.

In \S \ref{sExoticDegeneration}, we give new examples realizing (ii) in Theorem \ref{SchwarzianLimit}, Theorem \ref{LimitInPleatedSurface}, Theorem \ref{LImitInDevelopingMap}.
\subsection{Acknowledgements}
I thank Dick Canary,  David Dumas, Ko Honda, Misha Kapovich, and Joan Porti. 
The work done here is partially supported by German Research Foundation (BA 5805/1-1), 
GEAR Network (U.S. National Science Foundation grants DMS 1107452, 1107263, 1107367), and a JSPS Grant-in-Aid for Research Activity start-up (18H05833).

I deeply appreciate the anonymous referees for carefully and patiently reading through my original manuscript and giving various comments which significantly improved this paper.
\section{Preliminaries}\Label{sPreliminaries}

\subsection{Hyperbolic geometry}

Let $\tau$ be a hyperbolic structure on $S$. 
Let $L$ be a geodesic measured lamination on $\tau$. 
Given a geodesic loop $m$ on $\tau$, for a point $x$ in the intersection of $m$ and $L$, let $\angle_x(L, m) \in [0, \pi)$ denote the intersection angle of of the leaf $L$ and $m$ intersecting at $x$. 
Then,  the {\it angle} $\angle_\tau(m, L) \in [0, 1)$ between $L$ and $m$ be the maximum of $\angle_x(L, m)$ over all intersection points $x \in L \cap m$ if $L \cap m \neq \emptyset$, and  $\angle_\tau(m, L) = 0$ if $L \cap m = \emptyset$. 

Let $\phi\col \H^2 \to \tau$ denote the universal covering map.
Then  the $\phi$-inverse image $\ti{L}$ of $L$  is a $\pi_1(S)$-invariant measured lamination on $\H^2$. 
The pair $(\tau, L)$ induces a {\it bending map} $\beta\col \H^2 \to \H^3$ which is equivariant via an associated homomorphism $\rho\col \pi_1(S) \to \PSL_2\C$.
This mapping $\beta$ is defined by bending the universal cover $\H^2$ of $\tau$ along $L$, where the bending angle is given by the transversal measure of $\ti{L}$ (\cite{Epstein-Marden-87}).
Then the pair $(\tau, L)$ determines $\beta\col \H^2 \to \H^3$ uniquely up to $\PSL_2\C$; thus the pair $(\beta, \rho)$ is identified with $(\alpha \circ \beta, \alpha \rho \alpha^{-1})$ for $\alpha \in \PSL_2\C$.

It follows from Corollary 4.3 in \cite{Baba-15gt} (see also Theorem 5.1 in \cite{Baba-17}) that, if a geodesic loop on $\tau$ intersects the lamination in a small angle, then the holonomy along the loop must be hyperbolic. 
\begin{theorem}\Label{SmallIntersectionAngleImpliesHyperbolic}
There is a universal constant $\del > 0$ such that if $\angle_\tau(L, m) < \del$, then $\rho(m)$ is hyperbolic. 
\end{theorem}
\begin{proof}
Let $\ti{m}$ be a lift of $m$ to the bi-infinite geodesic in the universal cover $\ti\tau = \H^2$. 
Then, the restriction of $\beta$ to $\ti{m}$ is a $(1 + \ep)$-bilipschitz embedding (Corollary 4.3 in \cite{Baba-15gt}). 
Since $\beta$ is $\rho$-equivariant,  $\rho(m)$ is a hyperbolic element whose axis connects the ideal point of the bilipschitz embedding $\beta(\ti{m})$. 
\end{proof}
   
\subsection{$\CP^1$-structures}\Label{sCPS}
(General references of $\CP^1$-structures are found in \cite{Dumas-08, Kapovich-01}.)

A {\it $\CP^1$-structure} $C$, or a {\it complex projective structure}, on $S$ is a $(\CP^1, \PSL_2\C)$-structure, i.e. an atlas of charts embedding into $\CP^1$ with transition maps given by $\PSL_2\C$. 

Let $\ti{S}$ be the universal cover of $S$. 
Then, equivalently,  a $\CP^1$-structure is a pair $(f, \rho)$ of a local homeomorphism $f\col \til{S} \to \CP^1$ and a homomorphism $\pi_1(S) \to \PSL_2\C$ such that $f$ is $\rho$-equivariant. 
The map $f$ is called the {\it developing map} and $\rho$ is called the {\it holonomy representation} of $C$. 

The pair is defined up to $\PSL_2\C$, i.e. $(f, \rho) \sim (\alpha f,  \alpha \rho \alpha^{-1})$ for all $\alpha \in \PSL_2\C$.
Thus the holonomy is in the character variety $\rchi = {\rm Hom} (\pi_1(S), \PSL_2\C)\sslash \PSL_2\C$.

\subsubsection{Schwarzian parametrization}
Each $\CP^1$-structure corresponds to a holomorphic quadratic differential $q$ on a marked Riemann surface $X$.
Thus the deformation space $\PP$ of $\CP^1$-structures is an (affine) vector bundle over the Teichmüller space $\TT$, such that a fiber over a Riemann surface $X$ is the vector space $Q(X)$ of holomorphic quadratic differentials on $X$ (in fact, it is the cotangent bundle).
In this paper, considering the projection map $\Pi \col \PP \to \TT$ given by the uniformization, we regard the space of marked hyperbolic structures on $S$ as our real analytic zero section.

Although $\Hol \col \PP \to \rchi$ is a highly non-proper map (\cite{Hejhal-75}), for each $X \in \TT$, the restriction of $\Hol$ to the space $Q(X)$  is a proper embedding onto a complex analytic subvariety of $\rchi$ (see \cite[Theorem 11.4.1]{Gallo-Kapovich-Marden} and its proof). 
Moreover
\begin{theorem}[\cite{Kapovich-95,Tanigawa99}]
For every compact subset $K$ of $\TT$,  the restriction of $\Hol$ to $\Pi^{-1}(K)$ is a proper map.
\end{theorem}

\begin{corollary}\Label{DivergenceInT}
Suppose that $C_t \in \PP$ leaves every compact subset in $\PP$ and its holonomy $\rho_t$ converges in $\rchi$. 
Then the complex structure $X_t$ of $C_t$ also leaves every compact subset in $\TT$ as $t \to \infi$.
\end{corollary}

\subsubsection{Thurston's parametrization of $\CP^1$-structures}\Label{sThurston}
(\cite{Kullkani-Pinkall-94, Kamishima-Tan-92}, see also \cite{Baba20ThurstonParameter}.)
Thurston gave a homeomorphism 
$$\PP \cong  \TT \times \ML,$$ where $\TT$ is  the space of marked hyperbolic structures on $S$ and $\ML$ is the space of measured laminations on $S$. 

 A pair $(\tau, L) \in \TT \times \ML$ yields a pleated surface $\H^2 \to \H^3$ equivariant under the holonomy $\pi_1(S) \to \PSL_2\C$ of its corresponding $\CP^1$-structure on $S$. 
Given a $\CP^1$-structure $C$ on $S$, 
its associated {\it collapsing map} $\kap\col C \to \tau$ is a marking preserving continuous map which relates the developing map and the bending map of $C$. 
First, there is a measured lamination $\LLL$ on $C$ consisting of circular leaves, such that topologically $\LLL$ is obtained by replacing each periodic leaf $\ell$ of $L$ by cylinder foliated circumferences so that the weight of $\ell$ is equal to the total transversal measure of the foliated cylinder. 
The collapsing map $\kap$, conversely, collapses such foliated cylinders of $\LLL$ to their corresponding periodic leaves of $L$, and $\kap$ takes the strata of $\LLL$ to the strata of $L$.

Moreover, $\kap$ relates the developing map $f\col \ti{S} \to \CP^1$ and the pleated surface $\beta\col \H^2 \to \H^3$ in an equivariant manner:
For each $z \in \ti{S}$, let $B_z$ be the {\it maximal ball} in $\ti{C}$ whose core contains $z$. 
Let $\Psi_z\col B_z \to {\rm Conv} \bdr_\infi B_z \sub \H^3$ denote the orthogonal projection, where  ${\rm Conv} \bdr_\infi B_z $ is the hyperbolic plane ({\it support plane}) bounded by the boundary circle. 
Then, in fact,  the commutativity 
$$\beta \circ\ti\kap (z)= \Psi_z f(z),$$
holds equivariantly, where $\ti\kap\col \ti{C} \cong \H^2 \to \ti\tau$ be the lift of $\kap$ to a map between universal covers.  
Note that there is a canonical normal direction of the support plane ${\rm Conv} \bdr B_z$ at $\Psi_z f(z)$ toward $f(z)$.

\subsection{Epstein maps}
Let $C = (X, q)$ be a $\CP^1$-structure on $S$ in the Schwarzian coordinates, where $X$ is the complex structure of $X$, and $q$ is a holomorphic quadratic differential on $X$. 
Then, the integration of $\sqrt{q}$ along paths yields a singular Euclidean metric $E$ on $X$ in the same conformal class (see for example \cite{FarbMarglit12}). 
In the complex plane, the lines parallel to the real axis give a foliation of $\C$, and it has a transversal measure induced by the vertical length ({\it horizontal measured foliation}). 
Similarly, the lines parallel to the imaginary axis give a foliation of $\C$, and it has a transversal measure induced by the horizontal length ({\it vertical measured foliation}). 
Then, by pulling back the vertical and the horizontal foliations of $\C$, we obtain a vertical singular measured foliation $V$ and a horizontal singular measured foliation $H$ on $E$, where the singular points are the zeros of the differential $q$. 
Moreover $H$ and $V$ are orthogonal, and the vertical and the horizontal foliation of $\C$ are orthogonal.

Given a point $x \in \H^3$, we can normalize the unit disk model of $\H^3$ so that $x$ is the center of the disk; then the ideal boundary of $\H^3$ has the spherical metric uniquely determined by $x \in \H^3$. 
\begin{theorem}[Epstein \cite{Epstein84}]
Given a $\CP^1$-structure $C = (f, \rho)$ on $S$, there is a unique continuous $\rho$-equivariant map $\Ep \col \til{X}  \to \H^3$, such  that, for every point $z \in \til{X}$, the Euclidean metric of $\til{E}$ at $z$ agrees with the spherical metric at $f(z) \in \CP^1$ when $\CP^1$ is identified with $\s^2$ so that $\Ep(z) \in \H^3$ is at the center of the disk model of $\H^3$. 

Moreover $\Ep\col \til{X} \to \H^3$ is smooth away from the singular points of $\ti{E}$ (see Equation (3.1) in \cite{Dumas18HolonomyLimit}).
\end{theorem}

Let $U \H^3$ denote the unit tangent bundle of $\H^3$. 
Then $\Ep$ lifts to a (Lagrangian) immersion $\Ep_\ast\col T \til{E} \to U(\H^3)$  (\cite[Lemma 3.2]{Dumas18HolonomyLimit}) which is a unit normal vector of  the surface $\Ep \col \til{X} \to \H^3$ in the complement of the singular points of $E$.
 For $z \in \til{X}$, let $d(z)$ denote the Euclidean distance from $z$ to the set $Z$ of the zeros of the differential $q$.  
\begin{lemma}[Lemma 2.6, Lemma 3.4 in \cite{Dumas18HolonomyLimit}]\Label{DumasLemma}
Let $h'(z)$ and $v'(z)$ be the horizontal and vertical unit tangent vectors at $z \in \til{X} \minus \ti{Z}$. 
If $\frac{6}{d(z)^2} < \frac{3}{4}$, then 
\begin{enumerate}
\item $\| \Ep_\ast h'(z) \|  < \frac{6}{d(z)^2}$,  \Label{iDumasHorizontal} 
\item $\sqrt{2} < \| \Ep_\ast v' \|  < \sqrt{2}  +  \frac{6}{ d(z)^2}$,  \Label{iDumasVertical}
\item $h' (z), v' (z)$ are the principal directions of $\Ep$ at $z$, and \Label{iDumasPrincipleDirections}
\item $|k_v | < \frac{6}{d(z)^2} $,  where $k_v$ is the curvature of $\Ep$ in the $v$-direction. \Label{iVStraight}
\end{enumerate}
\end{lemma}

\section{A lifting property of paths in the character variety}\Label{sPathLiftingProperty}
\begin{definition}
A representation $\rho\col \pi_1(S) \to \PSL_2\C$ is {\sf elementary} if $\Im \rho$ fixes a point in $\H^3 \cup \CP^1$ or preserves two points on $\CP^1$. 
Equivalently, $\rho$ is elementary if $\Im \rho$ is strongly irreducible and $\Im \rho$ is unbounded in $\PSL_2\C$.
Otherwise $\rho$ is called non-elementary.
\end{definition}

Let $\RRR$ denote the $\PSL_2\C$-{\it representation variety} of $S$,  the space of representations $\pi_1(S) \to \PSL_2\C$.
By fixing a generating set $\gam_1, \dots, \gam_n$,  the topology of $\RRR$ is the restriction of the product topology on $\PSL_2\C^n$, which is independent on the choice of  $\gam_1, \dots, \gam_n$.
The Lie group $\PSL_2\C$ acts on $\RRR$ by conjugation, and  its GIT-quotient $$ \Psi\col \RRR \to \rchi = \{ \pi_1(S) \to \PSL_2\C \} \sslash \PSL_2\C$$
is called the $\PSL_2\C$-character variety of $S$.

Each fiber of this GIT-quotient is an {\it extended orbit equivalence} class: Namely, for $\rho_1, \rho_2 \col \pi_1 (S) \to \PSL_2\C$,  $\rho_1 \sim \rho_2$ if and only if the closure of the $\PSL_2\C$-orbit of $\rho_1 $ intersects that of $\rho_2$ in $\RRR$.
In fact, equivalently
$\rho_1 \sim \rho_2$ if and only if $\tr^2 \rho_1(\gamma) =\tr^2 \rho_2(\gamma)$ for all $\gamma \in \pi_1(S)$ \cite{HeusenerPorti04}. 
In particular, for  a non-elementary representation $\pi_1(S) \to \PSL_2\C$, its $\PSL_2\C$-orbit is a closed subset of $\PSL_2\C$ and  form a single equivalence class (\cite{Newstead(06)}).   
For $\rho \in \RRR$, let $[\rho]$ denote it equivalent class $\Psi(\rho)$ in $\rchi$.

\begin{proposition}\Label{LiftingHolonomyPath}
Suppose that $C_t$  $(t \geq 0)$ is a one-parameter family of $\CP^1$-structures on $S$, such that its holonomy $\eta_t \in \rchi$ converges to $\eta_\infi \in \rchi$.
Then  $\eta_t$ lifts a path $\rho_t \in \RRR$ which converges to $\rho_\infi \in \RRR$ as $t \to \infi$, so that $[\rho_\infi] = \eta_\infi$.
\end{proposition}

\begin{remark}
The limit  $\eta_\infi$ can be an elementary representation (\S \ref{sExoticDegeneration}), and thus this proposition is nontrivial.
In addition, there is $\eta \in \RRR$ with $[\eta] = \rho_\infty$ such that there is no lift $\eta_t$ of $\rho_t$ ending at $\eta$. 
\end{remark}

\proof[Proof of  Proposition \ref{LiftingHolonomyPath}]
Fix a generating set $\gamma_1, \dots, \gamma_n$ of $\pi_1(S)$.  
We divide the proof into three cases:
\begin{enumerate}
\item $\eta_\infty$ is non-elementary. \Label{iNonelementaryRho}
\item  $\eta_\infty$ is elementary and there is $\gamma \in \PSL_2\C$ such that $\eta_\infi(\gamma)$ is hyperbolic, i.e. $\tr^2(\gamma) \in \C \minus [0, 4]$.
\Label{iHyperblicRhoLimit} 
\item  $\eta_\infty$ is elementary and there is no hyperbolic element in its image, i.e. $\tr^2 \eta_\infi(\gamma) \in [0, 4]$ for all $\gam \in \pi_1(S)$. \Label{LiftingEllipticElementaryRep}
\end{enumerate}
{\it Case \ref{iNonelementaryRho}. }

\begin{lemma}
Suppose that $\eta_\infty$ is non-elementary.
For every lift $\rho_\infty \in \RRR$ of  $\eta_\infi \in \rchi$, there is a lift $\rho_t \in \RRR$ of the path $\eta_t \in  \rchi$ such that $\rho_t \to \rho_\infty$ as $t \to \infty$.  
\end{lemma}
\begin{proof}
Over non-elementary representations, $\Psi$ is a fiber bundle with fibers $\PSL_2\C$.
This implies the lemma. 
\end{proof}

{\it Case \ref{iHyperblicRhoLimit}.} Suppose that $\eta_\infty$  is elementary and there is $\gam \in \pi_1(S)$ such that $\eta_\infi(\gam)$ is hyperbolic. 
Then, if $\rho \in  \Psi^{-1}(\eta_\infi)$, letting $\ell$ be the axis of the hyperbolic element $\rho(\gam)$,  we have either:
\begin{enumerate}[(i)]
\item  $\Im \rho$ preserves $\ell$ and  contains an elliptic element which reverses the orientation of $\ell$, or  \Label{iElementaryIrreducible}
\item  $\Im \rho$ pointwise fixes the endpoints of $\ell$ on $\CP^1$. \Label{iElementayHyperbolicReducible}
\end{enumerate}

{\it Case (\ref{iElementaryIrreducible}).}
Suppose that $\rho \in \Psi^{-1}(\eta_\infi)$ contains an elliptic element which exchanges the endpoints of $\ell$.
\begin{claim}\Label{AlmostHyperbolicGeneratingSet}
There are generators  $\gamma_1, \gamma_2, \dots, \gamma_n$  of $\pi_1(S)$, such that, for each $i = 1, \dots, n,$ 
\begin{enumerate} 
\item $\rho(\gamma_i)$ is a hyperbolic element for  $i = 1, \dots, n-1$, and
\item $\rho(\gamma_n)$ is an elliptic element of order two about a geodesic orthogonal to $\ell$.
\end{enumerate}
\end{claim}

\begin{proof}
By the hypothesis, one can pick generators $\gamma_1, \gamma_2, \dots, \gamma_n$ of $\pi_1(S)$, such that  $\rho(\gamma_1)$ is a (nontrivial) hyperbolic element. 
Then we can, in addition, assume that $\rho (\gamma_2), \dots,  \rho (\gamma_n)$ are not $I$, by composing $\gamma_i \,(i \geq 2)$ with $\gamma_1$ if necessary. 
If $\rho( \gamma_i)$ is an elliptic element preserving the orientation of $\ell$, then $\rho (\gamma_1 \gamma_i)$ is hyperbolic--- thus without loss of generality, we can assume that if $\rho(\gam_i)$ is an elliptic element, it must reverse the orientation of $\ell$. 
Suppose that  $\rho(\gamma_i)$ and $\rho(\gamma_j)$ are both elliptic elements reversing the orientation of $\ell$; then $\rho(\gamma_i \gamma_j)$ preserves the orientation of $\ell$. 
Thus, by replacing $\gamma_j$ with $\gamma_i \gamma_j$, we can reduce the number of the generators which map to elliptic elements reversing the orientation of $\ell$. 
We can repeat such replacements of generators, we obtain a desired generating set.
\end{proof}
Let  $\gamma_1, \gamma_2, \dots, \gamma_n$ be the generating set of $\pi_1(S)$ obtained by \Cref{AlmostHyperbolicGeneratingSet}.
We show that there is indeed a lift $\rho_t$ in $\RRR$ of $\eta_t$ converging to $\rho$ as $t \to \infi$.

One can easily find  a lift $\rho_t \, (t \geq 0)$ so that $\rho_t(\gamma_1)$ converges to $\rho(\gamma_1)$. 
Then $\Axis (\rho_t(\gamma_1))$ must converge to $\ell$. 
For all $1 \leq i \leq n-1$,   $\Axis(\rho_t (\gamma_i))$  and $\Axis (\rho_t(\gamma_n))$ are asymptotically orthogonal, as $\eta_\infi$ is an equivalence class of some elementary representation. 
In particular, we can in addition assume that $\rho_t(\gamma_n)$  converges to $\rho(\gamma_n)$, so that $\Axis(\rho_t(\gamma_n))$ converges to a geodesic $m$ orthogonal to $\ell$. 
Then, for $1 < i < n$,   $\Axis(\rho_t(\gamma_i))$ converges $\ell$,  since it is asymptotically orthogonal to $m$ and $\eta_\infty$ is elementary.
Thus $\rho_t$ converges to $\rho$ as $t \to \infi$. 

{\it Case (\ref{iElementayHyperbolicReducible}).}
Next suppose that $\rho \in \Psi^{-1}(\eta_\infi)$ preserves the endpoints of $\ell$. 
Then, similarly to Claim \ref{AlmostHyperbolicGeneratingSet}, we can find a generating set $\gamma_1, \dots, \gamma_n$ such that $\eta_\infi(\gam_1), \dots, \eta_\infi(\gam_n)$ are all hyperbolic elements (i.e. $\tr^2 \eta_\infi(\gamma_i) \in \C \minus [0, 4]$).

Pick any lift $\rho_t$ of $\eta_t$ for $t \geq 0$ (which may not converge as $t \to \infi$). 

Fix a $\PSL_2\C$-invariant metric on the projectivized unit tangent bundle ${\rm PT^1}\H^3$ of $\H^3$.  
Then, given two geodesics $\ell_1, \ell_2$ in $\H^3$, we can measure their distance by embedding $\ell_1$ and $\ell_2$ into the bundle.
Thus, similarly, for all $1 \leq i, j \leq n$,  the distance between  $\Axis (\rho_t(\gam_i))$ and $\Axis(\rho_t(\gam_j))$  goes to zero as $t \to \infi$, since otherwise $\eta_\infty$ is an equivalent class of some non-elementary representations due to the limit of $\rho_t (\gam_i)$ and $\rho_t (\gam_j)$.  
Thus we can continuously conjugate $\rho_t$ by elements of $\PSL_2\C$ so that all axes of $\rho_t(\gam_1), \dots, \rho_t(\gam_n)$ converge to geodesics sharing an endpoint. 
Therefore $\rho_t$ converges as $t \to \infi$ by this normalization.

{\it Case \ref{LiftingEllipticElementaryRep}.}
Suppose that $\Im \eta_\infi$ contains no hyperbolic elements.
Given an elliptic element and a parabolic element in $\PSL_2\C$ sharing a fixed point on $\CP^1$ then their product is an elliptic element. 
Therefore we can pick generators $\gam_1, \dots,  \gamma_n$ of $\pi_1(S)$, such that $\eta_\infi(\gamma_i)$ are either all elliptic or all parabolic:
In fact, given a generating set $\gam_1, \dots,  \gamma_n$, if the $\eta_\infty$-image of at least one $\gam_i$ is elliptic, then by replacing $\gam_j$ with parabolic $\eta_\infty(\gam_j)$ with $\gam_i \gam_j$, we obtain a generating set with elements whose $\eta_\infty$-images are all elliptic.
 Pick any lift $\rho_t \in \RRR$ of the path $\eta_t \in \rchi$ for $t \geq 0$, which may not converge as $t  \to \infty$.
  
  \begin{definition}\Label{axis}
  For $\gamma \in \PSL_2\C$, the axis of $\gamma$ is the convex hull of the fixed point set in $\H^3 \cup \CP^1$ of $\gam$, and  we denote it by $\Axis(\gamma) \sub \H^3 \cup \CP^1$.
 \end{definition}
In particular, if $\gamma$ is hyperbolic or elliptic,  $\Axis(\gam)$ is a geodesic in $\H^3$ plus its endpoints in $\CP^1$, and
 if $\gamma$ is parabolic, $\Axis(\gam)$ is a single point on $\CP^1$.
Clearly an ideal point of $\Axis(\gam)$ is a fixed point of $\gam$ on $\CP^1$.

Suppose that $\gamma, \omega \in \PSL_2\C$ be hyperbolic or elliptic elements with axes $\ell_\gamma, \ell_\omega$. 
As above, we measure the distance between $\ell_\gamma, \ell_\omega$ by embedding them into the projective unit tangent bundle of $\H^3$.
\begin{lemma}\Label{EllipticAxiesGetCloser}
\begin{enumerate}
\item Suppose that $\eta_\infi(\gamma_i)$ and $\eta_\infi(\gamma_j)$ are both elliptic for distinct $1 \leq i, j \leq n$.  
Then the distance between $\Axis(\rho_t(\gamma_i))$ and $\Axis(\rho_t(\gam_j))$ in ${\rm PT^1}(\H^3)$ limits to zero as $t \to \infi$. \Label{iTwoElliptics}
\item 
Suppose that 
 $\eta_\infi(\gamma_i)$,  $\eta_\infi(\gamma_j)$,  $\eta_\infi(\gamma_k)$ are all elliptic for distinct $1 \leq i, j, k \leq n$. 
 Then there is a lift $\rho_t  \in \RRR$ of $\eta_t$ for $t \geq 0$, such that $\Axis(\rho_t(\gamma_i))$,  $\Axis(\rho_t(\gamma_j))$, and  $\Axis(\rho_t(\gamma_k))$ converge to geodesics sharing a common endpoint on $\CP^1$.  \Label{iThreeElliptics}
 \end{enumerate}
\end{lemma} 

\begin{proof}
(\ref{iTwoElliptics})
If there is a diverging sequence $0 < t_1 < t_2 < \dots$ such that the distance between $\Axis(\rho_t(\gamma_i))$ and $\Axis(\rho_t(\gam_j))$ in ${\rm PT^1}(\H^3)$ is bounded from below by a positive number, then  $\eta_\infi$ is non-elementary.
This is a contradiction.

(\ref{iThreeElliptics})
By (\ref{iTwoElliptics}), if the assertion of  (\ref{iThreeElliptics}) fails, there is a lift $\rho_t$ such that  $\Axis(\rho_t(\gamma_i))$,  $\Axis(\rho_t(\gamma_j))$, and  $\Axis(\rho_t(\gamma_k))$ converge to the distinct edges of an ideal triangle in $\H^3$. 
Then, $\eta_\infi$ is non-elementary against the hypothesis. 
\end{proof}
\begin{corollary}
Suppose that there is a generating set $\{\gam_1, \dots, \gam_n\}$ of $\pi_1(S)$, such that $\eta_\infi(\gam_1), \dots,  \eta_\infi(\gam_n)$ are all elliptic. 
Then, there is a lift $\rho_t \in \RRR$ of $\eta_t$   such that $\rho_t(\gam_1), \dots, \rho_t(\gam_n)$ converge to elliptic elements whose axes share an endpoint on $\CP^1$.   
\end{corollary}
Last we suppose that  $\eta_\infi(\gam_1), \dots,  \eta_\infi(\gam_n)$ are all parabolic, and we show that there is a lift of $\rho_t$ of $\eta_t$ to $\RRR$ such that $\rho_t$ converges to the trivial representation.

Pick a base point $O \in \H^3$. 
For each $t \geq 0$, let $\del_{t, i} = d_{\H^3} (O, \rho_t(\gam_i) O)$. 
Let $i_t \in \{1, \dots, n\}$ be such that \[\del_{t, i_t} =\max_{1 \leq  i  \leq n} \del_{t, i}.\]
\begin{lemma}\Label{ByCongugationParabolic}
Let $t_1 < t_2 <  \dots$ be a sequence diverging to $\infty$, such that, at $t_{t_k}$, the indices $i_{t_k} \in \{1, \dots, n\}\ (k = 1, 2, \dots)$ defined above  are a fixed constant $h$.
Suppose that there is a sequence $\omega_{t_k} \in \PSL_2\C$ such that the conjugation $\omega_{t_k} \rho_{t_k} (\gam_h) \omega_{t_k}^{-1} \eqqcolon \omega_{t_k}\!\cdot\rho_{t_k} (\gam_h)$ converges in $\PSL_2\C$, as $k \to \infi$, to a parabolic element   
$\begin{pmatrix}
1 & u\\
0&1   
\end{pmatrix}$  in $\PSL_2\C$ with $u \neq 0$.
Then,  for every $j = 1, \dots, n$, the conjugation $\omega_{t_k} \!\!\cdot\rho_{t_k} (\gam_j)$ accumulates to a bounded subset of 
 $\begin{pmatrix}
1 & \C\\
0&1   
\end{pmatrix} $ which has a diameter less than $|u|$ in $\C$.
\end{lemma} 
\begin{proof}
First we show that, unless $\omega_{t_k}\!\!\cdot\!\rho_{t_k} (\gam_j) \to I$, the limit of  the fixed point set of $\omega_{t_k}\!\!\cdot\rho_{t_k} (\gam_j) \sub \CP^1$ must converge to $\{\infty\}$.
Suppose, to the contrary, that this assertion fails. 
Then, up to a subsequence, the limit set of the fixed point set of $\omega_{t_k}\!\!\cdot\rho_{t_k} (\gam_j) \sub \CP^1$ converges to a point on $\CP^1$ not equal to $\infty$. 
For sufficiently large positive integers $p$, $\omega_{t_k}\!\!\cdot\!\rho_{t_k}(\gam_h \gam_i^p)$ are  hyperbolic elements and their translation lengths diverge to $\infty$ as $p \to \infty$(\cite[Lemma 2.1.1 (iii)]{Gallo-Kapovich-Marden}).
This contradicts that $\Im \eta_\infty$ consists of only parabolic elements. 
 
  For each $k = 1, 2, \dots$,
set
\begin{equation}
 \omega_{t_k}\!\!\cdot\rho_{t_k} (\gam_j) =
\begin{pmatrix}
a_{j, k} & b_{j, k}\\
c_{j, k} &d_{j, k}
\end{pmatrix} \Label{eEntries}
\end{equation}

Thus $c_{j, k} \to 0$ and $a_{j, k}, d_{j, k} \to 1$ as $k \to \infi$. 
Then the definition of $h$ implies that $b_{t_k, h} - \max_{1 \leq i \leq n} b_{t_k, i} \to 0$.
 Hence we have the upper bound on the image in $\C$.
\end{proof}
 By a straight computation, we obtain the following.
 \begin{corollary}\Label{RenoamalizeToBeIdentity}
For every $j = 1, \dots, n$, let $s_{j, k} > 0$ be a sequence in $k$, such that $s_{j, k} \to 0$ and $\frac{\sqrt{ \lvert c_{j, k}\rvert}}{s_{j, k}} \to 0$.
Then, using  the notation from (\ref{eEntries}), we have
\begin{equation}
\begin{pmatrix}
s_{j, k} & 0 \\
0 & s_{j, k}^{-1}\\
\end{pmatrix}
\begin{pmatrix}
a_{j, k} & b_{j, k}\\
c_{j, k} &d_{j, k}
\end{pmatrix} 
\begin{pmatrix}
s_{j, k} & 0 \\
0 & s_{j, k}^{-1}\\
\end{pmatrix} \to 
\begin{pmatrix}
1 & 0 \\
0 & 1\\
\end{pmatrix}\Label{eConvergingtoI}
\end{equation}as $k \to \infty$.\end{corollary}
Moreover, \Cref{RenoamalizeToBeIdentity} implies that
 the sequence $\max_{j = 1, \dots, n} {s_{j, k}}$ in $k$ yields the convergence (\ref{eConvergingtoI}) for all $j = 1, \dots, n$. 
Therefore we have the following.
\begin{proposition}
There is a continuous path $\omega_t  \in \PSL_2\C$ such that $\omega_t \!\cdot \rho_t(\gam_i)$ accumulates to a bounded subset of parabolic elements in  $\begin{pmatrix}
1 & \C\\
0&1   
\end{pmatrix} $  for each $i$. 
Therefore there is a continuous path  $\omega_t  \in \PSL_2\C$ such that $\omega_t \cdot \rho_t$  converges to the trivial representation in $\RRR$. 
\end{proposition}
We have completed the proof for all cases.
\Qed{LiftingHolonomyPath}

\subsection{Approximation of moduli}
Let $E$ be a singular Euclidean surface induced by a holomorphic quadratic differential on a Riemann surface $X$. 
A {\it regular annulus} $A_E$ is a cylinder embedded in $E$ such that there is a closed geodesic loop $\ell$ on $E$ and the annulus $A_E$ is foliated by loops equidistant from $\ell$. 
Minsky gave a useful approximation of the modulus of cylinders. 
  \begin{theorem}[\cite{Minsky92}, Theorem 4.6; see also \cite{Series12}, Theorem 6.2]\Label{Minsky}
  Let $E$ be a singular Euclidean surface induced by a holomorphic quadratic differential on a Riemann surface $X$. 
There are constant $0 < c < 1$ depending on the topology of the surface, such that, for every essential annulus $A$ embedded in $X$, there is a regular annulus $A_E$ in $E$ homotopy equivalent to $A$ satisfying
$\Mod(E_A) > c \Mod(A)$. 
\end{theorem}  

\section{Holonomy estimates away from zeros}\Label{sEpsteinSurface}
In this section, based on Dumas' work on Epstein surfaces \cite{Dumas18HolonomyLimit}, we give some further analysis of the Epstein surfaces in the horizontal direction.
We use Dumas' notations as below.
Let $g_1 = e^{\alpha_1} |dz|, g_2 = e^{\alpha_2} |dz|$ be two conformal metrics on a Riemann surface; then the Schwarzian derivative of $g_2$ relative to $g_1$ is the quadratic differential 
\begin{equation*}
B(g_1, g_2) = [(\alpha_1)_{zz} - {\alpha_2}_z^2 - (\alpha_1)_{z z} + (\alpha_1)_z^2] dz^2.
\end{equation*}
Let $C = (X, q)$ be a $\CP^1$-structure on $S$. 
Then, we set the following notations associated with $C$:
\begin{itemize}
\item Let $\tau$ be the hyperbolic metric on $S$ uniformizing $X$;
\item let $|\sqrt{q}|$ denote the singular Euclidean metric on $X$ obtain by integrating $\sqrt{q}$ along paths;
\item let $g_{\CP^1}$ be the spherical metric on $\CP^1$ given by some conformal identification $\CP^1 \cong \s^2$;
\item let $f\col \ti{X} \to \CP^1$ be the developing map of $C$, and $f^\ast(g_{\CP^1})$ be the pull back of the conformal metric $g_{\CP^1}$ by $f$ to the universal cover $\ti{X}$.
\end{itemize}
Then set
\begin{eqnarray*}
\omega  &=& 2 B(\tau, f^\ast(g_{\CP^1})),\\
\hat{\omega} &=& 2 B(|\sqrt{q}|, f^\ast(g_{\CP^1})),\\
\nu &=& 2B(\sigma, \sqrt{q}),
\end{eqnarray*}
which are holomorphic quadratic differentials on $\til{X}$.

  \subsection{Curvature of Epstein surfaces in the horizontal direction}

Let $k_h$ and $k_v$ be the principle curvatures of $\Ep\col \til{X} \to \H^3$ in the horizontal and the vertical directions, respectively.
First  
by Equation 3.7  in  \cite{Dumas18HolonomyLimit}
\begin{equation*}
k_v =     \frac{|\hat{\omega}| - |\omega|}{|\hat{\omega}| + |\omega|}.
\end{equation*}  
As the Gaussian curvature  $\kap_h \kap_v= 1$ ( \cite[p448]{Dumas18HolonomyLimit}), we have
\begin{equation*}
k_h =     \frac{|\hat{\omega}| + |\omega|}{|\hat{\omega}| - |\omega|}.
\end{equation*}  
In addition, recalling that $h'$ denotes a unit tangent vector in the horizontal direction at a non-singular point, we have
\begin{equation*}
\| \Ep_\ast h’ \|^2 =    \frac{{(|\hat{\omega}| -  |\omega|)^2}}{2 |\omega \hat{\omega}|}
\end{equation*}
(Equation 3.6  in \cite[p448]{Dumas18HolonomyLimit}).
Therefore 
\begin{eqnarray*}
(k_h \| \Ep_\ast (h’) \|)^2 &=& 
(\frac{|\hat{\omega}| + |\omega|}{|\hat{\omega}| -  |\omega|})^2  \cdot    \frac{{(|\hat{\omega}| -  |\omega|)^2}}{2 |\omega \hat{\omega}|} \\
&=&  1+ \frac{(|\hat{\omega}|^2 + |\omega|^2)}{2 |\omega \hat{\omega}|}\\
& =&  1 +  \frac{1}{2}(\frac{|\hat{\omega}|}{|\omega|} +  \frac{|\omega|}{|\hat{\omega}|}).\\
\end{eqnarray*}
Since $\hat{\omega} = \omega - \nu$ (\cite[p447]{Dumas18HolonomyLimit}), we have
\begin{eqnarray*}
\left|\frac{\hat{\omega}}{\omega}\right| = \left|1 - \frac{\nu}{\omega}\right|\, \text{ (\cite[p449]{Dumas18HolonomyLimit})}.
\end{eqnarray*}
By \cite[Lemma 2.6]{Dumas18HolonomyLimit}, we have
\begin{eqnarray*}
\left| \frac{\nu(z)}{\omega(z)} \right| &\leq& \frac{6}{d(z)^2} \, .\\
\end{eqnarray*}
Thus, recalling that $d(z)$ is the distance from the singular points, we have
\begin{eqnarray*}
\frac{|\hat{\omega}(z)|}{|\omega(z)|}  &=&  1 + O(d(z)^{-2}), \text{~and}\\
(k_h(z)  \|\Ep_\ast (h’(z))\| )^2 &=& 2 + O(d(z)^{-2}).
\end{eqnarray*}
Therefore, we have the following.
\begin{lemma}\Label{HorizontalTotalCurvature}
For all nonzero $z \in \ti{X}$ of the differential $\ti{q}$,
\begin{eqnarray*}
k_h(z) \| \Ep_\ast (h’(z))\|  = \sqrt{2} + O(d(z){^{-2}}).
\end{eqnarray*}
\end{lemma}

 \subsection{Holonomy estimates of long flat cylinders} 
 
Let $E$ be a singular Euclidean surface. 
A {\it flat cylinder} in $E$ is a cylinder foliated by closed geodesics. 
A cylinder $A$ in $E$ is {\it expanding} if there is a geodesic loop $\ell$ or a puncture $p$ on $E$, such that $A$ is foliated by a one-parameter family of circles equidistant from $\ell$ or $p$, respectively,  whose length strictly increases as the distance to $\ell$ or $p$ increases. 
 
  Let $\Ep\col \til{X} \to \H^3$ be the Epstein surface of a projective structure $C = (X, q)$ on $S$. 
  Let $\alpha \col [0, 1] \to \til{C} \cong \til{X}$ be an arc such that $\alpha(0)$ and $\alpha(1)$ are in $\ti{X} \minus \ti{Z}$ and  $\alpha$ differentiable at both endpoints. 
  Then the curve $\Ep \cc \alpha\col [0,1] \to \H^3$ is differentiable at both endpoints. 
Let $\zeta(\alpha) \in \PSL_2\C$ be such that $\zeta(\alpha)$ takes the unit tangent vector $\alpha'(0)$  to $\alpha'(1)$  on $\Ep$ and the unit normal $\Ep_\ast \alpha(0)$ to the unit normal $\Ep_\ast \alpha(1)$.
We call $\zeta(\alpha) \in \PSL_2\C$ the {\it holonomy (of $\Ep$) along} $\alpha$.
 \begin{definition}\Label{dRotationAngle}
 For $\alpha  \in \PSL_2\C$, the rotation angle in $[0, \pi]$ is the (unsigned) rotation angle of the tangent plane of $\CP^1$ at a fixed point of $\alpha$. 
 \end{definition}
 In the case that $\alpha$ has two fixed points on $\CP^1$, then the ``signed'' rotation angle of $\alpha$ which takes a value in $[-\pi, \pi]/ (\pi \sim -\pi)$  at a fixed point is $-1$ times the ``signed'' rotation angle at the other fixed point, where the sign is determined by the orientation from $\CP^1$; thus the unsigned rotation angle is well-defined in \Cref{dRotationAngle}. 
   
 Let $(E, V)$ be the singular Euclidean surface given by $C = (X, q)$.
\begin{definition}
Let $\alpha\col [0, 1] \to \H^3$ be a $C^1$-smooth arc on the Epstein surface  $\ti{E} \to \H^3$. 
Let $v(t)$ and $h(t)$ denote the (unit) vector fields along $\alpha$ tangent to the vertical and horizontal foliations of $E$, respectively. 

Let $\ell$ be a geodesic in $\H^3$. 
Let $\HHH$ be the foliation of $\H^3$ by the totally geodesic hyperbolic planes $H$ orthogonal to $\ell$. 
Note that these hyperbolic planes are isometrically identified by parallel transport along $\ell$, and thus their ideal boundary circles are also identified diffeomorphically. 

Suppose that $v(t)$ is transversal to the foliation $\HHH$. 
Let $H_t$ be the leaf of $\HHH$ containing $\alpha(t)$.
The {\sf translation length} of $\alpha$ along $\ell$ is the distance between $H_0$ and $H_1$ (i.e. the length of the segment of $\ell$ between $H_0$ and $H_1$).

As $v(t)$ is transversal to $\HHH$, then, by the orthogonal projection $\H^3 \to H_t$, the horizontal tangent vector $h(t)$ projects to a non-zero vector at the tangent space $T_{\alpha(t)} H_t$. 
This non-zero tangent vector determines a geodesic ray in $\H^3$ by being its initial tangent direction.
 Let $\theta(t) \in \bdr_\infi H_t$ be the endpoint of the geodesic ray in $H_t$  given by the tangent vector.
As all ideal boundaries $\bdr_\infi H_t$ are identified,  $\theta(t) \in \s^1$ lifts to $\ti\theta(t) \in \R$.
 The {\sf rotation angle} of $\alpha$ about $\ell$ is the total increase of $\ti\theta(t)$, which takes a value in $\R$.
 \end{definition}   

\begin{proposition}\Label{CylinderHonomyEstimate}
Let $C_i = (f_i, \rho_i)$ be a sequence of  $\CP^1$-structures on $S$, and let $(E_i, V_i)$ be the pair of a singular Euclidean structure $E_i$ and a vertical foliation $V_i$ on $E_i$ induced by the Schwarzian parameters of $C_i$. 
Suppose that there are a loop $m$ on $S$, a geodesic representative $m_i$ of $m$ on $E_i$ for each $i$,  and a flat cylinder $A_i$ in $E_i$ contains $m_i$, such that
\begin{itemize}
\item $m_i$ is in the middle of $A_i$, so that $A_i \minus m_i$ is a union of two isometric flat cylinders, 
\item  $\Mod(A_i) \to \infi$ as $i \to \infi$, and
\item the height $a_i$ of $A_i$ diverges to $\infty$ as $i \to \infi$.
\end{itemize}
Let $\bar{m}_i$ be a segment on the universal cover $\ti{E}_i$ obtained by lifting the simple closed curve $m_i$.
Then, by parametrizing $\bar{m}_i$ by arc length $s \in [0, \length (m_i)]$, for every $\ep > 0$, if $i > 0$ is sufficiently large, then
\begin{enumerate}
\item the translation length of $\Ep_i \bar{m}_i(s)$  along $\ell_i$ is $(1 + \ep)$-bi-Lipschitz  to $\sqrt{2}\,({\rm Re}   \int_{m_i} \sqrt{q_t})$, \Label{iVerticalTranslation} 
\item  the total rotation angle of $\Ep \bar{m}_i$ about $\ell_i$ is $(1 + \ep, \ep)$-bi-Lipschitz to  to  $\sqrt{2}\, (\Im\int_{m_i} \sqrt{q_t})$,  and\Label{iHorizontalTotalRotation}
\end{enumerate}
\end{proposition}

\begin{proof}
Isotope $m_i$ in $A_i$, fixing a point on $m_i$, so that  $m_i$ is a union of a vertical segment $u_i$ and a horizontal segment $w_i$ (Figure \ref{fVerticalAndHorizontal}). 
Then $m_i$ remains close to the middle of $A_i$.  

\begin{figure}
\begin{overpic}[scale=.6
] {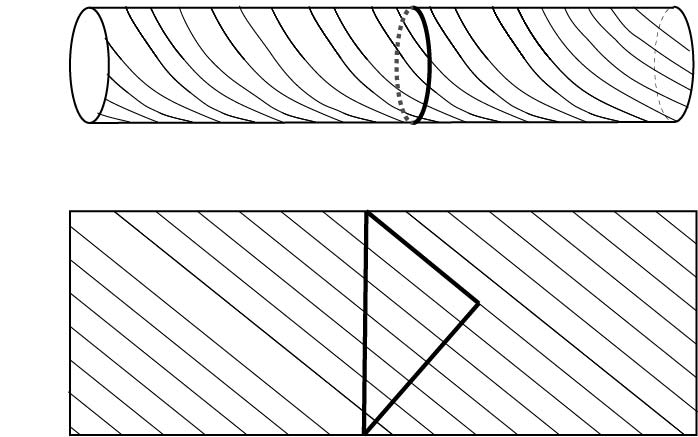} 
   \put( 64, 53){\contour{white}{\textcolor{Black}{$m_i$}}  }
  \put( 60, 26){\contour{white}{\textcolor{Black}{$u_i$}}}  
   \put( 63, 7){\contour{white}{\textcolor{Black}{$w_i$}}  }
      \put( 45, 15){\contour{white}{\textcolor{Black}{$m_i$}}  }
   \put( 35, 50){\contour{white}{\textcolor{Black}{$A_i$}}  }
      \end{overpic}
\caption{Isotope $m_i$ to a union of a vertical and horizontal segment. }\label{fVerticalAndHorizontal}
\end{figure}

We first analyze the vertical segment $\Ep_i | u_i$.
In the principal direction, the normal vector is preserved by parallel transports.
Thus, the parallel transport along the curve  $\Sigma_i |  u_i$ yields the holonomy $\zeta_i(u_i(s)) \in \PSL_2\C$. 
By the hypotheses, the distance from the loop $u_i \cup w_i$ and the set $Z_i$ of zeros of the differential $q_i$ diverges to $\infty$. 
Therefore, by Lemma \ref{DumasLemma} (\ref{iVStraight}), the curvature along $\Ep_i | u_i$ limits to zero, and it asymptotically has a constant speed $\sqrt{2}$ by \Cref{DumasLemma} (\ref{iDumasVertical}), so that its length is  $\sqrt{2}$ times the Euclidean length of $u_i$, which yields (\ref{iVerticalTranslation}).

   To analyze the total rotation angle in the vertical direction, we next consider the total curvature. 
  In a more general setting, the following holds. 
\begin{lemma}\Label{VerticalTotalCurvatureBound}
For every $\ep > 0$, if $R > 0$ is sufficiently large, then, if a vertical segment $u$ on a $\CP^1$-surface $C$ has Euclidean length less than $R/\ep$, then 
 total curvature of the curve $\Ep \vert u$ is less than $\ep$, where $\Ep\col \ti{C} \to \H^3$ is the Epstein surface of $C$. 
\end{lemma}
\begin{proof}
The curvature of the curve $\Ep \vert u$ at every point on $u$ is bounded from $\frac{6}{R^2}$ by Lemma \ref{DumasLemma} (\ref{iVStraight}).
Since, by the hypothesis, the length of $u$ is bounded from above by $\frac{R}{\ep}$, the total curvature along $u$ is bounded from above by $$\frac{R}{\ep}\cdot \frac{6}{R^2} = \frac{6}{\ep R}.$$
Therefore, if $R > \frac{6}{\ep^2},$
then the total curvature along $u$ is bounded from above by $\ep$.
\end{proof}
In our current setting, 
as $a_i \to \infty$ and $\Mod(A_i) \to \infty$, one can easily show that, for every $\ep$, the vertical segment $u_i$ satisfies the conditions of \Cref{VerticalTotalCurvatureBound} when $i$ is sufficiently large. 
Thus the following corollary holds. 

\begin{corollary}
The total (principal) curvature of the vertical segment $\Ep_i | u_i$ limits to zero as $i \to \infi$.
\end{corollary}
    
 We next show that the rotational holonomy along $u_i$ asymptotically vanishes as $i \to \infty$. 
\begin{lemma}\Label{AlmostNoTorsion}
For every $\ep > 0$, if $R > 0$ is sufficiently large, then, if a vertical segment $v$ on a $\CP^1$-surface $C$ has length less than $R/\ep$  and a distance at least $R$ from the singular set w.r.t. the singular Euclidean structure of $C$, then, letting $\Ep$ be its Epstein surface,  
the derivative of rotation of its $\Ep$-image is bounded from above by $\ep$. 
Moreover, the total rotation of its $\Ep$-image bounded from above by $\ep$ with respect to the geodesic $\ell$ connecting the endpoints of $\Ep$. 
\end{lemma}
\begin{proof}
Fix $\ep > 0$.
Let $v$ be a vertical segment on $C$ of length less than $R/ \ep$. 
Let $\alpha \col [0, \ell] \to \H^3$ be the curve $\Ep \circ v$, where $\ell$ is the Euclidean length of $v$. 
Let $s(t)$ be the geodesic segment  in $\H^3$ connecting $\alpha(0)$ and $\alpha(t)$ for each $t \in [0, \ell]$.
For $u \in [0, \ell]$, let $\Ep(u)$ be the surface which $s(t)$ sweeps out over $t \in [0, u]$, so that $\Ep(u)$ is bounded by $\alpha([0, u])$ and the geodesic segment $s(u)$ connecting its endpoints. 
Then, the intrinsic metric of  $\Ep(u)$ is a hyperbolic surface. 
Then, if $R > 0$ is sufficiently large, then $\Ep(u)$ isometrically embeds into a hyperbolic plane $\H^2$ so that its image is bounded by a geodesic segment isometric to $s(u)$ and a curve isometric to $\alpha(u)$. 
The curvature of the second segment is bounded from above the curvature of $\alpha \vert [0, u]$ at every point. 

Therefore, if $R > 0$ is sufficiently large, then the area of $\Ep$ is less than $\ep$ by the Gauss-Bonnet theorem to $\Ep$, since the total curvature  $\alpha$ is small.  
Let $n(t)$ denote the unit normal vector  $\Ep_\ast$ at $u(t)$.
Let $n'(t)$ be the parallel transport of $n(t)$ along the geodesic segment $s(t)$, so that $n'(t)$ be a tangent vector at $\alpha (0)$. 
By the Gauss-Bonnet theorem, the norm of the derivative $d n'(t)/ d t$  is bounded from above by the curvature of $\alpha$ and the derivative of the area of $\Ep(t)$ (\Cref{fSmallRotation}).
Thus, the total rotation of $n'(t)$ from $t = 0$ to $t = \ell$ is bounded from above by the sum of the total curvature of $\alpha$ and the total area of $\Ep$.
Therefore, by the combination of the small upper bounds above if $R> 0$ is sufficiently large, the total rotation is bounded by $\ep$.
 \begin{figure}
\begin{overpic}[scale=.2
] {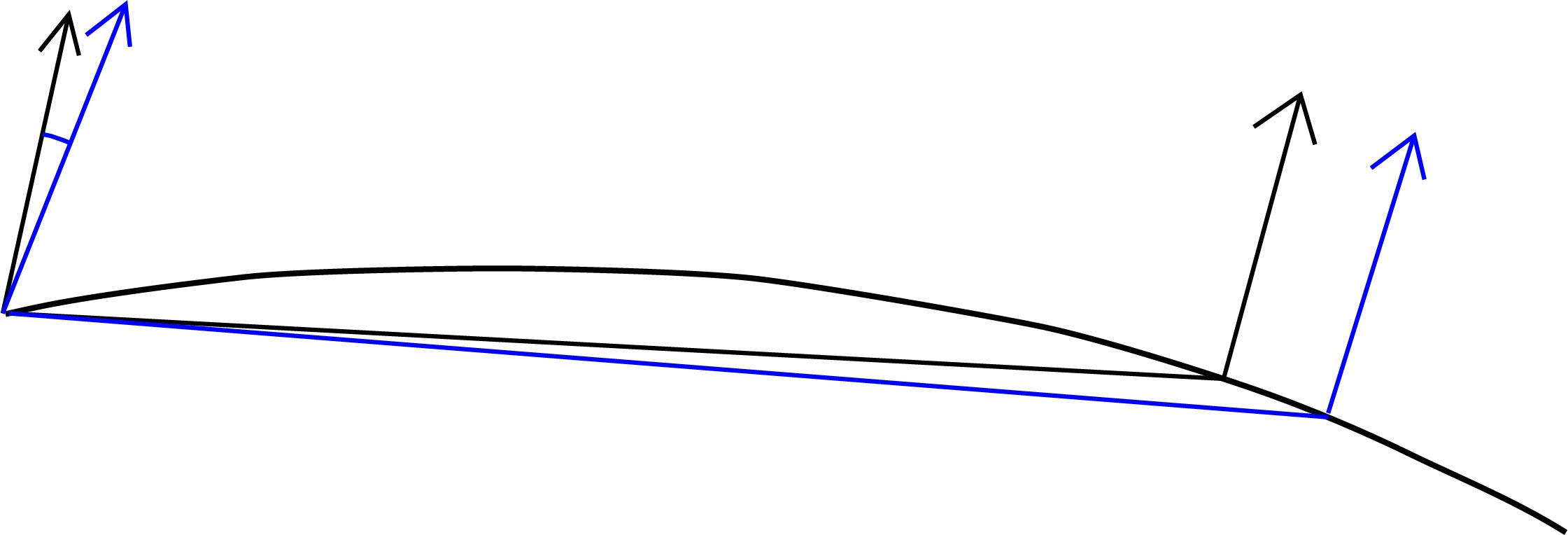} 
       \put(50 ,17 ){\contour{white}{$\alpha$}}  
         \put(40 ,5 ){\contour{white}{$s(t)$} } 
          \put( 70, 18){\contour{white}{$n(t)$ }  }
             \put(-10, 20){\contour{white}{$n'(t)$ }  }
      \end{overpic}
\caption{Infinitesimal change of the rotation angle $n'(t)$.}\label{fSmallRotation}
\end{figure}
\end{proof}

Next we analyze the holonomy along the horizontal segment $w_i$.
By Lemma \ref{DumasLemma} (\ref{iDumasHorizontal}), 
\begin{equation*}
 \length_{\H^3}  \Ep_i w_i  < \frac{6 \length_{E_i} m_i}{(a_i/3)^2} \to 0,
\end{equation*}
as $i \to \infty.$

\begin{proposition}\Label{VerticalAlongHorizontal}
Let $v_i(t)$ denote the tangent vector of $\Ep_i$ at $\Ep_i w_i(t)$ in the direction of $V_i$. 
For every $\ep> 0$, if $i$ is large enough, then
along $w_i$, $\Ep^\ast w_i(t)$ is contained in an $\ep$-ball in the unit tangent bundle ${\rm T^1} \H^3$. 
\end{proposition}

\begin{proof}
Let $\ti{E}_i$ be the universal cover of $E_i$.
Pick a lift $\ti{u}_i$ of the vertical segment $u_i$ in $E_i$ to $\ti{E}_i$.
Let $R_i$ be a Euclidean rectangle, in $\ti{E_i}$, bounded by vertical and horizontal edges,  such that $w_i$ divides $R_i$ into two isometric rectangles of half height (Figure \ref{fHorizontalImage}, left).
We may in addition assume that the height of $R_i$ divided by the width of $R_i$ goes to zero as $i \to \infi$. 

The vertical foliation $V_i$ and the horizontal foliation $H_i$ of $E_i$  induce a vertical and a horizontal foliation of $R_i$. 
By Lemma \ref{DumasLemma} (\ref{iDumasVertical}), for every $\ep > 0$, if $i$ is large enough, the restrictions of $\Ep_i$ to  vertical leaves in $R_i$ are $(\sqrt{2} -\ep, \sqrt{2} + \ep)$-bi-Lipschitz. 
By Lemma \ref{DumasLemma} (\ref{iDumasHorizontal}), the $\Ep_i$-images of  the horizontal leaves in $R_i$ have diameters less than $\ep$. 
Therefore, for sufficiently large $i$, the images of vertical leaves of $R_i$ are pairwise $\ep$-close in the Hausdorff metric  (Figure \ref{fHorizontalImage}  below).
As $v_i$ is tangent to the image of such a vertical leaf, we have the lemma.
   \endgroup 

\begin{figure}
\begin{overpic}[scale=.4
] {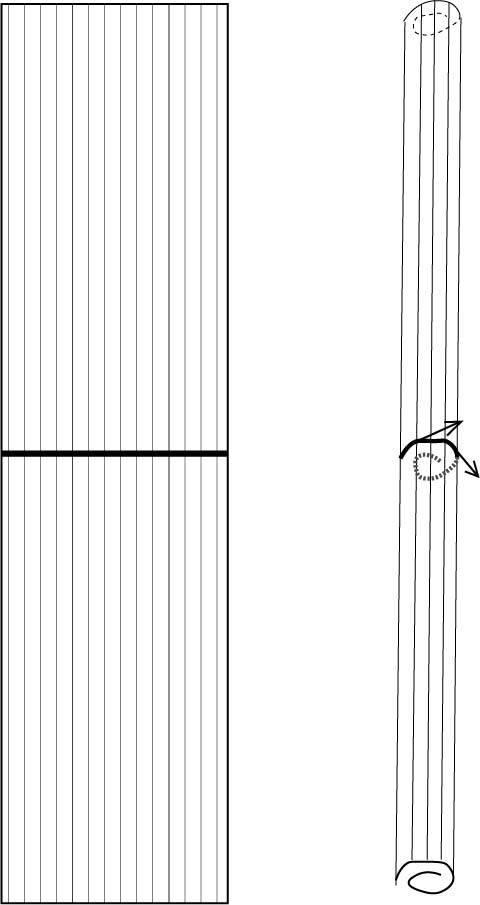} 
  \put(10, 28){\contour{white}{$R_i$}}  
  \put(50, 28){\contour{white}{$\Ep_i(R_i)$}}  
    \put(10 ,52 ){\contour{white}{\textcolor{Black}{$w_i$}}}  
      \end{overpic}
\caption{}\label{fHorizontalImage}
\end{figure}
\end{proof}
  We have already shown a good approximation of the holonomy along the vertical segment $u_i$. 
    For every $\ep > 0$, if $i$ is sufficiently large, then the translation length along $u_i$ is $(1 + \ep)$-bilipschitz to $\sqrt{2}$ times the Euclidean length of $u_i$ and the rotation is less than $\ep$ (\Cref{AlmostNoTorsion}). 
On the other hand, by Proposition \ref{VerticalAlongHorizontal} and \Cref{HorizontalTotalCurvature}, if $i$ is sufficiently large, then the total rotation along the horizontal segment $w_i$ is $(1 + \ep, \ep)$-bi-Lipschitz to $\sqrt{2}$ times the Euclidean length of $w_i$ and the translation is less than $\ep$. 
Thus we obtained, (\ref{iVerticalTranslation}) and (\ref{iHorizontalTotalRotation}).
\Qed{CylinderHonomyEstimate}

    \subsection{The exponential map and Epstein surfaces}
    Recall that, given a $\CP^1$-structure $C = (X, q)$ on $S$,  for $x \in C$, $~d(x)$ is the Euclidean distance from the singular set of the singular Euclidean structure $E$ induced by the holomorphic quadratic differential $q$. 
Note that, if $x \in C$ is not a singular point of $E$, then there is a neighborhood $U$ of $x$ in $E$ so that $U$ is isometrically embedded in the Euclidean plane $\C \cong \E^2$ so that 
 vertical leaves of $E$ in $U$ map into horizontal lines of $\C$, and horizontal leaves map into vertical lines. 

Consider the $\exp\col \C \to \C \minus \{0\}$.
Its domain $\C$ is isometrically identified with the Euclidean plane $\E^2$, and the codomain $\C \minus \{0\}$ admits a push-forward Euclidean metric.
Note that this induced Euclidean metric on $\C \minus \{0\}$ is invariant under the dilations $\C \to \C: z \mapsto k z$ for all $k \in \C \minus \{0\}$.
Therefore, given, any two distinct points $p, q$ in $\CP^1$, by a conformal mapping from  $\CP^1 \minus  \{p, q\}$ to $\C \minus \{0\}$, the complement $\CP^1 \minus \{p, q\}$ has the push-forward Euclidean metric. 
By abuse of notation, we denote this composition by $\exp\col \C \to \minus  \{p, q\}$ and call it the {\it normalized exponential map. }

Let $(p, q)$ be the geodesic in $\H^3$ connecting $p$ to $q$. 
Recalling that $\CP^1$ is the ideal boundary of $\H^3$, 
let $\Psi\col \CP^1 \minus \{p, q\} \to (p, q)$ be the orthogonal projection along a geodesic rays in $\H^3$. 
Let $\Psi_\ast\col \CP^1 \minus \{p, q\}  \to T^1\H^3$ be the map taking $z \in \CP^1 \minus \{p, q\}$ to the unit tangent vector at $\Psi(z) \in \H^3$ which is tangent to the geodesic ray from $\Psi(z)$ to $z \in \bdr \H^3$. 

For $r > 0$, let $Q_r(z)$ be a $r$-neighborhood of a point $z$ of the singular Euclidean surface $E$ in the $L^\infty$-metric (w.r.t. the vertical and the horizontal directions). 
 If $Q_r(z)$ contains no singular point, then it is a square with horizontal and vertical edges of length $2 r$. 

    \begin{proposition}\Label{AlmostExponentialMap}
For every $\ep > 0$, there is $R > 0$ such that, if $z \in \til{C}$ satisfies $d(z) > R$,  then we have a normalized exponential map $\exp\col \C \to \CP^1 \minus \{p, q\}$ and can isometrically embed the $\frac{1}{\ep}$-neighborhood $Q_{1/\ep}(z)$ of $z$ in $\C$ exchanging the horizontal and the vertical directions, such that, in the $C^0$-metric, 
\begin{enumerate}
\item  the restriction of the Epstein surface $\Sigma$ to $Q_{1/\ep}(z)$ is $\ep$-close to $w \mapsto \Psi_{\ast} \exp (\frac{w}{\sqrt{2}})$,  \Label{iSpeteinSurfaceAwayFromZero}
\item the restriction of $\Sigma_\ast$ to $Q_{1/\ep}(z)$ of $z$ is $\ep$-close to $w \mapsto \Psi_\ast \exp (\frac{w}{\sqrt{2}})$, and \Label{iSpeteinSurfaceNormalAwayFromZero}
\item the restriction of  the developing map $f$ to $Q_{1/\ep}(z)$ is $\ep$-close to the normalized exponential map. \Label{iAlmostExponentialMapDev}
\end{enumerate}

\end{proposition}
\proof
 we prove the desired approximations by showing them along all leaves of the restriction of the vertical foliation $V$ and the horizontal foliation $H$ to the square $Q_{1/\ep}(z)$. 

For every $\epsilon' > 0$, 
by Lemma 
\ref{DumasLemma} and \Cref{AlmostNoTorsion}, if $R > 0$ is sufficiently large, then
\begin{enumerate}[(i)]
\item  the restriction of  $\Sigma$ to each leaf of the vertical foliation $V$ in $Q_{\frac{1}{\epsilon'}}(z)$ is a smoothly $(\sqrt{2} - \epsilon', \sqrt{2} + \epsilon')$-bilipschitz embedding, \Label{iVerticalSegment}
\item the restriction of  $\Sigma$ to each leaf of the horizontal foliation $H$ in $Q_{\frac{1}{\epsilon'}}(z)$  has derivative less than $\epsilon'$, and   \Label{iHorizontalSegment}
\item the derivative of the rotation of $\Sigma_\ast$ along a vertical leaf in $Q_{\frac{1}{\epsilon'}}(z)$ is bounded from above by $\ep'$, and the total rotation along the leaf is also bounded from above by $\ep'$. \Label{iNormalAlongVertical}
\end{enumerate}

Pick a vertical leaf $v_0$ in $Q_{\frac{1}{\ep}}(z)$, and let $\ell$ be the geodesic in $\H^3$ passing through the endpoints of the $(\sqrt{2} - \epsilon', \sqrt{2} + \epsilon')$-bilipschitz curve $\Sigma \vert v_0. $
We normalize the exponential map with respect to the endpoints of this geodesic. 
Then (\ref{iVerticalSegment}) and (\ref{iHorizontalSegment}) implies  (\ref{iSpeteinSurfaceAwayFromZero}) with this normalization.

We next show (\ref{iAlmostExponentialMapDev}).
We first analyze $f$ on each vertical leaf. 
By (\ref{iVerticalSegment}) and (\ref{iNormalAlongVertical}), the restriction of the developing map $f$ to $v_0$ is $\epsilon'$-close to the normalized exponential map, by isometrically embedding $e$ onto $\C \cong \E^2$ in the scaled Euclidean metric $\sqrt{2} E$ (i.e. the metric on $v_0$ is scaled by $\sqrt{2}$).  

The $\Sigma$-images of horizontal segments are very short curves in $\H^3$. 
Therefore,  for every $\ep' > 0$, if $R > 0$ is sufficiently large, then for each vertical leaf $v$ of $Q_{\frac{1}{\ep}}(z)$, the restriction of $f$ to $v$ is $\ep'$-close to the restriction of the normalized exponential map to a vertical segment in $\C$ by isometrically embedding $v$ w.r.t. $\sqrt{2} E$. 

Next, we analyze $f$ on horizontal leaves. 
Let $h$ be a horizontal leaf in $Q_{\frac{1}{\ep}}(z)$.
Consider the vector field along $h$ consisting of the unit vectors in the vertical direction.
Then, for every $\epsilon' > 0$, if $R > 0$ is sufficiently large, then, as in the proof of \Cref{VerticalAlongHorizontal},  
the image of the tangent vectors are $\ep'$-close to each other in the $C^0$-topology. 
By the curvature estimate along the horizontal direction in  
Lemma \ref{HorizontalTotalCurvature}, for every $\ep' > 0$ if $R > 0$ is large enough,  the amount of the total rotation of $f$ along every horizontal segment in  $Q_{\frac{1}{\ep}}(z)$ is close to the horizontal length times $\sqrt{2}$. 
Therefore, a restriction of $f$ to every horizontal segment $h$ is $\ep'$-close to the restriction of $\exp$ when $h$ is isometrically embedded onto a horizontal segment after scaling the length of $h$ by $\sqrt{2}$. 
Therefore, a restriction of $\Sigma_\ast$ to every horizontal segment $h$ is $\ep'$-close to the restriction of $\Psi_\ast \exp$ when $h$ is isometrically embedded onto a horizontal segment w.r.t the $\sqrt{2} E$-metric. 

We proved that the restrictions of $f$ to horizontal and vertical leaves in $Q_{\frac{1}{\ep}}(z)$ are $\ep'$-close to the normalized exponential map when  $Q_{\frac{1}{\ep}}(z)$ is isometrically embedded in $\C$.
This immediately implies (\ref{iAlmostExponentialMapDev}).

Finally (\ref{iSpeteinSurfaceAwayFromZero}) and (\ref{iAlmostExponentialMapDev}) immediately imply (\ref{iSpeteinSurfaceNormalAwayFromZero}), since $f(z)$ and $\Psi(z)$ determines $\Psi_\ast(z)$.
\Qed{AlmostExponentialMap}

\section{Holonomy maps for surfaces with punctures}\Label{sHolonomySurfaceWithPuncture}

\subsection{Classification of cusps of $\CP^1$-structures}

\begin{definition}
Let $F$ be a surface with punctures. 
A $\CP^1$-structure on $F$ is a pair $(X, q)$ of a Riemann surface structure $X$ on $F$ and a holomorphic quadratic differential $q$, such that at each puncture of $X$, $q$ is at most a pole of order two.
\end{definition}
This class is a natural class to consider, especially in our setting due to the upper injectivity radius bound (see Theorem \ref{UpperInjectivityRadiusBound}).
\begin{proposition}\Label{CuspClassification}
Let $F$ be a closed surface with at least one puncture $c$ such that the Euler characteristic of $F$ is negative. 
 Let $C = (f, \rho)$ denote a $\CP^1$-structure on $F$ expressed by a developing pair.
Denote by $\ell_c$ the peripheral loop around $c$. 
Let $C \cong (\tau, L)$ denote Thurston parameters, and $(E, V)$ be the singular Euclidean structure $E$ with the vertical foliation $V$ given by the Schwarzian parameters of $C$. 
\begin{enumerate}
\item Suppose that a cusp neighborhood of $c$ in $E$ is an expanding cylinder of infinite modulus shrinking towards $c$. 
Then 
\begin{itemize}
\item $\rho(\ell_c)$ is parabolic, 
\item $c$ has  a horodisk quotient neighborhood, and
\item in Thurston parameters $(\tau, L)$,  $c$ also has a horodisk quotient neighborhood where the lamination $L$ is the empty lamination. 
\end{itemize}
  \Label{iShrinkingCylinder}
\item Suppose that a cusp neighborhood of $c$ in $E$ is a (half-infinite) flat cylinder $F$ of infinite modulus. \Label{iFaltCylinderEnd}
Then exactly one of the following holds.
\begin{enumerate}
\item \Label{iVNotOrthogonal} The  circumferences of $F$ are {\it not} orthogonal to $V$,  $\rho(\ell_c)$ is hyperbolic, and $\sqrt{2}\int_{\ell_c} \sqrt{q}$ is its complex translation length. 

In Thurston parameters, the cusp $c$ corresponds to boundary component $b$ of $\tau$ whose length is the real part of the translation length (in $\C / 2 \pi i \Z$). 
\item \Label{iVOrthogonalt} The circumferences of $F$ are orthogonal to  $V$. 
\begin{itemize}
\item If $\sqrt{2}\, V(\ell_c)$ is {\it not} a $2\pi$-multiple, then $\rho(\ell_c)$  is an elliptic element of angle $\sqrt{2} \int_{\ell_c} \sqrt{q} \ (\in \R)$.
In Thurston parameters, $c$ is a cusp of  $\tau$ and the total weight of leaves of $L$ around $c$ (counted with multiplicity) is, modulo $2\pi$, equal to 
the rotation angle of $\rho(\ell_c)$.
\item  If $\sqrt{2}\, V(\ell_c)$ is a $2\pi$-multiple, then $\rho(\ell_c )$ is either the identity $I$ or a parabolic element. 
In Thurston parameters, $c$ is a cusp of $\tau$ and the total weight of  $L$ around $c$  is the $2 \pi$-multiple. 
\end{itemize}
\end{enumerate}
\end{enumerate}
\end{proposition}
In (\ref{iVOrthogonalt}), by ``counted with multiplicity'', we mean that, if a single leaf of $L$ has both endpoints at $c$, the weight of the leaf is counted twice.  

\proof
(\ref{iShrinkingCylinder})
We first describe an intuition, and then make it precise. 
As the Euclidean distance to the cusp is finite in $E$,  in the hyperbolic metric on $X$,  the quadratic differential  $q$  vanishes asymptotically towards the cusp $c$. 
A Riemann surface with the zero differential (in our parametrization) corresponds to a hyperbolic structure.

To make it precise, for $t > 0$, let $D_t$ be the punctured disk of radius $t$ centered at $c$. 
Note that the $c$ may be the zero of the quadratic differential $q$ induced by $C$. 
Thus, if $t > 0$ is small enough, $D_t$ is a union of the Euclidean semi-disks of radius $t$ foliated by geodesics parallel to the diameter segment.
Consider the restriction of $q$ to $D_t$. 
Then, by conformally identifying a once-punctured unit disk with $D_t$, the holomorphic quadratic differential on $D_t$, the differential goes to zero uniformly on every compact subset as $t \to \infty$.

The solution of the Schwarzian equation depends continuously in the differential. 
As a punctured disk with the zero differential corresponds to a hyperbolic structure $h$ with a cusp at the puncture, and the holonomy around the cusp is parabolic.
Therefore, the developing map of $D_t$ converges to the developing map of the hyperbolic cusp-neighborhood structure $h$, which is a quotient of horodisk by the infinite cyclic group generated by a parabolic element. 
By the equivariance property of the developing maps, the holonomy of $D_t$ around the cusp must converge to a parabolic element, and as the holonomy of $D_t$ around the cusp is independent of $t > 0$, the holonomy is genuinely parabolic. 
Moreover, if one deforms a little bit the hyperbolic structure $h$ on the punctured disk to any other $\CP^1$-structure on the punctured disk keeping the holonomy around the cusp parabolic, it still contains a horodisk quotient as a cusp neighborhood.  
 Therefore $c$  has a horodisk quotient neighborhood in $C$. 

 In Thurston parameters, $c$ is a cusp of $\tau$, and $L$ is the empty lamination in a sufficiently small neighborhood of $c$. 

(\ref{iFaltCylinderEnd}) By  \Cref{AlmostExponentialMap}, the developing map of the half-infinite flat cylinder becomes closer and closer to the exponential map $\exp \col \C \to \C^\ast$ as a point in the domain approaches the cusp, where, in the domain $\C$,  the vertical direction corresponds to the real direction and the horizontal direction corresponds to the imaginal direction (to be precise, the $\exp$ is composed with the calling to the domain $\C$ by $\sqrt{2}$). 
Thus the assertions about the holonomy along $\ell_c$ hold.

It remains only to show the description in Thurston parameters.

(\ref{iVNotOrthogonal}) 
By Proposition \ref{CylinderHonomyEstimate},
outside of a large compact set of $F$,  all circumferences of $F$ are admissible loops. 
Therefore an appropriate neighborhood of $c$  corresponds to an infinite grafting cylinder. 
By \cite[Proposition 8.3]{Baba-17},  the hyperbolic surface $\tau$ has a (possibly open) boundary component corresponding to $c$, and its boundary length is indeed the translation length of the hyperbolic element $\rho(\ell_c)$.

(\ref{iVOrthogonalt})
The developing map in an appropriate cusp neighborhood is the exponential map $\exp\col \C \to \C^\ast$ so that the deck transformation corresponds to the translation in the imaginary direction in the domain. 

Therefore, $c$ is a cusp of $\tau$ and the total weight of leaves of $L$ near $c$ must be the length of the circumference times $\sqrt{2}$ (Proposition \ref{CylinderHonomyEstimate} (\ref{iHorizontalTotalRotation})). 
\Qed{CuspClassification}

\subsection{$\PSL_2\C$ and fixed points on $\CP^1$}\Label{sHatPSL}
In order to construct an appropriate holonomy map for a surface with punctures, we will make $\PSL_2\C$ slightly bigger as a topological space, by carefully pairing its elements with their fixed points on $\CP^1$.
Let $(\CP^1)^2/\Z_2$ denote the set of unordered pairs of points on $\CP^1$. 
Let $\reallywidehat{\PSL_2\C}$ be the set of all pairs $(\gamma, \Lambda) \in \PSL_2\C \times ((\CP^1)^2/ \Z_2)$ such that
\begin{itemize}
\item if $\gamma$ is a hyperbolic element with zero rotation (i.e. $\tr \gamma \in \R \minus [-2,2]$ when $\gam$ is lifted to ${\rm SL}(2, \C)$), then $\Lambda$ is a pair of (not necessarily distinct) fixed points of $\gamma$,  and
\item otherwise,  $\Lambda$ is the pair $(a, a)$ of identical fixed points $a \in \CP^1$ of $\gamma$.
\end{itemize}

We call the pair $\Lambda$ a {\it framing}.  
In particular, if $\gamma = I$, then $\Lambda$ can be $(a, a)$ for any $a \in \CP^1$.
The second case also includes the case where $\gamma$ is a hyperbolic element with non-zero torsion.
(By abuse of notation, if $\Lambda$ is a pair $(a, a)$ of identical points  on $\CP^1$, for simplicity, we may regard $\Lambda$ as a single point $a$.)

Fock and Goncharov introduced a framing of a representation, which equivariantly assigns a single fixed point to each peripheral element (\cite{FockGoncharov06}). 
What is new here is that we are assigning a pair of fixed points in the first case.

Next we define a (non-Hausdorff) topology on  $\reallywidehat{\PSL_2\C}$ by the following open base of neighborhoods at each $(\gamma, \Lambda) \in \reallywidehat{\PSL_2\C}$.
\begin{itemize}
\item 
    If $\gam$ is hyperbolic, then, 
for every (small) connected neighborhood $U$ of $\gamma$ in  $\PSL_2\C$ consisting of hyperbolic elements,  
the set of all pairs $(\gamma', \Lambda') \in \reallywidehat{\PSL_2\C}$ such that
\begin{itemize}
\item if $\tr \gam$ is real and  $\sharp \Lambda = 2$, then for $\gam' \in U$ with $\tr \gam'$ real,  $\sharp \Lambda' = 2$, and
\item otherwise $\gamma' \in U$ and,  $\Lambda'$ is a pairs of identical points identified with $\Lambda$ by identifying  the fixed points of $\gamma$ with those of $\gamma'$ by a path connecting $\gam$ to $\gam'$ in $U$.
\end{itemize}
   
   \item If $\gamma$ is {\it not} hyperbolic, then the topology near $(\gamma, \Lambda)$ is given by the product topology of $\PSL_2\C \times (\CP^1)^2/ \Z_2$ equipped with the Hausdorff topology on $ (\CP^1)^2/ \Z_2$.

\end{itemize}

\begin{remark}\Label{CuspParameter} 
Let $C = (f, \rho)$ be a $\CP^1$-structure  on a surface with punctures. 
Let $\alpha \in \pi_1(S)$ be such that its free homotopy class is the peripheral loop around a cusp $c$ of $C$. 
Then $\gam$ corresponds to a unique element  in $(\CP^1)^2/\Z_2$ as follows:
As the universal cover $\til{C}$ of $C$ is conformally identified with $\H^2$ by the uniformization,  let $\til{c}$ be the point on the ideal boundary of $\bdr \H^2$ fixed by $\alpha \in \pi_1(S)$. 
Let $(\tau, L) \in \TT \times \ML$ be the Thurston parametrization of $C$, and let $\LLL$ be the circular measured lamination on $C$ which descends to $L$.
For each leaf $\ell$ of $\ti\LLL$ ending at $\ti{c}$, the corresponding endpoint of the circular arc $f(\ell)$ on $\CP^1$ is a fixed point of $\rho(\alpha)$.
If $\LLL$ is non-empty in a small neighborhood of the cusp, let $\Lambda$ be the set of such half leaves of $\ti\LLL$ ending at $c$.
Then $\alpha$ corresponds to a unique element $(\rho(\alpha), \Lambda)$ in $\reallywidehat{\PSL_2\C}$.
If $\LLL$ is empty near the cusp,  an appropriate cusp neighborhood of $c$ is a horodisk quotient, and $\alpha$ corresponds to $(\rho(\alpha), \Lambda)$, where $\Lambda$ is the parabolic fixed point of $\rho(\alpha)$.
\end{remark}
\subsection{Cusp neighborhoods in Thurston parameters}
The following lemma determines the isomorphism classes of cusp neighborhoods of $\CP^1$-structures in Thurston coordinates. 
\begin{lemma}\Label{CuspNbhdInThurstonCoordinates}
Let $C = (f, \rho)$ be a $\CP^1$-structure on a surface $F$ with cusps.
Let $C \cong (\tau, L)$ be Thurston parameters of $C$.
Then, for each cusp $c$ of $C$,  its small neighborhood  (i.e. its germ) in $C$ is determined by
\begin{itemize}
\item  the holonomy $h \in \PSL_2\C$ around $c$,
\item the transversal measure of a peripheral loop around $c$ given $L$, and
\item  if $h$ is hyperbolic, the direction in which the leaves of $L$ spirals towards the boundary component. 
\end{itemize}
(See Figure \ref{fCuspNeighborhoods}.)
\end{lemma}

\begin{figure}
\begin{overpic}[scale=.4,
] {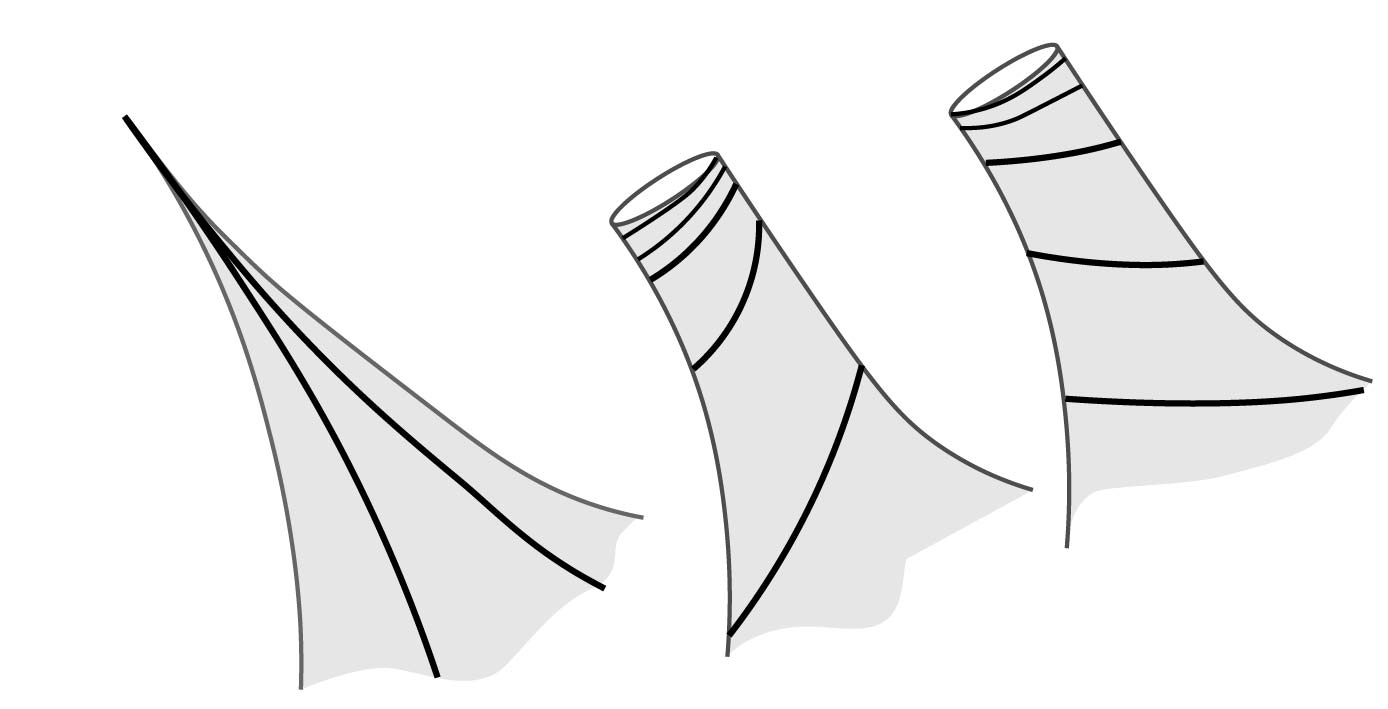} 
  \put(0 , 0){\textcolor{Black}{\small elliptic, parabolic, identity}} 
  \put(45 , 46){\textcolor{Black}{\small hyperbolic}} 
 \put(60 , 7){\textcolor{Black}{Different spiraling directions}}  
      \end{overpic}
\caption{Cusp neighborhoods in Thurston parameters}\Label{fCuspNeighborhoods}
\end{figure}

\proof
Let $(E, V)$ be the pair of a singular Euclidean structure $E$ on $F$, and $V$ be a vertical foliation on $E$ induced by $C$. 

{\it Hyperbolic Case.} First suppose that $h \in \PSL_2\C$ is hyperbolic.  
Then, by Proposition \ref{CuspClassification}, its cusp neighborhood,  in $(E, V)$,  corresponds to a half-infinite cylinder $A$, and the complex translation length is $\sqrt{2}\int_{\ell_c} \sqrt{q}$, where $\ell_c$ is a peripheral loop of $c$.  

The developing map $f$  of a small neighborhood of $c$ is a restriction of the exponential map $\C \to \C^\ast$.
Thus the complex translation length determines the deck transformation on the domain $\C$ by $\Z \cong \langle \ell_c \rangle$, which determines the $\CP^1$-structure of a small cusp neighborhood. 

The cusp $c$ corresponds to the geodesic boundary circle $b$ of $\tau$ whose length is equal to the translation length of  $h$. 
By the properties of bending maps,  one can show that the total weight of $L$ along $\ell_c$ times $\sqrt{2}$ is the rotational angle of $h$ and the direction of rotation in which leaves of $L$ spiral towards $b$ determines the orientation of the angle (Figure \ref{fSpiralDirection}). 
\begin{figure}
\vspace{5mm
}
\begin{overpic}[scale=.4, 
] {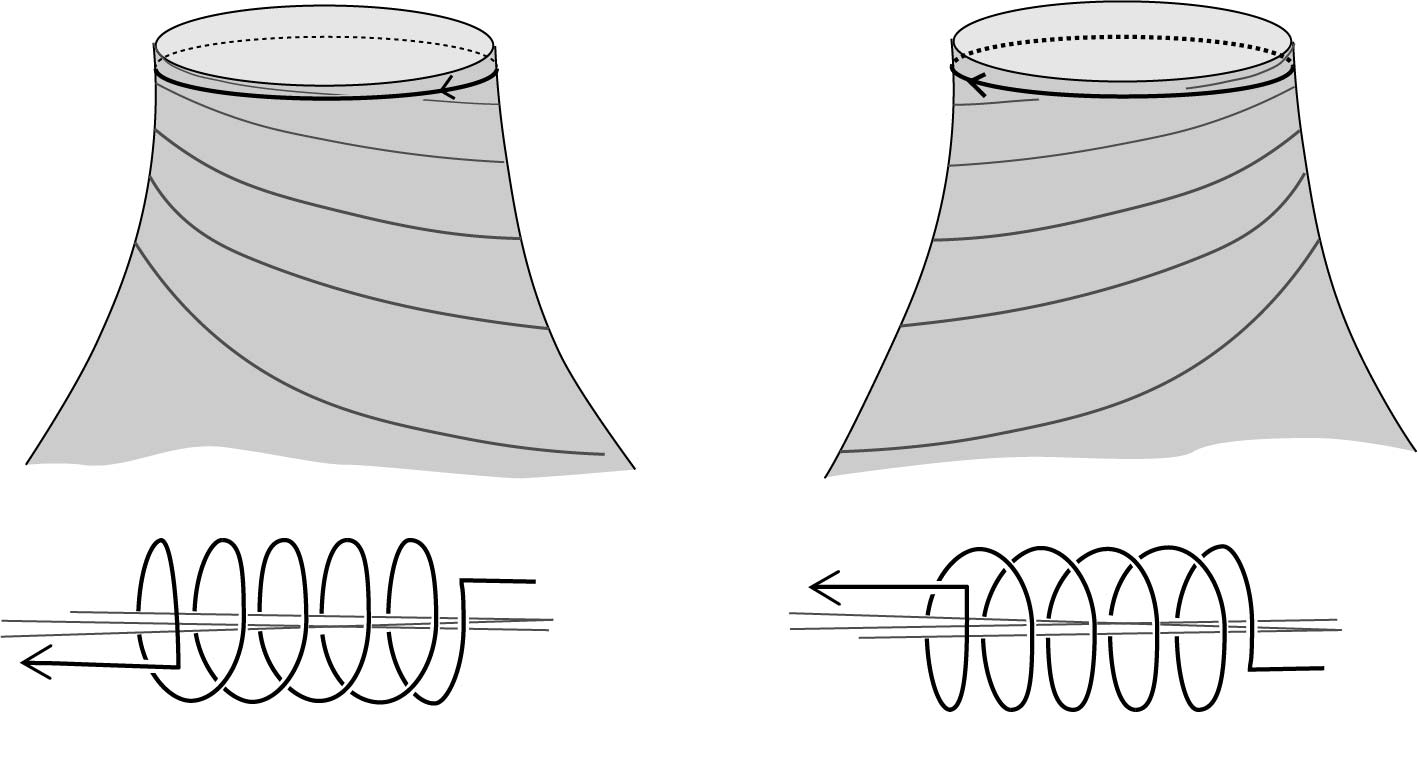} 
      \end{overpic}
\caption{The opposite spiral directions give the holonomy the opposite rotational directions.}\label{fSpiralDirection}
\end{figure}

{\it Parabolic Case.}
Suppose that $h$ is parabolic. 

If a neighborhood of a cusp $c$ in $E$ is an expanding cylinder shrinking towards $c$, then a neighborhood of $c$ in $(\tau, L)$ is a hyperbolic cusp with the empty lamination 
(Proposition \ref{CuspClassification} (\ref{iShrinkingCylinder})).

Next suppose that the cusp neighborhood of $c$ in $(E, V)$ is a half-infinite flat cylinder $A$ in $E$. 
Then the circumferences of $A$ are orthogonal to $V$, and $\sqrt{2}V(\ell_c)$ is a positive $2\pi$-multiple.

Let $\Ep\col \til{X} \to \H^3$ be the Epstein map associated with $C = (X, q)$.
Let $\ti{V}$ be the pull-back of $V$ to the universal cover of $E$, and let $\ti{c}$ be the lift of $c$ to the ideal boundary of $\ti{X} \cong \H^2$.
Let $\gam \in \pi_1(F)$ be the element which fixes $\ti{c}$ such that its free homotopy class is $\ell_c$.
Then, for every leaf $\ell$ of $\ti{V}$ ending at $\ti{c}$, its image $\Ep(\ell)$ is indeed a quasi-geodesic limiting to the parabolic fixed point of $\rho(\gam)$ on $\CP^1$,  and its curvature of $\Ep(\ell)$ converges to zero as it approaches the fixed point by Lemma \ref{DumasLemma}. 
Therefore $c$ corresponds to a cusp of $\tau$.
By 
Proposition
\ref{CylinderHonomyEstimate}, the total weight of the leaves must be  $\sqrt{2}V(\ell_c)$.

{\it  Elliptic Case.}
The proof when $h$ is elliptic is similar to the parabolic case. 
\Qed{CuspNbhdInThurstonCoordinates}

    \fontsize{12pt}{12pt}\selectfont
  Let $D$ be the unit closed disk in $\C$ centered at the origin $O$.  
 Let $D^\ast = D \minus \{O\}$, and let $\ell$ be the peripheral loop around the origin.  
  Let $\PPP(D^\ast)$ denote the space of all developing pairs $(f, h)$ for the $\CP^1$-structures on $D \minus \{O\}$ (not up to $\PSL_2\C$) so that $O$ is a cusp and the boundary circle is smooth, where $f\col \til{D^\ast} \to \CP^1$ is the developing map and $h \in \PSL_2\C$ is the holonomy along $\ell$. 
Recall from Remark \ref{CuspParameter} that each cusp corresponds to a unique element $(\gamma, \Lambda)$ in $\reallywidehat{\PSL_2\C}$.
Let $\ch{D}^\ast$ be a subsurface of $D^\ast$ obtained by removing a regular neighborhood of the boundary circle of $D^\ast$. 

By the following proposition, the deformation of the $\CP^1$-structures of the cusp neighborhoods is locally modeled on $\reallywidehat{\PSL_2\C}$.

\begin{proposition}\Label{PerturbingCusp}
Let $F$ be a closed surface minus finitely many points, and let $C$ be a $\CP^1$-structure on $F$, and pick its developing pair $(f, \rho)$.  
Then,  each cusp $c$ of $C$ has a disk neighborhood $\Sigma = (f, \gamma) \in \PPP(D^\ast)$ of $c$ in $C$   with the following properties:
\begin{enumerate}
\item Let $(\gam, \Lambda) \in \reallywidehat{\PSL_2\C}$ be the element corresponding to the peripheral loop around $c$.
Then,  for every $\ep > 0$ and every compact subset $K$ of the universal cover $\til{\Sigma}$, there is a subset $U = U(K, \ep)$ of $(\gam, \Lambda)$ in $\reallywidehat{\PSL_2\C}$, such that, for every $(\gamma', \Lambda') \in U$, \Label{iLocalLocalHomeo}
\begin{enumerate}
\item if $\sharp \Lambda = 1$, then there is $\Sigma' = \Sigma'(\gamma', \Lambda') \in \PPP(D^\ast)$  with holonomy $\gam'$ and the framing $\Lambda$, such that its developing map $f'$ of $\Sigma(\gamma', \Lambda')$ is $\ep$-close, in $C^1$-topology, to the developing map $f$ of $\Sigma$ in $K$, 
\item if $\sharp \Lambda = 2$, then there is a neighborhood $W$ of $\gam$ in $\PSL_2\C$, such that, for every $\gamma' \in W$, there is $\Sigma' = \Sigma'(\gamma', \Lambda') \in \PPP(D^\ast)$  with holonomy $\gam'$ and a unique framing $\Lambda$, such that its developing map $f'$ of $\Sigma(\gamma', \Lambda')$ is $\ep$-close, in $C^1$-topology, to the developing map $f$ of $\Sigma$ in $K$. 
\end{enumerate}
\item 
Moreover, $\Sigma'$ is uniquely determined on $\ch{D}^\ast$ by an isotopy of $D^\ast$ (uniqueness near the cusp).
\Label{iUniquenessNearCusp}
\end{enumerate}
 \end{proposition}
 
\proof[Proof of Proposition \ref{PerturbingCusp}]

We divide the proof by the isometry type of $\gam$. 
In each case, we construct a deformation of $\Sigma$ in a small neighborhood in $\reallywidehat{\PSL_2\C}$ by specifying the deformation of a fundamental membrane.

{\it Elliptic Case.}
First,
suppose that $\gam= I$ or $\gam$ is an elliptic element. 
Then the puncture $O$ corresponds to a unique point $f(O)$ on $\CP^1$ by continuously extending $f$.
Then pick a cusp neighborhood $\Sigma$ biholomorphic to a punctured disk, such that the development of the boundary circle is a round circle $\alp$ on $\CP^1$ and there is a unique Lie subgroup of $\PSL_2\C$ isomorphic to $SO(2)$ which preserves $\alpha$ and $f(O)$.  
We identify $\CP^1$ with $ \C \cup \{\infty\}$ so that the puncture $f(O)$ is at the origin and $\alpha$ is the unit circle of $\C$ centered at the origin $f(O)$.

Pick a ``fan-shaped fundamental domain'' in $\til{D^\ast}$ bounded by three circular arcs $e_1, e_2, e_3$ such that
\begin{itemize}
\item $f|e_1$ and $f|e_2$ are radii  of $\alpha$ connecting $f(O)$ to points on $\alpha$, so that $\gam f(e_1) = f(e_2)$ are orthogonal to $\alpha$, and 
\item $f | e_3$ immerses into $\alpha$, and it connects the endpoints of $e_1$ and $e_2$
\end{itemize}
(Figure \ref{fFundamentalDomainTrivialHolonomy}, left). 
Let $q$ be the endpoint of the arc $f(e_1)$ on $r$.  

If the neighborhood $U$ of $(\gam, f(O))$ is sufficiently small, then  given $(\gam', \Lambda') \in U$, one can easily construct a $\CP^1$-structure $\Sigma' = (f', \gam')$ close to $\Sigma$ on $D^\ast$ realizing $(\gam', \Lambda’)$. 
Indeed, we pick $z \in \Lambda'$, we can construct a  fundamental membrane bounded by $e_1', e_2', e_3'$ such that, 

\begin{enumerate}
\item $f'(e_1')$ is a straight line on $\C$ connecting $z$ and $fq$,
\item $f'(e_2')$ is $\gam'(f'(e_1'))$ (which is a circular arc connecting $z$ and $\gam(q)$),
\item $f'(e_3')$ is an arc connecting $q$ to $\gam(q)$ so that  $f(e_3')$ is a segment of a trajectory under a one-dimensional Lie subgroup of the affine transformations of $\C$ preserving $z$, and
\item $f'(e_i)$ is close to $f(e_i)$ in the Hausdorff topology on $\CP^1$.
\end{enumerate}
(see Figure
\ref{fFundamentalDomainTrivialHolonomy}, right).
(The choice of $z$ may not be unique if  $r$ is identity and $\tr r' \in \R \minus [-2,2]$, i.e. hyperbolic without screw motion)

On the other hand, one can easily see that,  for every small deformation $\Sigma'$ of $\Sigma$, there is a ``fan-shaped'' fundamental membrane satisfying all conditions  (1) - (4) such that the fundamental membranes coincide on $\ch{D}^\ast$. 
Therefore, we have the uniqueness property of $\Sigma'$ near the cusp.

{\it Generic hyperbolic case.}
Let $(\tau, L)$ be the Thurston parametrization of $C$, and let $\LLL$ be the Thurston lamination on $C$. 
Let $\ell$ be the peripheral loop around $O$.
Suppose that $\gam$ is hyperbolic and $L(\ell) \neq 0$, so that $\Lambda$ is a single point. 
Then $\tau$ has a geodesic boundary loop $b$ corresponding to the cusp $c$ and, as $L(\ell) > 0$,   leaves of $L$ spiral towards  $b$.
Let $\til{b}$ be a lift of $b$ to the universal cover $\til{\tau}$ of $\tau$, so that $\til{b}$ is a boundary geodesic of $\til{\tau}$.
Then those spiraling leaves lift to geodesics in $\til{\tau}$ having a common endpoint at an endpoint of $\til{b}$; by the bending map $\beta\col \til{\tau} \to \H^3$, the endpoint maps to the point $\Lambda$. 
Accordingly, the leaves of $\LLL$ near the cusp $O$ develop onto circular arcs ending at $\Lambda$.

Normalize $\CP^1 = \C \cup \{\infi\}$ by an element of $\PSL_2\C$, so that $0 = \Lambda$ and the other fixed point of $\gamma$ is at $\infty$. 
Let $(E, V)$ be the foliated singular Euclidean structure given by $C$.
Then, there is a half-infinite flat cylinder $A$ in $E$ which corresponds to a cusp neighborhood of $c$\,; then each circumference has a positive transversal measure given by the horizontal foliation. 
Therefore, one can take a cusp neighborhood $\Sigma$ bounded by a loop $m$  such that $m$ develops onto a {\it spiral} on $\CP^1$, i.e. a curve invariant under a one-parameter subgroup in $\PSL_2\C$ which contains $\gamma$.

Take, similarly,  a  ``fan-shaped'' fundamental domain $F$ in the universal cover $\til{\Sigma}$ which is bounded by three smooth segments $e_1, e_2, e_3$ such that 
\begin{itemize}
\item $e_1$ and $e_2$ are half-leaves of $\ti\LLL$ such that $\gam e_1 = e_2$  and the circular arcs $f(e_1)$ and $f(e_2)$ end at $0 \in \C$,, and
\item $f(e_3)$ is in a segment of the spiral which connects the other endpoints of $f(e_1)$ and $f(e_2)$
\end{itemize}
(Figure \ref{fPerburgingGenericHyperbolicCusp}, left).
Then $\gamma(f(e_1)) = f(e_2)$ by the equivariant property. 

Take a sufficiently small neighborhood $U$ of $(\gamma, \Lambda)$ such that the subset $W \sub \PSL_2\C$ of holonomy elements of pairs in $U$  consists of only hyperbolic elements closed to $\gam$; then,  for all $(\gamma', \Lambda') \in U$,   the fixed point $\Lambda'$ of the hyperbolic element $\gam'$ uniquely corresponds to the fixed point of $\gam$ in $\Lambda$ by every short path connecting $\gam'$ to $\gam$ in $W$. 
Then, similarly to the elliptic case, one can easily find a $\CP^1$-structure on $D^\ast$ close to $\Sigma$  which realizes $(\gamma', \Lambda')$,  by constructing a fundamental membrane close to $F$  (Figure \ref{fPerburgingGenericHyperbolicCusp}). 

On the other hand, for every small deformation $\Sigma''$ of $\Sigma$ realizes $(\gamma', \Lambda')$, one can easily find a fundamental membrane of $\Sigma''$ so that it coincides, on $\ch{D}^\ast$ with that of $\Sigma'$ constructed above. 

\begin{figure}
\begin{overpic}[scale=.4,
] {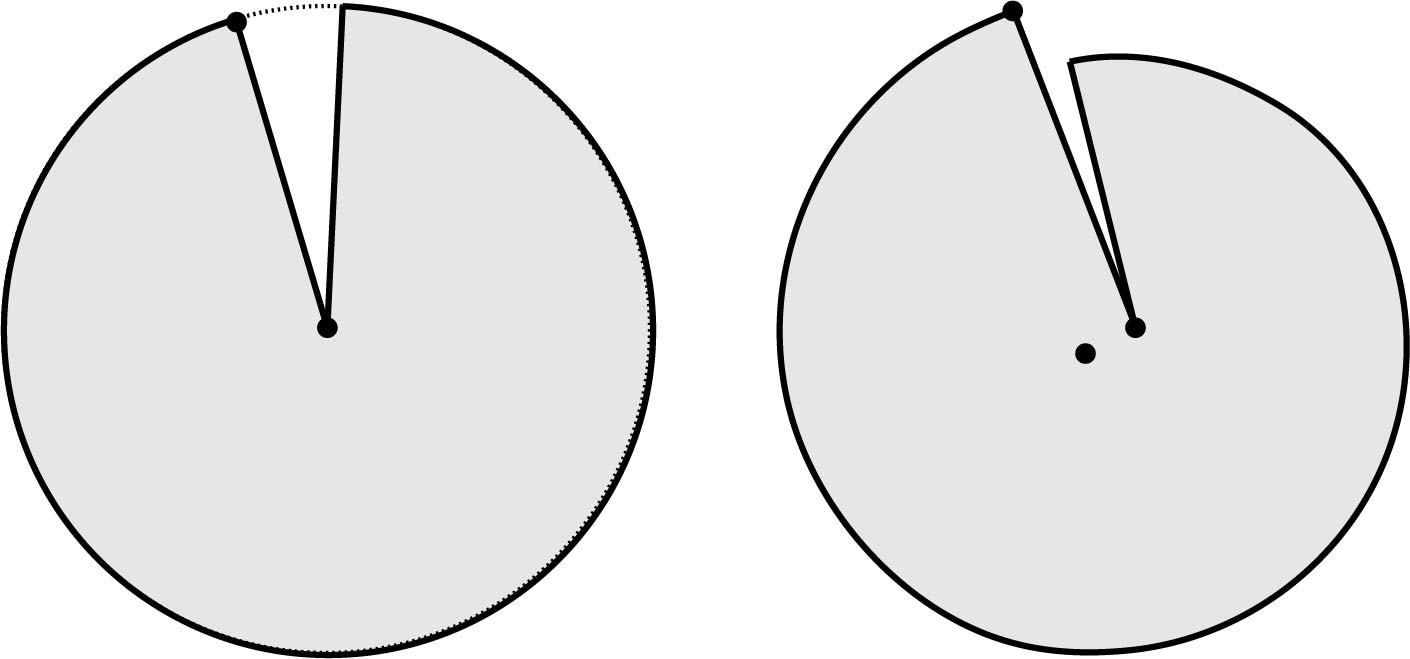} 
  \put(20 ,18 ){\textcolor{Black}{$O$}} 
    \put(10 ,30 ){\textcolor{Black}{$f(e_1)$}} 
    \put(38 , 10 ){\textcolor{Black}{$r$}}  
      \put(25 ,30 ){\textcolor{Black}{$f(e_2)$}}  
        \put(20 , 3 ){\textcolor{Black}{$f(e_3)$}}  
                \put(19 , 43 ){\textcolor{Black}{$q$}}  
  \put(66 ,28 ){\textcolor{Black}{$f'(e_1')$}} 
  \put(80 , 32 ){\textcolor{Black}{$f'(e_2')$}} 
    \put(92 ,10 ){\textcolor{Black}{$r'$}}  
                    \put(74 , 45 ){\textcolor{Black}{$q$}}  
      \put( 80 , 20 ){\textcolor{Black}{$z$}} 
      \end{overpic}
\caption{Perturbing a fundamental membrane of a cusp with elliptic holonomy.}\label{fFundamentalDomainTrivialHolonomy}
\end{figure}

\begin{figure}
\begin{overpic}[scale=.4, 
] {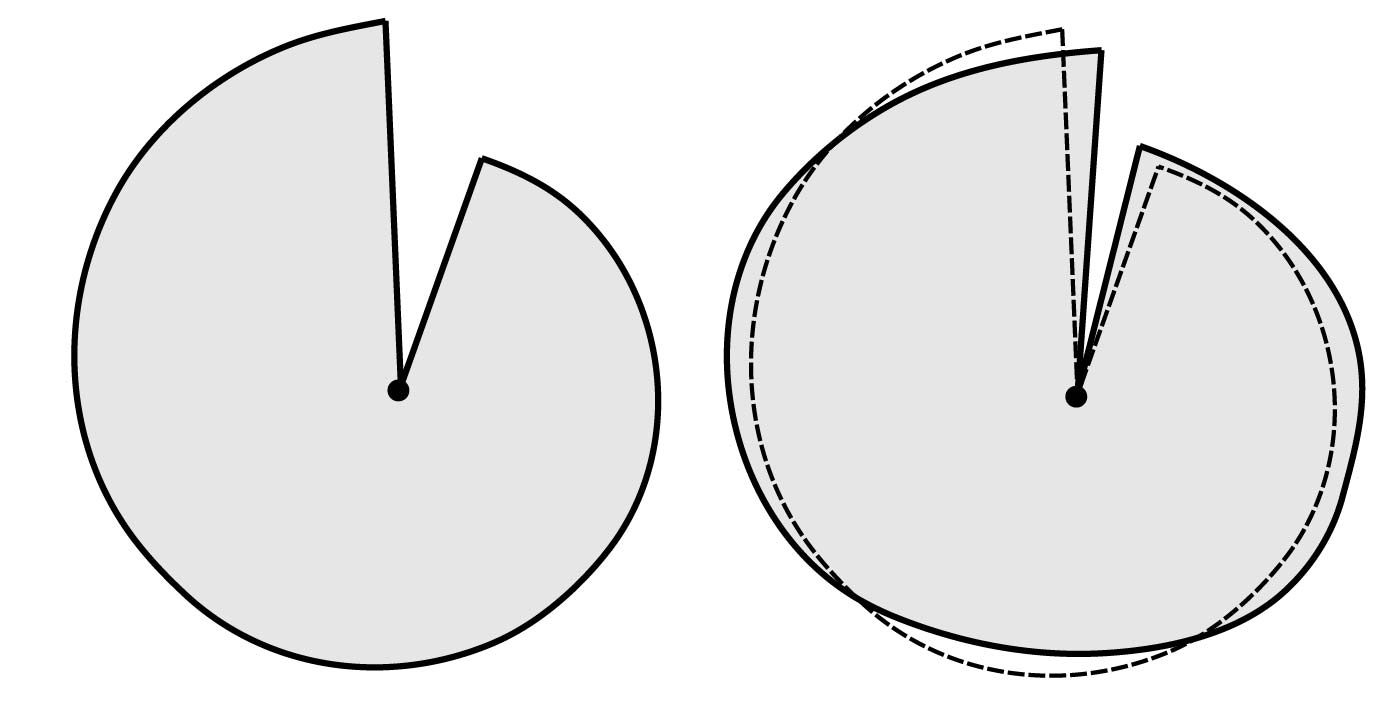} 
   \put(14 , 30){\textcolor{Black}{$f(e_1)$}}  
 \put(32 , 27){\textcolor{Black}{$f(e_2)$}}  
 \put(21 , 7){\textcolor{Black}{$f(e_3)$}}  
 \put(26 , 17){\textcolor{Black}{$0$}}  
      \end{overpic}
\caption{Perturbing a fundamental membrane of a cusp with a hyperbolic holonomy.}\Label{fPerburgingGenericHyperbolicCusp}
\end{figure}

\begin{figure}
\begin{overpic}[scale=.4
] {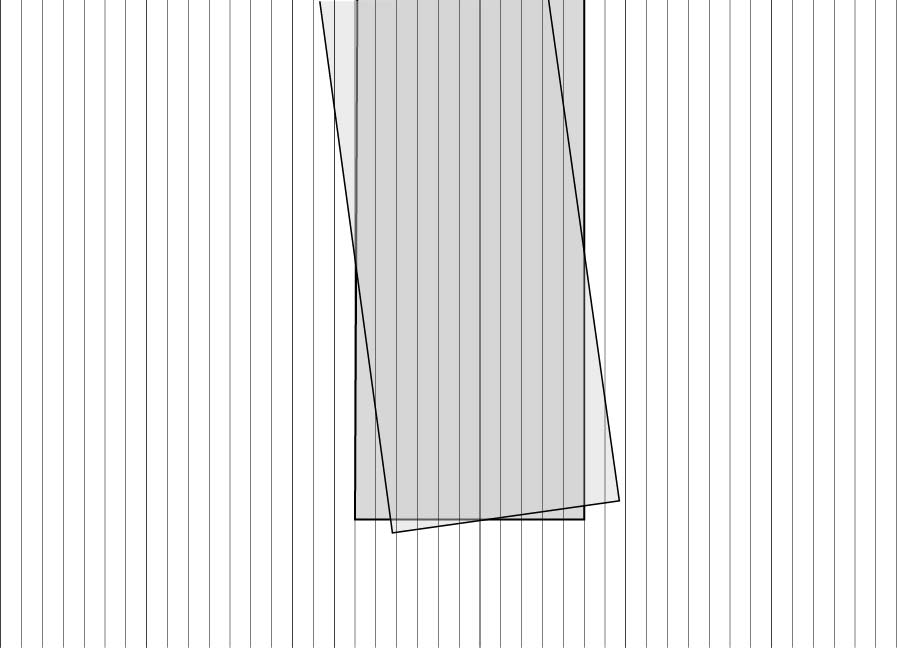} 
      \end{overpic}
\caption{Deformation of a hyperbolic cusp neighborhood. }\Label{fHyperbolicCuspNbhdNoTorsion}
\end{figure}
{\it Special hyperbolic case ($\sharp \Lambda = 2$).}
Suppose that $\gam$ is hyperbolic and $L(\ell) = 0$ (in particular $\tr \gam \in \R$). 
Then the boundary component $b$ of $\tau$ is a leaf of $L$ with weight infinity (\cite[Proposition 8.3]{Baba-17}).

Let $\kap \col C \to \tau$ be the collapsing map. 
Then  $ \kap^{-1}(b) \eqqcolon F$ is a half-infinite cylinder.  
The developing map of $F$ is the restriction of $\exp\col \C \to \C^\ast$ to a half-space bounded by a horizontal line in $\C$. 
Then we identify the universal cover $\til{F}$ of $F$ with the half-space, so that $\gam$ acts as a horizontal translation $t_\gam$. 
Take a fundamental domain $Q$ in  $\til{F}$ such that $Q$ is a vertical half-infinite strip in $\C$ bounded by two vertical rays and one horizontal segment (\Cref{fHyperbolicCuspNbhdNoTorsion}).

If $W$ is a small neighborhood of $\gam$ in $\PSL_2\C$ consisting of hyperbolic elements, for every $\gam' \in  W$, there is a translation $t_{\gam'}$ of $\C$ (close to the horizontal translation $t_\gam$), such that $t_{\gam'}$  descends to $\gam'$ by the exponential map up to $\PSL_2\C$.
Therefore, there is a small deformation of $\Sigma$ realizing $(\gam', \Lambda')$, and  (\ref{iLocalLocalHomeo}) holds.

On the other hand, arbitrary deformations of the cusp neighborhood $F$ contain such a deformation of such a half-infinite strip fundamental domain on $D$. 
Moreover, if $U$ is sufficiently small, then if there are two $\ep$-small deformations of  $F$ with the same framed holonomy $(\gam', \Lambda')$, up so isotopy, the structures on $D^\ast$ coincide by the $\ep$-closeness to $F$. 
Thus the uniqueness holds (\ref{iUniquenessNearCusp}). 

{\it Parabolic case. }
Suppose that $\gam$ is parabolic. 
Then in Thurston parameters, the puncture corresponds to a cusp of the hyperbolic surface $\tau$, and the total weight of $L$ along the peripheral loop $\ell$ is a non-negative $2\pi$-multiple.
Then, similarly to the case that $\gam = I$, we can show the claim 
by finding a cusp neighborhood and a fundamental domain in its universal cover which is bounded by circular arcs. 
\Qed{PerturbingCusp}

\subsubsection{Holonomy maps of $\CP^1$-structures with cusps}

Let $F$ be a closed surface minus finitely many points $p_1, \dots, p_n$.
Recall that $\PPP(F)$ denotes the space of all developing pairs $(f, \rho)$ for $\CP^1$-structures on $F$. 
Let  $(f, \rho) \in \PPP(F)$. 
Then $(f, \rho)$ gives a $\CP^1$-structure on $F$, and we let $X$ be its induced complex structure on $F$. 
Identify the universal cover $\ti{X}$ of $X$ with $\H^2$; then for each $i = 1, \dots n$, pick a lift $\til{p}_i$ of $p_i$ to a point on the ideal boundary of  $\til{X}$.
Then, by Remark \ref{CuspParameter}, for every $(f, \rho) \in \PPP(F)$ and a puncture $p_i$, we have a corresponding element in $(\gamma_i, \Lambda_i)\in \reallywidehat{\PSL_2\C}$.
Thus, by the definition of the topology of $\reallywidehat{\PSL_2\C}$, we have a continuous mapping from $\hol\col \PPP(F) \to (\reallywidehat{\PSL_2\C})^n \times \RRR(F)$ taking $(f, \rho) \in \PPP(F)$ to $( (\gamma_i, \Lambda_i)_{i = 1}^n, \rho)$.
In fact, $\hol$ yields a holonomy theorem in our setting. 
  
\begin{theorem}\Label{HolImmersion}
Every $(f, \rho) \in \PPP(F)$ has a neighborhood $W$ such that 
$$\hol |W$$ is a {\it local homeomorphism} onto its image.
{Moreover, for any $(f, \rho) \in \PPP(F)$, if there is a path $\rho_t ~(t > 0)$ converging to $\rho$ in $\RRR(F)$ as $t \to \infi$, then there is a lift of $\rho_t$ to  a path in $\PPP(F)$ for $t \gg 0$ converging to $(f, \rho)$. 
}
\end{theorem}
    \begin{remark}\Label{rLocalHolonomyImage}
The image of $\hol(W)$ is contained in $$\{\, ( (\gamma_i, \Lambda_i)_{i = 1}^n, \rho) \mid  \rho \in  W, \rho(\alp_i) = \gam_i\, (i = 1, 2, \dots, n) \,\}.$$
Furthermore, its subset cut by the condition on the framing given by  \Cref{PerturbingCusp} (\ref{iLocalLocalHomeo}) determines the local image $\hol(W)$.
\end{remark}

\proof
Let $(f, \rho) \in \PPP(F)$, and 
let $C$ be the $\CP^1$-structure on $F$ given by the developing pair $(f, \rho)$. 
Applying Proposition \ref{PerturbingCusp} to a small $\ep > 0$, we obtain, for each $i = 1, \dots, n,$ a (small)  cusp neighborhood $C_i$ of the puncture $p_i$ of $C$, and a neighborhood $U_i$ of $(\gamma_i, \Lambda_i)$  in $\reallywidehat{\PSL_2\C}$ modeling the deformation of $C_i$.
Let $N_i$ be the underlying topological cusp neighborhood of the punctured surface $F$ supporting $C_i$.
Without loss of generality, we can assume $C_1, \dots, C_n$ are disjoint in $C$. 
Let $C_i'$ be an open cusp neighborhood of $p_i$ smaller than $C_i$ and $U_i$ be a subset of $\reallywidehat{\PSL_2\C}$ containing $\hol((f, \rho))$ given by \Cref{PerturbingCusp}(\ref{iUniquenessNearCusp}), such that the small deformation of $C_i$ on $C_i'$ is parametrized the framed holonomy in $U_i$.

Let $N_i''$ be a (even smaller) cusp neighborhood of $p_i$ whose closure is contained in the interior of $N_i'$.
 Let $\ch{F}$ be $F \minus \sqcup_i N_i''$, and let $\check{C}$ be the restriction of $C$ to $\ch{F}$.
For every $(\gamma_i', \Lambda_i') \in U_i$, let $C_i(\gamma_i', \Lambda_i')$ denote the unique $\CP^1$-structure on $N_i'$ with the framed holonomy $(\gamma_i', \Lambda_i') \in U_i$ such that   $C_i(\gamma_i', \Lambda_i')$  is sufficiently close to $C_i$. 

We shall regard $(f, \rho)$ as a smooth section $\Sigma$ of  a $\CP^1$-bundle  $B$ over $F$ such that $\Sigma$ is transversal to the horizontal foliation $H_\rho$ associated with $\rho$ (see for example \cite{Goldman22GeometricStructuresOnManifolds})
Let  $\ch{\Sigma}$ be the restriction of $\Sigma$ to the bundle over the subsurface $\ch{F}$.  
Then, there is a neighborhood $U$ of $\rho$ in the representation variety $\RRR(F)$ such that, for each $\xi \in U$, letting $H_\xi$ be the horizontal foliation of $B$ associated with $\xi$,
 $\ch{\Sigma}$ is  still transversal to $H_\xi$ by the openness of transversality; then  $\ch{\Sigma}$ yields a projective structure $\ch{C}_\xi$ on $\ch{F}$ with holonomy $\xi$. 
In this way, we obtain a unique $\CP^1$-structure on $\ch{F}$ close to $(f, \rho)$ on $\ch{F}$. 
This new structure is unique in a compact subset of $\ch{F}$ whose interior contains the closure of $F \minus \sqcup_{i = 1}^n N_i'$.

For each $i$, pick any $\Lambda_i$ in $\Fix \xi_i(\gam_i) \in \CP^1$ so that $(\xi_i(\gam_i), \Lambda_i) \in U_i$.
Then $C_i(\xi_i(\gam_i), \Lambda_i)$ is its associated deformation. 
Then we can glue $\ch{C}_\xi$ and $C_i(\xi_i(\gam_i), \Lambda_i)$ in the overlapping region, and obtain a desired developing pair for a $\CP^1$-structure on $F$. 
Consider the subset $W$ in $\Pi_{i = 1}^n U_i \times \RRR(F)$ consisting $(\gamma_i, \Lambda_i)_{i = 1}^n, \rho)$ satisfying $\rho(\alp_i) = \gam_i\, (i = 1, 2, \dots, n)$; clearly $W$ contains $\hol(f, \rho)$.
In this way, given a sufficiently small neighborhood of $\hol((f, \rho))$ in this subset $W$, for every element in this neighborhood, we construct a developing pair realizing it.
This new $\CP^1$-structure on $F$ is unique  by the uniqueness of the thick part $\ch{C}_\xi$ on $F \minus \sqcup{N'_i}$ and the uniqueness of the cusp neighborhoods $C_i(\xi_i(\gam_i), \Lambda_i)$ on $N_i'$.

Notice that $W$ projects to a neighborhood of $\rho$ in $\RRR$.
The path lifting along a path in $\RRR$ easily follows from the construction as $U$ is a neighborhood of $\rho$ in $\RRR(F)$. 
\Qed{HolImmersion}

\section{Bound on the upper injectivity radius}\Label{sInjetivityRadius}
Recall that $C_t = (f_t, \rho_t)$ is a path of $\CP^1$ structures on $S$ such that $C_t$ diverges to $\infty$ and the equivalence class $[\rho_t] \eqqcolon \eta_t$ converges in the character variety as $t \to \infty$. 
Recall also that $C_t = (X_t, q_t)$ is the expression in the Schwarzian parameters.
  
Let $E_t$ be the singular Euclidean structure on $X_t$  given by $ |q_t^{\frac{1}{2}}|$.
Let $R(E_t) \geq 0$ denote that the {\it upper injectivity radius} of $E_t$. 
In this section we show
\begin{theorem}\Label{UpperInjectivityRadiusBound}
Suppose that $X_t$ is pinched along a multiloop $M$. 
Then the upper injectivity radius $R(E_t)$ of $E_t$ is bounded from above for all $t \geq 0$. 
\end{theorem}

Immediately we have the following. 
\begin{corollary}\Label{NoLargeExpandingEnds}
There is an upper bound for the area of the expanding cylinders in $E_t$ for all  $t \geq 0$. 
\end{corollary}

The rest of this section is a proof of  Theorem \ref{UpperInjectivityRadiusBound}.
  We suppose, to the contrary, that $\limsup R(E_t) = \infty$ and show that $\rho_t$ cannot converge.
Let $M_t$ be a geodesic representative of $M$ on $E_t$ (in the Euclidean metric) such that, for every $\ep > 0$ if $t > 0$ is sufficiently large, then $M_t$ is contained in the $\ep$-thin part of $X_t$.
We will find  a conformally thick part which is, in the Euclidean metric, bigger than its adjacent thick parts:

\begin{lemma}\Label{HolonomyAlongLoop}
Suppose that there is  a diverging sequence $(0 <)\, t_1 < t_2 < \dots$ such that $E_{t_i}$ contains a flat cylinder $A_{t_i}$  homotopy equivalent to a fixed loop $m$ of $M$
such that 
\begin{enumerate}
\item $\Mod{A_{t_i}} \to \infi$ as $i \to \infi$ \Label{iModA}, and
\item the circumference of $A_{t_i}$ limits to $\infty$ (equivalently $\Area A_{t_i} \to \infi$) as $i \to \infi$.  \Label{iDivergingArea}
\end{enumerate}
Then,
 leaves of the vertical foliation $V_{t_i}$ must be asymptotically orthogonal to the circumferences of $A_{t_i}$. 
\end{lemma}
\begin{proof}
  Suppose, to the contrary, that $V_{t_i}$ is {\it not} asymptotically orthogonal to circumferences. 
  Then, up to a subsequence, we may assume that there is a limiting angle $\theta_\infi \in [0,  \pi/2)$ between the angle between $V_{t_i}$ and the circumferences of $A_{t_i}$. 
Let $m_{t_i}$ be a geodesic representative of $m$  which sits in the middle of  $A_{t_i}$. 
Since $\theta_\infty \neq \pi/2$, Hypotheses  (\ref{iModA}) and (\ref{iDivergingArea}) imply that the transversal measure of the horizontal foliation $H_{t_i}$ along  $m_{t_i}$  diverges to infinity as $i \to \infi$. 
By \Cref{AlmostExponentialMap}, the translation length of $\rho_{t_i} (m_{t_i})$ is asymptotically $\sqrt{2}$ times the transversal measure. 
Therefore, the translation length of $\rho_{t_i} (m_{t_i})$ must diverge to infinity, which contradicts the convergence of $[\rho_t]$ as $t \to \infi$. 
\end{proof}

\begin{proposition}\Label{BigCompnenetShrinkingEnds}
Suppose that there are a component $F$ of $S \minus M$ and  a diverging sequence $(0 <) ~  t_1 < t_2 < \dots$ such that, letting $F_{t_i}$ be the component of  $E_{t_i} \minus M_{t_i}$ homotopic to $F$ on $S$,  
\begin{itemize}
\item $\Area_{E_{t_i}} F_{t_i} \to \infi$ as $i \to \infi$, and.
\item   for each boundary component $\ell$ of $F$, there is an expanding cylinder $B_{\ell, t_i}$ in $F_{t_i}$ bounded by the boundary component $\ell_i$ of $F_{t_i}$ homotopic to $\ell$ on $S$ such that
\begin{itemize}
\item $B_{\ell, t_i}$ shrinks toward $\ell_i$, i.e. $\ell_i$ is the shorter boundary component of $B_{\ell, t_i}$, and 
\item $\Mod B_{\ell, t_i} \to \infi$ as $i \to \infi$.
\end{itemize}
\end{itemize}
Then $[\rho_{t_i}] | \pi_1 F$ diverges to  $\infty$ in $\rchi$ as $i \to \infi$. 
\end{proposition}
\begin{proof}
Let $k_i > 0$  be such that $k_i \Area(F_{t_i}) = 1$ for each $i = 1, 2, \dots$.
 Then, as $\Area F_{t_i} \to \infi$, thus $k_{t_i} \to 0$ as $i \to \infi$.
All ends of $F_{t_i}$ have conformally long expanding cylinders shrinking towards adjacent components.
Take a base point in the thick part of $F_{t_i}$. 
Let $\hat{F}$ denote the compact surface with finitely many punctures, obtained by pinching the boundary loops of $F$ to puncture points.
Then the space of all holomorphic quadratic differentials on Riemann surfaces structures on $\hat{F}$ with Euclidean area one is a sphere of finite dimension. 
Then, by compactness,  up to a subsequence
\begin{itemize}
\item $k_i E_{t_i}$ converges, in the Gromov-Hausdorff topology,  to a compact singular Euclidean surface minus finitely many points, $E_\infi$,  which is homeomorphic to $F$, and
 \item the restriction of $k_i V_{t_i}$ to $k_i E_{t_i}$ converges to a measured foliation $V_\infi$ on $E_\infi$.
\end{itemize}
Take a piecewise geodesic loop $\ell$ on $E_\infi$ such that
\begin{enumerate}
\item  $\ell$ does not cross any singular point of $E_\infi$, 
\item each segment of $\ell$ is either vertical or horizontal,  and $\ell$ contains at least one vertical segment, and 
\item $\ell$ is a geodesic in the $L^\infty$-metric, so that at adjacent singular points, $\ell$ bends in the different direction by an angle $\pi/2$. 
\end{enumerate}

In fact, if $V_\infty$ contains a periodic leaf, then take it as $\ell$, which obviously satisfies the conditions.
Otherwise, $V_\infty$ contains a minimal irrational subfoliation, using the density of each leaf in the subfoliation,  a standard closing lemma gives a desired loop $\ell$ as in  \Cref{fClosingLemma} (see \cite[I.4.2.15]{CanaryEpsteinGreen84}).
 \begin{figure}
\begin{overpic}[scale=.15
] {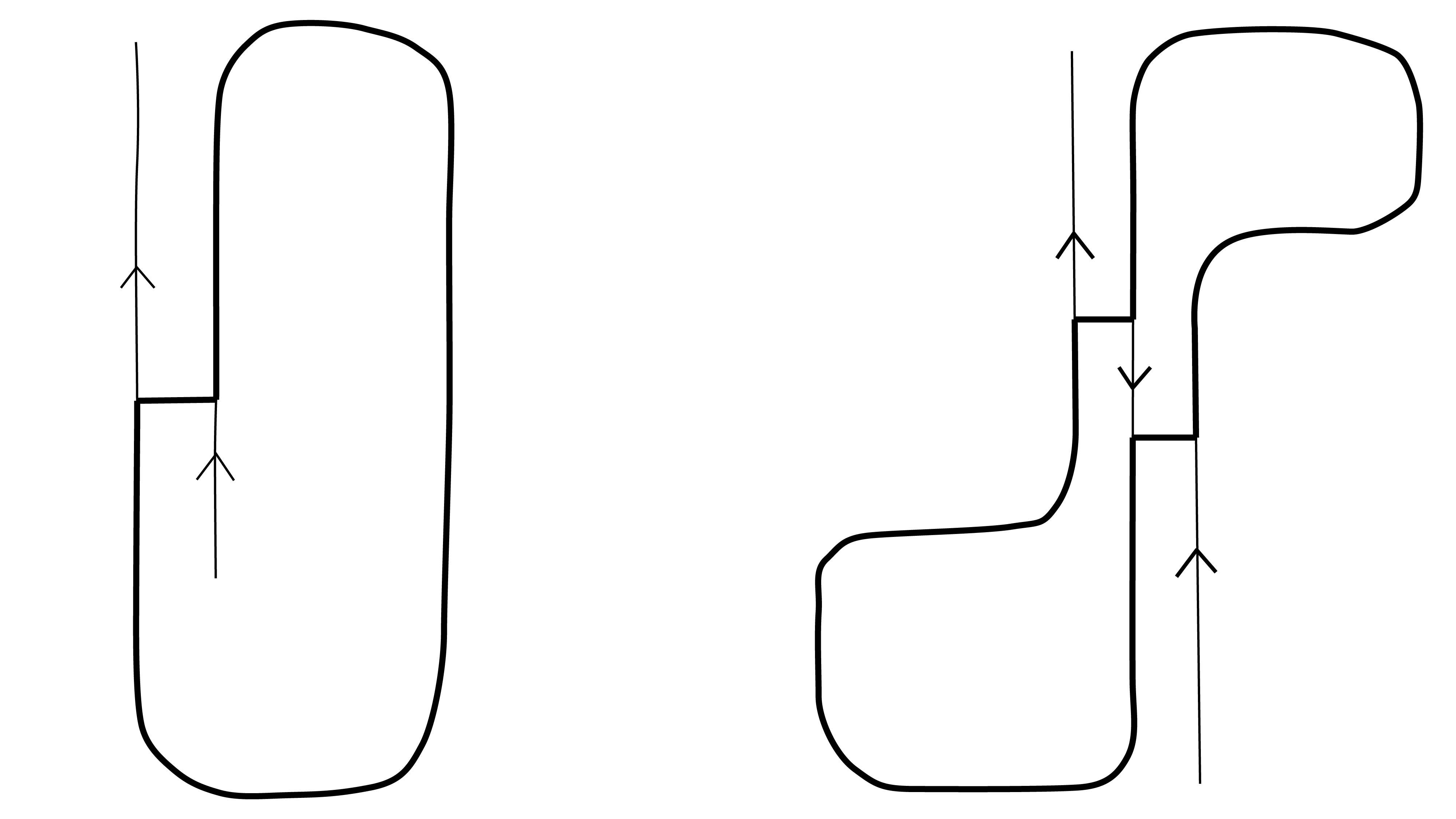} 
      \end{overpic}
\caption{Staircase closed loops $\ell$ consisting of long vertical segments and short horizontal segments }
\Label{fClosingLemma}
\end{figure}
      By the convergence $k_i E_i \to E_\infty$,
 for $i$ large enough, we pick a piecewise geodesic loop $\ell_i$ on $E_i$ satisfying the properties (1), (2), (3) such that $\ell_i$ has the same number of horizontal and vertical segments as $\ell$ has,  
  and $k_i \ell_i$ on $k_i E_i$ converges to $\ell$ on $E_\infty$ smoothly on each segment as $i \to \infty$. 
  Then the distance from $\ell_i$ to the singular set of $E_i$ goes to $\infty$ as $k_i \to 0$. 
 Therefore, by  \Cref{AlmostExponentialMap},  $\rho_{t_i}(\ell)$  is a hyperbolic element of translation length close to $\sqrt{2}$ times the total length of the vertical segment of $\ell_i$. 
Then, as $k_i \to 0$,  the total vertical length of $\ell_i$ on $E_i$ goes to infinity, and therefore $\tr \rho_i(\ell)$ must diverge to infinity. 
\end{proof}

Let $m_1, \dots, m_p$ be the loops of the multiloop $M$. 
\begin{proposition}\Label{ExpandingCydinderWithMaximalSystole} 
For every (large) $T > 0$, there are $t >  T$ and $k \in \{1, \dots, p\}$ such that
\begin{equation*}
\frac{1}{2}< \frac{\length_{E_t} (m_k)}{\max_{i = 1, \dots, p} \length_{E_t} (m_i)} \leq 2,
\end{equation*}
and 
 $\Mod_{E_t}(m_k)$ is $\frac{1}{3}$-dominated by an expanding cylinder $B_{k, t}$ homotopic to $m_k$, i.e. 

\begin{equation*}
\frac{\Mod B_{k, t}}{\Mod_{E_t} m_k}  > \frac{1}{3}.
\end{equation*} 
\end{proposition}
\proof
For  $u > T$, let $m_{k_u}$ be the loop realizing $\max_{i = 1, \dots, q} \length_{E_u} (m_i)$.
We may assume that $\max_{i = 1, \dots, p} \length_{E_t} (m_i) \to \infty$ as $t \to \infty$: in fact, otherwise,  since $\limsup R(E_t) = \infty$, Proposition \ref{BigCompnenetShrinkingEnds} implies that $[\rho_t]$ diverges in $\rchi$.

We first show that if a long flat cylinder persists, then its circumference must stay almost the same.
Namely
\begin{claim}\Label{PersistingFlatCylinder} 
For every $\ep > 0$, there is $K > 0$ such that, if  
there are  $w > u > K$ and  a flat cylinder in $E_t$ of height at least $K$ homotopic to $m$, then,  for every $t \in [u, w]$, 
 then  $$1 - \ep < \frac{\length_{E_t} m}{\length_{E_u} m} < 1 + \ep$$ for all $t \in [u, w].$ 
\end{claim}
 \begin{proof}
By \Cref{HolonomyAlongLoop}, 
for every $\ep > 0$, if $K > 0$ is sufficiently large, then the vertical foliation $V_t$ is $\ep$-almost orthogonal to circumferences of the flat cylinder homotopic to $m$.
Then, by Proposition \ref{CylinderHonomyEstimate}, for every $\ep > 0$, if $K > 0$ is sufficiently large, then the total rotation angle along $m$ is $(1 + \ep)$-bi-Lipschitz to $\sqrt{2}\length_{E_{t}} m$ for $t \in [u, w]$. 
As the holonomy of $\rho_t(m)$ converges as $t \to \infi$, for every $\ep > 0$, if $K$ is sufficiently large, then the total rotation along $m$ must be $\ep$-almost constant for all $t \in [u, w]$. 
Thus, if $K$ is sufficiently large, then the ratio of $\length_{E_t} m$  and $\length_{E_u} m$ is $\ep$-close to $1$.
\end{proof}

By Claim \ref{PersistingFlatCylinder}, for every $\ep > 0$, if $K > 0$ is sufficiently large, then,  if a flat cylinder $\frac{1}{4}$-dominates $\Mod m_{k_u}$ for all $t \in [u, w]$ for some $u > K$; then 
  $1 - \ep < \frac{\length_{E_t} m_{k_u}}{\length_{E_u} m_{k_u}} < 1 + \ep$ for all $t \in [u, w]$. 
Suppose, in addition,  that there is a loop $m_h$  of $M$ not $m_{k_u}$, such that $m_h$ on $E_t$ becomes exactly twice as long as $m_{k_u}$ on $E_u$ for the first time at $t = w < v$ after  $t = u$. 
 Then, by applying  Claim \ref{PersistingFlatCylinder} to $m_h$, we can show that there is $t \in [u, w)
 $ such that $\Mod_{E_t} {m_h}$ is $1/3$-dominated by an expanding cylinder:
Indeed,  otherwise, $\max_{i = 1, \dots, p} \length_{E_t} (m_i) $ must bounded from above by $\frac{3}{2} \length_{E_t}(m_{k_u})$ for all $t \in [u, w]$. 
 \Qed{ExpandingCydinderWithMaximalSystole}
 \begin{corollary}
There are a component $F$ of $S \minus M$ and a diverging sequence $0 < t_1 < t_2 < \dots$ such that the corresponding component $F_{t_i}$ of $E_{t_i} \minus M_{t_i}$ satisfies the assumptions of Proposition \ref{BigCompnenetShrinkingEnds}.
\end{corollary}
\begin{proof}
By Proposition \ref{ExpandingCydinderWithMaximalSystole},
there is a loop $m$ of $M$ and a diverging sequence $t_1 < t_2 < \dots$ such that 
\begin{itemize}
\item $\length_{E_{t_i}} m \to \infi$ as $i \to \infi$,
\item  $$\frac{1}{2} < \frac{\length_{E_{t_i}} m}{\max \{\length_{E_{t_i}} m_1, \dots, \length_{E_{t_i}} m_p\}} < 2$$ for all $i = 1, 2, \dots$, and
\item there is an expanding cylinder $B_{t_i}$ homotopic $m$ which $\frac{1}{3}$-dominates $\Mod_{E_{t_i}} m$.
\end{itemize}
Then, up to a subsequence, we may in addition assume that $B_{t_i}$ is expanding in the same direction. 
Then, let $F$ be the connected component of $S \minus M$ such that $m$ is a boundary component of $F$ and $B_{t_i}$ expands towards $F$.
As the size of $F_{t_i}$ becomes bigger and bigger than the length of $\length_{E_{t_i}} m$,  the first assumption of  \Cref{BigCompnenetShrinkingEnds} holds. 
Thus, by the second condition on the loop $m$ and the sequence $\{t_i\}$,  the second assumption of \Cref{BigCompnenetShrinkingEnds} is satisfied.
\end{proof}
By this corollary, we obtained a contradiction by \Cref{BigCompnenetShrinkingEnds} against the convergence of $\rho_t$. 
Hence we obtain Theorem \ref{UpperInjectivityRadiusBound}.
 
\section{Convergence of $\CP^1$-structures away from  pinched loops}\Label{sConvergenceOfThickPart}
We continue to suppose that $X_t$ is pinched along a multiloop. 
We will first see that the holonomy $\rho_\infi(m)$ determines the type of a conformally long Euclidean cylinder in $E_t$ which is homotopic to $m$ for $t \gg 0$. 
\begin{lemma}\Label{FlatOrExpanding}
\begin{enumerate}
\item Suppose that there are a sequence $t_1 < t_2 < \dots$ diverging to $\infty$ and a sequence of expanding cylinders $B_{t_i}$ in $E_{t_i}$ homotopic to  $m$  at time $t_i$, such that $\Mod_{E_{t_i}} B_{t_i} \to \infty$ as $t \to \infi$.\Label{iExpandingCylinderParabolic}
 Then $\rho_\infi(m)$ is parabolic.
\item Suppose that there is a sequence of flat cylinders $A_{t_i}$ in $E_{t_i}$ homotopic to a fixed loop $m$ on $S$ such that $\Mod{A_{t_i}}$ diverges to $\infty$ and the circumference of $A_{t_i}$ is bounded from below and above by positive numbers. 
 Let $w \in \C$ be such that the Möbius transformation $z \mapsto  (\exp w) z$ conjugates to $\rho_\infi(m)$.
 Then,  $\sqrt{2}\int_m \sqrt{q_{t_i} } $ converges
 to $w \mod 2\pi i$ up to a sign. \Label{iHolonomyOfFlactCylinder}
  \end{enumerate}
\end{lemma}

\begin{proof}
(\ref{iExpandingCylinderParabolic})
If a puncture of a $\CP^1$-structure corresponds to a regular point of its holomorphic quadratic differential, its peripheral holonomy is parabolic. 
Suppose that there are a sequence $t_1 < t_2 < \dots$ and an expanding cylinder $B_{t_i}$ in $E_{t_i}$ homotopic to $m$ such that $\Mod{B_{t_i}} \to \infty$ as $t \to \infi$.
 Then, by Corollary \ref{NoLargeExpandingEnds}, the length of the shorter boundary component of $B_{t_i}$ goes to zero as $i \to \infty$, and it asymptotically corresponds to, at most, a pole of order one of the quadratic differential. 
(A pole of order at least two corresponds to an infinite area end.)
Therefore $\rho_\infi(m)$ is parabolic, against the hypothesis.  

(\ref{iHolonomyOfFlactCylinder}) follows immediately from  Proposition \ref{CylinderHonomyEstimate}.
\end{proof}
      
   Given a compact surface $F$ with boundary, let $\hat{F}$ denote the surface with punctures obtained by pinching each boundary component of $F$ to a (puncture) point.
\begin{proposition}\Label{OneCylinderEnd}
Let $\ep > 0$ be a number less than the Bers constant. 
Let $F$ be a component of $S \minus M$, and let $F_t^\ep$ be the component of the conformally $\ep$-thick part of $E_t$ isotopic to $F$ for $t \gg 0$.
Then, if 
\begin{equation*}
\liminf_{t \to \infi} \Area_{E_{t}} (F_t^\ep) > 0, 
\end{equation*} 
there is a path of $\CP^1$-structures $\hat{F}_t$ on the punctured surface $\hat{F}$ such that 
\begin{enumerate}
\item 
for every $\ep > 0$, if $t > 0$ is sufficiently large, then $F_t^\ep$ isomorphically embeds into $\hat{F}_t$,  \Label{iCompletingFt}
\item \Label{iEndCylinder}
for each boundary component $\ell$ of $F$,  there is a  cylinder $A_{\ell, t}$ in $E_t$ homotopic to $\ell$ such that
\begin{itemize}
 \item $\Mod A_{\ell, t} \to \infi$ as $ t \to \infi$;
 \item $A_{\ell, t}$ is either a flat cylinder for all $t \gg 0$ or an expanding cylinder shrinking towards the adjacent component of $S \minus M$ across $m$ for all  $t \gg 0$;
\end{itemize}
\item $\hat{F}_t$ contains $A_{\ell, t}$ for every boundary component $\ell$ of $F$.  \Label{iCylindricalEnds}
\end{enumerate}

\end{proposition}
\proof
We first show that,
for each boundary component $\ell$ of $F$, there is a cylinder $A_\ell$ homotopic to $\ell$, such that 
\begin{enumerate}[(i)]
\item $\Mod A_{\ell, t} \to \infi$ as $t \to \infi$, \Label{iDiveringMod} and
\item $A_{\ell, t}$ remains either a flat cylinder for all sufficiently large $t > 0$ or an expanding cylinder shrinking forwards $\ell$ for all sufficiently large $t > 0$, \Label{iFaltOrExpanding}
\end{enumerate}

Let $Y_t, Z_t, W_t$ be disjoint cylinders homotopic to $\ell$,  such that $Z_t$ is a maximal flat cylinder,  $Y_t$ is the maximal expanding cylinder expanding towards the thicker part of $F_t$ and $W_t$ is the maximal expanding cylinder expanding towards the adjacent component across the geodesic representative $\ell_t$ of $\ell$.

As $X_t$ is pinched along $M$, by \Cref{Minsky},
 $\max\{  \Mod Y_t,  \Mod Z_t,  \Mod W_t\} \to \infi$ as $t \to \infi$.
Let $\diam W_t$ and $\diam Y_t$ denote the diameters of $W_t$ and $Y_t$, respectively,  in the Euclidean metric $E_t$. 
Then, by $\liminf_{t \to \infi} \Area_{E_{t}} (F_t^\ep) > 0$ and the upper injectivity radius bound (Theorem \ref{UpperInjectivityRadiusBound}), the ratio  $\frac{diam W_t}{diam Y_t + 1}$ is bounded from above for all $t > 0$.
Thus  $\frac{\Mod W_t}{\Mod Y_t  + 1}$ is bounded from above for all $t > 0$. 
Therefore $\Mod Y_t +   \Mod Z_t$ diverges to $\infi$ as $t \to \infi$. 
We claim, moreover, that either $\lim \Mod Y_t = \infi$ or $\lim \Mod Z_t = \infi$ holds.

\begin{lemma}\Label{ModYDiverges}
Suppose that $\limsup_{t \to \infi}\Mod Y_t = \infi$. 
Then $\Mod Y_t \to \infty$ as $t \to \infty$. 
\end{lemma}
\begin{proof}
Let $t_1 < t_2 < \dots$ be a sequence with $\lim_{i \to \infi} \Mod Y_{t_i} = \infi$.
Then the circumference of $Z_{t_i}$ limits to zero, and  by Lemma \ref{FlatOrExpanding} (\ref{iExpandingCylinderParabolic}), $\rho_\infi(\ell)$ is parabolic. 

Suppose to the contrary that there is a sequence $s_1 < s_2 < \dots$ diverges to $\infty$ such that 
$\Mod Y_{s_i}$ is bounded from above by some constant $b > 0$.
Then $\Mod Z_{s_i} \to \infty$, and the circumference of $Z_{s_i}$ is bounded from below $c > 0$.
On the other hand, since $\Mod Y_{t_i} \to \infty$, 
the circumference of $Z_{t_i}$ goes to zero as $i  \to \infty$. 
We can assume that $s_1 < t_1 < s_2 < t_2 < \dots$ by taking subsequences of $s_i$ and $t_j$ if necessary. 

Therefore, for every $r \in (0, c)$, for every sufficiently large $i$, there is $u_i \in [s_i, t_i]$, such that the circumference of $Z_{u_i}$ is $r$.
Then, as $\Mod Y_{u_i}$ is bounded from above,  $\Mod Z_{u_i} \to \infty$ as $i \to \infty$.  

 Then, by \Cref{FlatOrExpanding} (\ref{iHolonomyOfFlactCylinder}), the limit holonomy of $\rho_{u_i}(m)$ is determined by the complex length of the circumference. 
For almost all $r \in (0,c)$, $\rho_{u_i}(m)$ is not parabolic. 
This contradicts the convergence of $\rho_t$ as $\rho_\infty(\ell)$ is parabolic.
\end{proof}

Then,  $Y_t$ satisfies (\ref{iDiveringMod}) and  (\ref{iFaltOrExpanding}).

Next suppose that $\limsup_{t \to \infi}\Mod Y_t < \infi$. 
Then $\Mod Z_t$ diverges to $\infi$ as $t \to \infi$, and the circumference of $Z_t$ converges to a positive number. 
Then $Z_t$ satisfies   (\ref{iDiveringMod}) and  (\ref{iFaltOrExpanding}).

We shall construct  $\hat{F}_t$ satisfying  (\ref{iCylindricalEnds}) as follows. 
Suppose that, for a boundary component $\ell$ of $F$,  $\lim_{t \to \infi} \Mod Y_t = \infi$.
Let $\hat{Y}_t$  be an expanding cylinder of infinite modulus, obtained by extending the expanding cylinder $Y_t$ only in the shrinking direction, so that $\hat{Y}_t$ is conformally a punctured disk.  
Then replace $Y_t$ by $\hat{Y_t}$ in $E_t$ by gluing $E_t \minus Y_t$ and $\hat{Y}_t$ along the boundary component of $\hat{Y}_t$. 
Then the boundary component $\ell$ of $F$ corresponds to the puncture of $\hat{Y}_t$.  

Next suppose that $\limsup_{t \to \infi} \Mod Y_t < \infi$. 
 Then,  since $\Mod Z_t \to \infi$ and the circumference $\operatorname{Circ}(Z_t)$ converges to a positive number as $t \to \infi$, we extend the flat cylinder $Z_t$,  in the direction of $W_t$, 
 to  the half-infinite flat cylinder $\hat{Z_t}$; then $\hat{Z}_t$ is conformally a punctured disk.
 Then replace $Z_t$ in $E_t$ with $\hat{Z_t}$ so that it has a puncture corresponding to $\ell$. 

By applying,  such a replacement for all boundary component $\ell$ of $F$, we obtains a  desired complete singular Euclidean surface $\hat{F}_t$ satisfying (\ref{iCompletingFt}), (\ref{iEndCylinder}), (\ref{iCylindricalEnds}), as (\ref{iEndCylinder}) follows from (\ref{iDiveringMod}) and  (\ref{iFaltOrExpanding}).
\Qed{OneCylinderEnd}

 \begin{theorem}\Label{ConvergenceThickPart}
Let $F$ be a component of $S \minus M$. 
Let $\ep > 0$ be less than the Bers constant of $S$. 
For every $t > 0$ large enough, let $F_t^{\ep}$ be the component of the $\ep$-thick part of $C_t$ isotopic to $F$. 
 \begin{enumerate}
 \item Suppose that $$\liminf_{t \to \infi} \Area_{E_{t}} (F_t^{\ep}) = 0.$$ 
 Then, there is a continuous function $\ep_t > 0$ in $t$ with $\lim_{t \to \infi} \ep_t = 0$, such that  $F_t^{ \ep_t}$ converges (in the Gromov-Hausdorff topology) to a complete hyperbolic structure on a closed surface with finitely many punctures, denoted by $\hat{F}_\infi$, which is homeomorphic to $F$, as $t \to \infty$. \Label{iZeroAreaHyperbolic}
\item
Suppose that $$\liminf_{t \to \infi} \Area_{E_{t}} (F_t^{\ep}) > 0.$$ 
Then,  $\hat{F}_t$ accumulates to a bounded subset on the space of $\CP^1$-structures on $\hat{F}$.  \Label{iPositiveAreaThickPart}
 Moreover, if $\rho_\infi(m) \neq I$ for each boundary component $m$ of $F$, then $\hat{F}_t$ converges to a $\CP^1$-structure on $\hat{F}$ as $t \to \infty$. 
\end{enumerate}
\end{theorem}
\begin{remark}
In Case (\ref{iPositiveAreaThickPart}), similarly to (\ref{iZeroAreaHyperbolic}), one can take a  sequence $t_1 < t_2 < \dots$ diverging to $\infty$ so that $\hat{F}_t$ converges to a $\CP^1$-structure $\hat{F}_\infty$ on $\hat{F}$.
Then, for every $\ep > 0$ less than the Bers' constant,  the $\ep$-thick part $F_{t_i}^{\ep}$ converge to a subsurface of $\hat{F}_\infty$. 
If, in addition, the $\rho_\infty (m) \neq I$ for every boundary component of $F$, then $F_t^{\ep}$ converge to a subsurface of $\hat{F}_\infty$. 
  \end{remark}

\begin{proof}
(\ref{iZeroAreaHyperbolic})
Let $t_1 < t_2 < \dots$ be a diverging sequence such that $\Area (F_{t_i}) \to 0$ as $t \to \infi$.
Then the holomorphic quadratic differential on $F_{t_i}$ asymptotically vanishes. 
Thus, for every small $\ep > 0$,  $F_{t_i}^{ \ep}$ and $X_{t_i} | F_{t_i}^{ \ep}$ asymptotically identical, where $X_{t_i}$ is regarded as a hyperbolic surface by the uniformization theorem for each $i$. 
Here, by asymptotically identical, we mean that, for every $\upsilon > 0$ and every compact set $K$ in the universal cover $\H^2$ of   $X_{t_i}$, if $i$ is sufficiently large,   the developing map of $F_{t_i}^{ \ep}$ is $\upsilon$-close to the developing map of the hyperbolic structure $X_{t_i} | F_{t_i}^{ \ep}$ on $K$.
    
The holonomy representations of $F_{t_i}^{\ep_i}$ and $X_{t_i} | F_{t_i}^{ \ep}$ are asymptotically identical in the character variety. 
As the holonomy of $F_{t_i}^{\ep_i}$ converges in the representation variety, the holonomy of $X_{t_i} | F_{t_i}^{ \ep}$ must converge in the representation variety. 
Thus $X_{t_i} | F_{t_i}^{ \ep}$ converges to a complete hyperbolic structure $\sigma_\infty$ on $F$.
 Therefore $F_t^{\ep_t}$ must genuinely converge to $\sigma_\infi$ (without taking a subsequence). 
In particular $\Area_{E_t} F_t^{\ep_t} \to 0$ as $t \to \infi$.

(\ref{iPositiveAreaThickPart})
Suppose that $\liminf_{t \to \infi} \Area F_t^{ \ep}  > 0$ for sufficiently small $\ep > 0$.
Then let $\hat{F}_t$ denote the singular Euclidean structure on $\hat{F}$ obtained from $F_t$ by Proposition \ref{OneCylinderEnd}.
 Then $\hat{F}_t$ induces a $\CP^1$-structure on $\hat{F}$. 
 Let $(Y_t, w_t)$ be the Schwarzian parameterization of $\hat{F}_t$.
 Then, indeed,  every puncture of $Y_t$ is, at most, a pole of order two. 

 As $X_t$ is pinched along a multiloop $M$,  $Y_t$ is bounded in the Teichmüller space $\TT(\hat{F})$.
By \Cref{UpperInjectivityRadiusBound}, the upper injectivity radius of $\hat{F_t}$ is also bounded from above, and $(Y_t, w_t)$ is also bounded in the parameter space.
Thus, the $\CP^1$-structures $\hat{F}_t$ are contained in a compact subset of the deformation space of $\CP^1$-structures on $\hat{F}$.
Therefore $\hat{F}_t$ accumulates to a bounded subset in the deformation space of $\CP^1$-structures on  $\hat{F}$.

Moreover, if each peripheral loop has non-trivial holonomy at $t = \infi$, by \Cref{HolImmersion}, the convergence of the holonomy of $\hat{F}_t$ implies the convergence in $(\reallywidehat{\PSL_2\C})^n \times \RRR(F)$.
Therefore $\hat{F}_t$ has a unique limit in $\PPP(\hat{F})$. 
\end{proof}

 Theorem \ref{ConvergenceThickPart} immediately implies 
 \begin{corollary}\Label{DifferentialConvertesMultiloop}
Suppose that $X_t$ is pinched along a multiloop $M$. 
Then, for every sequence $t_1 < t_2 < \dots$ diverging to $\infty$, up to a subsequence, $X_{t_i}$ converges to a nodal Riemann surface $X_\infi$ and $q_{t_i}$ converges to a  regular quadratic differential on $X_\infi$.
\end{corollary}

\section{Degeneration by neck-pinching}\Label{sDegeneration}
In this section, we summarize our main theorems on asymptotic behavior under neck-pinching.

Let $C_t = (f_t, \rho_t), ~ t \geq 0$ be a path of $\CP^1$-structures which diverges to $\infty$ in the deformation space, such that its holonomy $[\rho_t]   \eqqcolon \eta_t$ converges in the character variety $\rchi$. 
By Proposition \ref{LiftingHolonomyPath}, we can assume that the holonomy $\rho_t \in \RRR$ also converges in the representation variety.
Let $X_t$ be the complex structure of $C_t$.

\begin{theorem}\Label{LimitHolonomyOfm}
Suppose that $X_t$ is pinched along a loop $m$.
Then  $\rho_\infi(m)$ is either $I$ or a parabolic element. 
Moreover $\rho_t(m) \neq I$ for large enough $t > 0$.
\end{theorem}
Recall that $\phi\col \ti{S} \to S$ is the universal covering map. 
Let $N_m$ be a regular neighborhood of $m$ in $S$. 
Regard the loop $m$ also as a fixed element of $\pi_1(S)$ representing $m$, and let $\ti{N}_m$ be the component of $\phi^{-1}(N_m)$ preserved by $m \in \pi_1(S)$.
\begin{theorem}[Convergence of developing maps]\Label{PinchedAlongLoop} 
Suppose that $X_t$ is pinched along a loop $m$. 
Then, exactly one of the following two holds.
\begin{enumerate}
\item 
    \begin{itemize}
    \item  $\rho_\infi(m)$ is parabolic;
    \item the cusp neighborhoods of $C_\infi$ are horodisk quotients;
    \item $f_t\col \ti{S} \to \CP^1$ converges a $\rho_\infi$-equivariant continuous map $f_\infi \col \til{S} \to \CP^1$ uniformly on compact subsets;
    \item there is a multiloop $M$ on $S$ consisting of finitely many parallel copies of $m$, such that 
$f_\infi$ is a local homeomorphism on $\ti{S} \minus \phi^{-1}(M)$ and it takes  each component of $\phi^{-1} (M)$ to its corresponding parabolic fixed point.
\end{itemize}
\item  $\rho_\infi(m) = I$, and,  for every sequence $t_1 < t_2 < \dots$ diverging to $\infi$, up to a subsequence,  there is a  path of markings $S \to C_t$ such that, as $i \to \infi$,
\begin{itemize}
\item $C_{t_i} | S \minus N_m $ converges to a $\CP^1$-structure on a surface with punctures homeomorphic to $S \minus m$; \Label{iThickPartConvergence}
\item the axis $a_i$ of $\rho_{t_i}(m)$ converges to a point on $\CP^1$ or a geodesic in $\H^3$;
\item the restriction of   $f_{t_i}$ to $\til{S} \minus \phi^{-1}(N_m)$ converges to a continuous map,  and each boundary component of $\ti{N}_m$ maps to an ideal point of  $\lim_{i \to \infi} a_i$.\Label{iBoundaryComponentAtEndpoints}
\end{itemize}

\end{enumerate}
\end{theorem}
 
For each $t \geq 0$, 
let $(\tau_t, L_t) \in \TT \times \ML$ be the Thurston parameterization of $C_t$, and let $\beta_t \col \H^2 \to \H^3$ be the $\rho_t$-equivariant pleated surface. 
In fact, $\beta_t$ converges a continuous map to $\H^3 \cup \CP^1$:
\begin{theorem}\Label{LimitOfPleatedSurface}
Suppose that $X_t$ is pinched along a loop $m$ on $S$. 
Let $N_m$ be a regular neighborhood of $m$ on $S$.
Then, by taking an appropriate path of markings $\iota_t\col S \to \tau_t$,  exactly one of the following two holds:
\begin{enumerate}
\item $\rho_\infi(m) \in \PSL_2\C$ is parabolic, and $\beta_t\col \til{S} \to \H^3$ converges to a $\rho_\infi$-equivariant continuous map  $\beta_\infi \col \til{S}  \to \H^3 \cup \CP^1$ uniformly on compact subsets as $t \to \infi$, such that  $\beta^{-1}_\infi(\CP^1)$ is a $\pi_1(S)$-invariant multicurve which is $\pi_1(S)$-equivariantly homotopic to the multicurve $\phi^{-1}(m)$.
\item $\rho_\infi(m) = I \in \PSL_2\C$, and for every diverging sequence $t_1 < t_2 < \dots$,  up to a subsequence, $\beta_{t_i}\col \til{S} \to \H^3$ converges to a $\rho_\infi$-equivariant continuous map $\beta_\infi \col \til{S} \to \H^3 \cup \CP^1$  as $i \to \infi$ and the axis $a_i$ of $\rho_{t_i}(m)$ converges to a point $\CP^1$  or a geodesic of $\H^3$ such that 
\begin{itemize}
\item if $\lim_{i \to \infi} a_i$ is a point on $\CP^1$, then $\beta_\infi^{-1}(\CP^1) = \phi^{-1}(m)$, and
\item if $\lim_{i \to \infi} a_i$  is a geodesic $a_\infi$ in $\H^3$, then $\beta_\infi$ takes each component of $\phi^{-1}(N_m)$ to its corresponding limit geodesic $a_\infi$ and each component of $\til{S} \minus \phi^{-1}(N_m)$  to either a pleated surface in $\H^3$ or a single point on $\CP^1$.
\end{itemize}
\end{enumerate}
\end{theorem}

In order to prove Theorem \ref{LimitHolonomyOfm}, Theorem \ref{PinchedAlongLoop} and Theorem \ref{LimitOfPleatedSurface},
 we carefully observe the behavior of $C_t$, fixing the isometry type of $\rho_\infi(m)$.
In particular, for Theorem \ref{LimitHolonomyOfm},  we will show that, supposing, to the contrary, that $\rho_\infi(m)$ is hyperbolic  (\S\ref{sHyperbolicNeck}) or elliptic   (\S\ref{sEllipticNeck}), then $\rho_t$ cannot converge. 
The convergence when  $\rho_\infi(m) = I$ is given in \S\ref{mTrivialHolonomy} and the convergence when  $\rho_\infi(m)$ is parabolic is given in  \S \ref{sParabolicNeck}. 
 
 \section{$\CP^1$-structures on punctured surfaces with elementary holonomy}
\begin{lemma}\Label{PunctureAtDomain}
Let $F$ be a closed surface with finitely many punctures, such that the Euler characteristic of $F$ is negative.  
Let $C = (f, \rho)$ be a $\CP^1$-structure on $F$ such that
\begin{itemize}
\item  $\rho$ is an elementary representation, and 
\item for each puncture of $C$, its peripheral holonomy is non-hyperbolic (so that its developing image is a single point on $\CP^1$). 
\end{itemize}
 Let $\Lambda$ be the subset in $\CP^1$ of cardinality  $0, 1,$ or $2$ which $\Im \rho$ preserves as a set. 
Then,  there is at least one puncture of $C$ which maps to a point in the complement $\CP^1 \minus \Lambda \eqqcolon \Omega$ by $f$.
\end{lemma}
\begin{proof}
The discrete subset $f^{-1}(\Lambda)$ in $\ti{F}$ descends a finite subset $D$ on $F$. 

We can assume that $\Lambda$ is a non-empty set, since 
if $\Lambda$ is the empty set, then the assertion is obvious. 
First, suppose that the cardinality of $\Lambda$ is two, then $\Omega$ admits a complete Euclidean metric invariant under $\Im \rho$. 
Then, if all cups of $F$ map to $\Lambda$, $F \minus D$ admits a complete Euclidean metric, which is a contradiction against the Euler characteristic of $F$. 

Next, suppose that the cardinality of $\Lambda$ is one. 
Suppose, to the contrary,  that all cups of $C$ map to the point $\Lambda$.
Then $C \minus D$ has a complex affine structure. 

We claim that $C \minus D$ is {\it complete}, i.e. the developing map of $C \minus D$ is a diffeomorphism onto $\C$, when we normalize $\dev C$ so that $\{\infty \}$ corresponds to the punctures.
Suppose, to the contrary,  that $C \minus D$ is incomplete.
As the cardinality of $\Lambda$ is not two,  $\Im \rho$ does not preserve an incomplete point of $C \minus D$ in $\C$. 
Thus  $C$ admits Thurston's parametrization $(\tau, L)$ where $\tau$ is a finite area hyperbolic structure on $F$ and $L$ is a measured lamination on $\tau$ (Theorem \cite[Theorem 11.6]{Kulkani-Pinkall-94}, cf \cite[Theorem 3.1]{Baba-17}). 
Since $F$ is incomplete and the cardinality of $\Lambda$ is not two, there is a maximal ball $B$ of $\dev F$ such that its ideal point set contains two distinct points in $\C$. 
Then the holonomy of $F$ must contain a hyperbolic element in $\PSL_2\C$ whose fixed points are in $\C$, whose endpoints are close to those two points in $\C$. 
This leads to a contradiction to all cups mapping to the same point.  
Therefore, the $C \minus D$ is complete. 

Thus, the holonomy of $F$ consists of parabolic elements fixing $\infty$. 
Then the Euler characteristic of $F \minus D$ is zero, since $F\minus D$ admits Euclidean structure. 
 Therefore $F$ has a positive Euler characteristic, which is a contradiction. 
  \end{proof}
\begin{proposition}\Label{PunctureNotfixed}
Let $F$ be a closed surface with two punctures $p$ and $q$ such that the Euler characteristic of $F$ is negative. 
Suppose that $C = (f, \rho)$ is a $\CP^1$-structure on $F$ such that
\begin{itemize}
\item the holonomy of $C$ is  elementary,  and the stabilizer of $\Im \rho$ (in $\PSL_2\C$) is non-discrete, and
\item  the degrees of $f$ around the two punctures are the same. 
\end{itemize}
Then, no cusp of $F$ maps to the subset $\Lambda$ defined in \Cref{PunctureAtDomain}.
\end{proposition}
\begin{proof}
By Lemma \ref{PunctureAtDomain}, we can assume that  $p$ does not develop to  $\Lambda$. 
As the Euler characteristic of $F$ is negative, we let $C \cong (\tau, L)$ be the Thurston parameters of $C$\,; then by the assumption of the holonomy, $p$ and $q$ correspond to cusps of $\tau$. 
Then, as the degrees at  $p$ and $q$ agree, the total weights of leaves of $L$ around the punctures are the same. 

Suppose, to the contrary, that a puncture $q$ develops to a point of $\Lambda$.
Then $f$ takes all lifts of $q$ to the same point $r$ of $\Lambda$:
 Otherwise, as $\Lambda$ has cardinality two, $\Im \rho$ contains hyperbolic elements, and it also contains an elliptic element exchanging the points of $\Lambda$; then the stabilizer of $\Im \rho$ must be discrete against the hypothesis.
 
 Let $\ell$ be a leaf of $L$ initiating from $q$. 
 Then its lift $\ti\ell$ to the universal cover of $\tau$ maps, by the bending map $\beta \col \H^2 \to \H^3$,  to a geodesic in $\H^3$ initiating from $q$.  
 As all lifts of $p$ map to $r$, the other endpoint of $\beta(\ti\ell)$ is the image of a lift of $q$.
Therefore all leaves of $L$ initiating from  $q$ must end at $p$.
 For every complementary region $R$ of $\tau \minus L$, letting $\til{R}$ be the universal cover of $R$ (in $\til{\tau} = \H^2$), at most, one ideal point of $\ti{R}$ maps to $q$ by the pleated surface. 

Moreover, every leaf of $L$ initiating from $p$ must end at $q$, since the total weights of $L$ around $p$ and $q$ agree. 
Let $L_{p, q}$ be the sublamination of $L$ consisting of the isolated leaves of $L$ connecting $p$ and $q$. 
This implies that each component $\sigma$ of  $\tau \minus L_{p, q}$ has a negative Euler characteristic. 
Since no leaves of $L \minus L_{p, q}$ has an endpoint on the boundary of  $\tau \minus L_{p, q}$, 
 the restriction of $\rho$ to $\pi_1(\sigma)$ is non-elementary, which is a contradiction.  
\end{proof}   
\section{Parabolic limit}\Label{sParabolicNeck}
In this section, we assume that $\rho_\infi(m)$ is parabolic, and analyze the limit of $C_t$ as $t \to \infty$ in terms of its bending map and developing map. 
First, by Theorem \ref{ConvergenceThickPart}, for each component $F$ of $S \minus m$, by taking an appropriate base point $b_t$ in the thick part of $C_t$ homotopic to $F$,  $(C_t, b_t )$ converges to a $\CP^1$-structure $F_\infi$  on a compact surface with one or two punctures, such that $F_\infi$ is homeomorphic to $F$.
Let $C_\infi$ be the disjoint union of all such geometric limits $F_\infi$ over all thick parts. 
Then $C_\infi$ is a $\CP^1$-structure on a closed surface with two cusps homeomorphic to $S \minus m$. 
Note that $C_\infi$ is not connected if and only if $m$ is separating.
Then the limit holonomy has the following algebraic property.
\begin{lemma}\Label{ParabolicNonelementary}
Suppose that $\rho_\infi(m)$ is parabolic. 
Then, for each component $F$ of $S \minus m$, $\rho_\infi(F)$ is non-elementary.
\end{lemma}

\begin{proof}
Since $S$ is a closed oriented surface of genus at least two, each component of $S \minus m$ is also of hyperbolic type.
Thus let $(\sigma, \nu)$ be the Thurston parameterization of $F_\infi$, where $\sigma$ is a complete closed hyperbolic with one or two cusps homeomorphic to $F$ and $\nu$ is a measured lamination on $\sigma$. 
Clearly, the cusps of $F_\infi$ correspond to the cusps of $\sigma$.
Then there is a bi-infinite simple geodesic $\ell$ properly embedded in $\sigma$ such that $\ell$ is a leaf of $\nu$ or disjoint from $\nu$ (note that each endpoint of $\ell$ is at a cusp of $\sigma$).

Let $\beta \col \H^2 \to \H^3$ be the bending map given by $(\sigma, \nu)$, such that $\beta$ is equivariant via $\rho_\infi | \pi_1(F)$.
Let $\ti{\ell}$ be a lift of $\ell$ to the universal cover $\H^2$ of $\sigma$. 
Then the endpoints of $\ti\ell$ are parabolic fixed points in the ideal boundary of $\H^2$.  
Let $\gam_1, \gam_2 \in \pi_1(F)$ be the peripheral elements fixing the endpoints.
As $\ell$ does not cross $\nu$, its image $\beta(\ti\ell)$ is a geodesic in $\H^3$. 
Moreover, as $\beta$ is $\rho_\infty$-equivariant,  $\rho_\infty(\gam_1)$ and  $\rho_\infty(\gam_2)$ are parabolic elements fixing the different endpoints of  $\beta(\ti\ell)$. 
Therefore $\rho_\infty(\gam_1)$ and  $\rho_\infty(\gam_2)$ are non-commuting parabolic elements in $\PSL_2\C$, and they generate a non-elementary subgroup of $\PSL_2\C$.
\end{proof}

Proposition 
\ref{CuspClassification} implies that the developing map extends to cups with parabolic holonomy. 
\begin{proposition}\Label{ExtendingDevToPunctures}
Let $C = (f, \rho)$ be a $\CP^1$-structure on a closed surface with finitely many punctures, denoted by $F$, such that the holonomy around each puncture is parabolic. 
Then the developing map $f\col\til{F} \to \CP^1$ extends continuously to the lift of cups so that they map to their corresponding parabolic fixed points.  
\end{proposition}
\begin{proof}
Set $C \cong (\tau, L)$ in Thurston's parameters, where $\tau$ is a hyperbolic surface homeomorphic to $F$ and $L$ is a measured lamination on $\tau$.
For each cusp $c$ of $C$, 
by Proposition 
\ref{CuspClassification},  as the holonomy $\rho$ around $c$ is parabolic element in $\PSL_2\C$,  $c$ corresponds to a cusp of $\tau$ and the total weight of leaves of $L$ ending at the cusp is either $0$ or a positive multiple of $2\pi$.
Let  $\beta\col \H^2 \to \H^3$ be the bending map, and let $\ti{L}$ be the $\pi_1(F)$-invariant measured lamination on $\H^2$ by pulling back $L$ by the universal covering map $\H^2 \to \tau$. 
Let $r$ be a geodesic ray in the universal cover $\H^2$ ending at a parabolic fixed point $p$ of a peripheral element of $\pi_1(S)$. 
Then $r$ eventually does not cross the $\ti{L}$. 
Thus the curve $\beta(r)$ is eventually a geodesic ray in $\H^3$ ending at $p$. 
By the correspondence between the developing map and the pleated surface, the assertion follows. 
\end{proof}
Recall that $\phi\col \ti{S} \to S$ denotes the universal covering map.
Then the above lemmas imply a good convergence of the developing map of $C_t$ away from $m$.
\begin{theorem}\Label{DevelopingMapWithParabolicCusp}
Suppose $\rho_\infi(m)$ is parabolic.
Then there is a regular neighborhood $N$ of $m$ such that 
$f_t  | \ti{S} \minus \phi^{-1} (N)$ converges to a $\rho_\infi$-equivariant continuous map $f_\infi\col \ti{S} \minus \phi^{-1} (N) \to \CP^1$ uniformly on compact subsets, such that the developing image of each boundary component of $\ti{S} \minus \phi^{-1} (N)$ maps to its corresponding parabolic fixed point. 
\end{theorem}
\begin{proof}
By \Cref{ConvergenceThickPart} (\ref{iPositiveAreaThickPart}),
the restriction $C_t$ to $S \minus N$ converges to $C_\infty$ as $t \to \infty$ by taking an appropriate isotopy of $S$ uniformly.
Since $\rho_\infty(F)$ is non-elementary (\Cref{ParabolicNonelementary}), the restriction of $f_t$ to $\ti{S} \minus \phi^{-1}(N)$ converges to the developing map of $C_\infty$ uniformly on compact subsets. 
By  \Cref{ExtendingDevToPunctures}, each boundary component $\ti{S} \minus \phi^{-1} (N)$ converges to its corresponding parabolic fixed point uniformly on compact subsets. 
\end{proof}

In the rest of this section,  we show the convergence of the developing map of $C_t$ on the entire surface.  
First we analyze the holonomy of $C_t$ along $m$. 
\begin{proposition}\Label{AlmostParabolicCuspsAreHyperbolicOrElliptic}
For sufficiently large $t > 0$,
$\rho_t(m)$ is not the identity element of $\PSL_2\C$. 
Moreover, if the cusp neighborhoods of $C_\infty$ are horodisk quotients. 
Then, for sufficiently large $t > 0$,
$\rho_t(m)$ is hyperbolic.
\end{proposition}

\begin{proof}
Set $C_t \cong (\tau_t, L_t) \in \TT  \times \ML$ in Thurston's parameters for $t > 0$. 
Similarly set $C_\infty \cong (\tau_\infty, L_\infty)$, where $\tau_\infty$ is a complete hyperbolic structure on $F \minus m$ with finite volume, and $L_\infty$ is a measured geodesic lamination on $\tau_\infty$.

Let $m_t$ denote the geodesic representative of $m$ on $\tau_t$. 
Then, the length of $m_t$ on $\tau_t$ converges to $0$ as $t \to \infty$ since $\rho_\infty(m)$ is parabolic.

Suppose, to the contrary, that there is a sequence $t_1 < t_2 < \dots$ diverging to $\infty$ such that $\rho_{t_i}(m)$ is {\it not} hyperbolic. 
Then a leaf $\ell_i$ of $L_{t_i}$ intersects the geodesic loop $m_{t_i}$ for each $i = 1, 2 \dots$. 
Pick a point $p_i$ on $m_{t_i} \cap L_{t_i}$. 
Pick a lift $\ti{m}_t$ of $m_t$ to the universal cover $\ti\tau_i \cong \H^2$ which is preserved by an element $\gam_m$ in $\pi_1(S)$ whose free homotopy class is $m$. 
Then, for each $i$, let $p_{i, j} ~ (j \in \Z)$ be the lifts of $p_{t_i}$ on $\ti{m}_{t_i}$ in $\H^2$ indexed linearly, so that $p_{i, j} = \gam_m^j \cdot p_{i, 0}$.

For $t > 0$,  
let $\beta_t\col \H^2 \to \H^3$ be the $\rho_t$-equivariant bending map induced by $(\tau_t, L_t)$. 
Then, since $\{ p_{i, j} \}_{j \in \Z}$ is an orbit of the infinite cyclic group generated by $\gam_m$,  its image $\{\beta_{t_i}(p_{i, j})\}_{j \in \Z}$ is an orbit of the cyclic group generated by $\rho_{t_i} (\gam_m) \in \PSL_2\C$.
Then, since $\rho_{t_i}(m)$ is elliptic or parabolic (possibly the identity), by basic hyperbolic geometry, the points $\beta_{t_i}(p_{i, j})$  is contained in a totally geodesic hyperbolic plane $H_{t_i}$ in $\H^3$.
(In comparison, if $\rho_{t_i}(m)$ is hyperbolic and its screw rotation angle is not a multiple of $\pi$, then most of  its orbits do not lie in a totally geodesic plane.) 

Note that $H_{t_i}$ is uniquely determined by the choice of $p_i$ and the lift $\ti{m}_i$, unless $\rho_{t_i}(m)$ is the identity. 

If $\rho_{t_i}$ is the identity element in $\PSL_2\C$, then, letting $\ti\ell_i$ be the leaf of $\ti{L}_{t_i}$ intersecting $\ti{m}_i$ in $p_{i, j}$,  let $H_{t_i}$ be the hyperbolic plane orthogonal to the geodesic $\beta_{t_i}(\ti\ell_i)$ in the point $\beta_{t_i}(p_{i, j})$ for some $j \in \Z$.
Clearly $H_{t_i}$ is independent on the choice of $j \in \Z$, as $\rho_{t_i}(m)$ is the identity.

The infimum of $\angle_{\tau_{t_i}}(m_{t_i}, L_{t_i}) \geq 0$ over $i = 1, 2, \dots$ is positive, since $\angle_{\tau_i}(m_{t_i}, L_{t_i})$ is close to zero, then $\rho_{t_i}$ must be hyperbolic (\Cref{SmallIntersectionAngleImpliesHyperbolic}). 
Then, there is $\del > 0$, such that,  if $i$ is large enough, then, if a leaf $\ell$ of $\ti{L}_{t_i}$ intersects $\ti{m}_{t_i}$, then the angle between the geodesic $\beta_{t_i} (\ell)$ and  the hyperbolic plane $H_{t_i}$ is at least $\del$. 
Indeed,  otherwise, $\lim_{i \to \infty}\angle_{\tau_i}(m_{t_i}, L_{t_i}) = 0$. 

Recall that $\tau_\infty$ is a complete hyperbolic surface of finite volume homeomorphic to $S \minus m$, so that each boundary component of $S \minus m$ corresponds to a cusp of $\tau_\infty$. 
Pick a loop $\alpha$ on $S$ such that
\begin{enumerate}
\item $\alpha$ essentially intersects $m$ in a single point if $m$ is non-separating, and in two points if $m$ is separating, 
\item each segment $\alpha \minus m$ descends to a geodesic $g$ on $\tau_\infty$ with endpoints at cusps, and $g$ does not crossing $L_\infty$. \Label{iNotCrossingLamination}
\end{enumerate}

Below we show that the translation length of $\rho_{t_i}(\alpha)$ diverges to $\infty$, which contradicts the convergence of $\rho_t$. 
We assume that $m$ is non-separating, and one can similarly prove the case when $m$ is separating. 

For each $i = 1, 2 \dots$,  let $\alpha_i$ be the piecewise geodesic loop  on $\tau_{t_i}$  to homotopic to $\alpha$, such that 
\begin{itemize}
\item $\alpha_i$ is a union of two geodesic segments, 
\item one geodesic segment $s_i $ of $\alpha_i$ has its interior contained in $\tau_{t_i} \minus m_{t_i}$, and at each endpoint,  $s_i$ meets $m_{t_i}$ orthogonally, and
\item the other geodesic segment $u_i$ contained in $m_{t_i}$. 
\end{itemize}

Since $\tau_{t_i}$ is pinched along $m$ as $i \to \infty$,  the length of $s_i$ goes to $\infty$. 
Let $\ti\alpha_i$ be a lift of $\alpha_i$ to $\H^2$ which is a simple piecewise geodesic, and it is a bilipschitz curve. 

For each $i = 1,2, \dots$, let $\ti{u}_i$ be a lift of $u_i$ to a geodesic segment of $\ti\alpha_i$.
Then, let $\ti{m}_i$ be the lift of $m_{t_i}$ to $\H^2$ which contains $\ti{u}_i$, and let $\gam_{\ti{u}_i} \in \pi_1(S)$ be the element preserving $\ti{m}_{t_i}$.
For every $\ep > 0$ if $i$ is large, the $\beta_{t_i}(\ti{u}_i)$ is contained in the $\ep$-neighborhood  the $\rho_{t_i}(\gam_{\ti{u}_i})$-invariant hyperbolic plane $H_{\ti{u}_i}$ above, since $\length_{\tau_{t_i}}{m_{t_i}}$ goes to $0$. 

Let $\ti{s}_i$ be a lift of $s_i$ to a segment of $\ti\alpha_i$. 
Then, the length of $\ti{s}_i$ goes to $\infty$ as $i \to \infty$. 
For every $\ep > 0$, by (\ref{iNotCrossingLamination}), the transversal measure of $s_i$ by $L_{t_i}$ in the $\ep$-thick part of $\tau_{t_i}$ limits to $0$ as $i \to \infty$. 
In addition, there is $r > 0$, such that, the intersection angle of $L_{t_i}$ and $s_i$ in the $r$-thin part of $\tau_{\tau_{t_i}}$  goes to zero as $i \to \infty$.
Therefore, for every $\ep  > 0$, if $i$ is sufficiently large, then the restriction of $\beta_{t_i}$ to $\ti{s}_i$ is a $(1- \ep, 1 + \ep)$-bilipschitz embedding. 
Let $g_i$ be the bi-infinite geodesic in $\H^3$ passing through the endpoints of $\beta_{t_i} (\ti{s}_i)$. 

Let $u_{i, 1}, u_{i, 2}$ be the lifts of $u_i$ to the geodesic segments of $\ti\alpha_{t_i}$ which are adjacent to $\ti{s}_i$. 
 Then let $H_{i, 1}$ and $H_{i, 2}$ be the hyperbolic planes corresponding to  $u_{i, 1}$ and $u_{i, 2}$, respectively. 
Then, $g_i$ transversally intersects $H_{i, 1}$ and $H_{i, 2}$ at angle at least $\del/2$. 
Moreover, for every $\ep > 0$, if $i$ is large enough, then those intersection points are $\ep$-close to the endpoints of $\beta_{t_i} (\ti{s}_i)$.
Therefore, the distance between the hyperbolic planes $H_{i, 1}$ and $H_{i, 2}$ goes to $\infty$ as $i \to \infty$ (\Cref{fLongTranslationAlongPiecewiseGeodesic}). 
Therefore the translation length of $\rho_t(\alpha)$ goes to $\infty$ as desired. 
This contradicts the hypothesis. 
 Therefore $\rho_t(m)$ must be parabolic for sufficiently large $t > 0$.
\definecolor{lavendermist}{rgb}{0.9, 0.9, 0.98}
\begin{figure}
\begin{overpic}[scale=.2
] {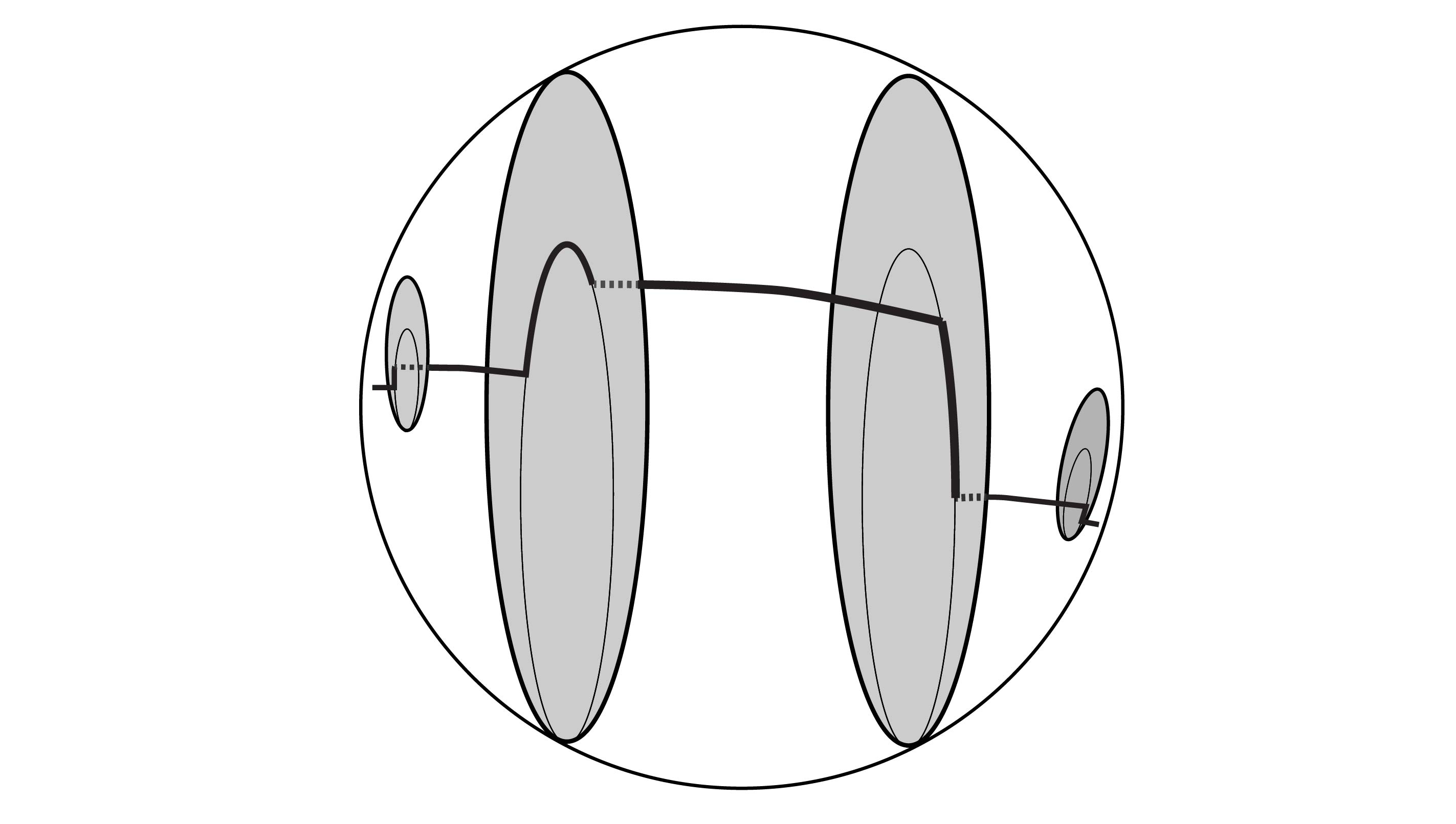} 
   \put(45.5 , 39){$\beta_{t_i}(\ti{s}_i)$}
  \put(35 , 20){\contour{lavendermist}{$H_{i, 1}$}  }
  \put(60 , 20){\contour{lavendermist}{$H_{i, 2}$}  }
      \end{overpic}
\caption{The quasi-geodesic $\beta_{t_i}(\ti\alpha_i)$ preserved by the hyperbolic element $\rho_{t_i}(\alpha)$.}\Label{fLongTranslationAlongPiecewiseGeodesic}
\end{figure}
\end{proof}

Let  $\phi\col \ti{S} \to S$ be the universal covering map. 
Let $\kap_t \col C_t \to \tau_t$ denote the collapsing map of $C_t$, and $\ti\kap_t \col \ti{C}_t \to \H^2$ denote its lift from the collapsing of the universal cover  (\S \ref{sThurston}).
We next show the convergence of the bending map. 
\begin{theorem}\Label{LimitPleatedSurfaceParabolic}
Suppose that $\rho_\infi(m)$ is parabolic. 
Then, up to an isotopy of $S$ in $t$, 
$\beta_t \cc \ti\kap_t \col \ti{S} \to \H^3$ converges to a $\rho_\infi$-equivariant continuous map $\alpha\col \til{S} \to  \H^3 \cup \CP^1$ such that 
\begin{itemize}
\item $\alpha^{-1}(\CP^1)$ is a $\pi_1(S)$-invariant multicurve on $\ti{S}$ isotopic to $\phi^{-1}(m)$ though $\pi_1(S)$-invariant multicurves, and
\item for each component $P$ of $\ti{S} \minus \phi^{-1}(m)$, the restriction $\beta_t \cc \ti\kap_t | P$ converges to the pleated surface for the component of $C_\infi$ corresponding to $P$.
\end{itemize}
\end{theorem}

\begin{proof}
The second assertion holds immediately from 
Theorem \ref{DevelopingMapWithParabolicCusp}.

The axis $a_t$ of $\rho_t(m)$ converges to the parabolic fixed point of $\rho_\infi(m)$.
By \Cref{AlmostParabolicCuspsAreHyperbolicOrElliptic},  $\rho_t(m)$ is a hyperbolic element for sufficiently large $t > 0$.
Let $D \sub \H^3$ be a horoball centered at the parabolic fixed point of $\rho_\infi(m)$. 
Then we pick a continuous path of $\rho_t(m)$-invariant subsets $D_t$ in $\H^3$ bounded by the surface equidistant from the axis of $\rho_t(m)$ so that 
$D_t$ converges to $D$ as $t \to \infi$.

Pick a sufficiently small $\del > 0$. 
For sufficiently large $t > 0$, 
let $N_t^\del$ be the component of the $\del$-thin part of $\tau_t$ homotopic to $m$.
Let $\ti{N}_t^\del$  be the lift of $N_t^\del$ to the universal cover $\ti\tau \cong \H^2$. 
If $\del > 0$ is sufficiently small, by the convergence of $\rho_t$,  the $\beta_t$-image of $\ti{N}_t^\del$ is eventually contained in $D_t$. 
This implies the first assertion.
\end{proof}

Next, we prove that cusp neighborhoods of the limit surface are isomorphic to cusp neighborhoods of a hyperbolic surface. 
\begin{proposition}\Label{HoroballQuotientCusp}
Suppose that $\rho_\infty(m)$ is parabolic.
The cusps of $C_\infty$ must be horodisk quotients.
\end{proposition}
\proof
Suppose, to the contrary, that the cusp neighborhoods of $C_\infty$ are {\it not} horodisk quotients. 

Let $C_t \cong (\tau_t, L_t)$ denote the Thurston's parameters of $C_t$. 
Then, as $\rho_\infty(m)$ is parabolic,  $L_t (m)$ converges to a non-negative integral multiple $2\pi n$ of $2\pi$. 
As the limit cusp neighborhoods are assumed to be  {\it not} horodisk quotients, $n$ is a positive integer. 
Similarly, let $C_\infty \cong (\tau_\infty, L_\infty)$ denote Thurston parameters of $C_\infty$.
Thus the $L_\infty$-transversal measure of each peripheral loop of $C_\infty$ is $2\pi n$.

For sufficiently large $t > 0$, $\rho_t(m)$ is not the identity; let $a_t$ be its axis (\Cref{axis}).
Pick $\del > 0$ less than the two-dimensional Margulis constant.
Let $N_t$ be the $\del$-thin part of $\tau_t$ homotopic to $m$. 
Let $\ti{N}_t$ be the lift of $N_t$ to the universal cover $\H^2$. 
If $\del > 0$ is sufficiently small, for all $t$ large enough,  each component of  $N_t \cap L_t$ is a geodesic segment connecting one boundary component of $N_t$ to the other.
Since the transversal measure of each peripheral loop of  $L_t$ is close to $2 \pi n > 0$ 
Thus, for $t \gg 0$, pick a fundamental domain $F_t$ in $\ti{N}_t$ bounded by two leaves of $\ti{L}_t$ such that a component $F_{t, 1}$ of $F_t \minus \ti{m}_t$ converges to a fundamental domain of the bending map $\beta_\infty\col \H^2 \to \H^3 \cup \CP^1$ (\Cref{LimitPleatedSurfaceParabolic}) near a cusp of $\tau_\infty$. 

 Let $\ell_t$ be a leaf of $\ti{L}_\infty$ bounding $F_t$, so that, for each component $r_t$ of  $\ell_t \minus \ti{m}_t$, the restriction of $\beta_t$ converges to a bi-infinite geodesic in $\H^3$ as $i \to \infty$.
Clearly the length of $\ell_t \cap \ti{N}_t$ goes to $\infty$, and the length of each segment of $\ell_t \cap \ti{N}_t \minus \ti{m}_t$ goes to $\infty$ as $t \to \infty$.

 Let $F_{t, 2}$ be the other component of $F_t \minus \ti{m}_t$.  
Then there is an element $\gam_t$ of $G_t$ such that the restriction of $\beta_t$ to $\gam_t F_{t, 2}$ converges to the fundamental domain of the other cusp of $C_\infty$.

We first show that if $\rho_t(m)$ is hyperbolic, it must be  ``almost elliptic'' for sufficiently large $t > 0$.
\begin{claim}\Label{AlmotElliptic}
Suppose that there is a sequence $t_1 < t_2 < \dots $ diverging to $\infty$, such that $\rho_{t_i}(m)$ is hyperbolic for each $i = 1, 2, \dots$. 
Then, the complex translation of $\rho_{t_i} (m)$ goes to zero from the imaginary direction as $i \to \infty$. 
In other words,  the sequence $\tr^2 \rho_{t_i}(m) \in \C$ converges to $4$ tangentially to  the real ray $\{ x \in \R \mid x \leq 4\}$.
\end{claim}
\begin{proof}
Suppose to the contrary that there is a sequence $t_1 < t_2 < \dots$ such that $\rho_{t_i}(m)$ is hyperbolic and the complex translation length converges to $0$ from the non-imaginary direction. 
As $\rho_{t_i}(m)$ is hyperbolic, the axis is a geodesic and it converges to the parabolic fixed point of $\rho_\infty(m)$.
Pick a point $p_i$ on $\beta_{t_i} (F_{t_i})$ closest to  $a_{t_i}$ in $\H^3$.
Let $R_i$ be the set of points in $\H^3$ whose distance from $a_{t_i}$ is at most the distance from $p_i$ to the axis $a_{t_i}$.

For each $i$, let $G_i$ be a one-dimensional Lie subgroup of $\PSL_2\C$ containing $\rho_{t_i}(m)$ such that the infinite cyclic group $\langle \rho_{t_i}(m) \rangle$ is asymptotically dense in $G_i$ as $i \to \infty$ w.r.t. the path metric on $G_i$ induced by the invariant metric on $\PSL_2\C$.
Since the complex translation length of $\rho_{t_i}$ converges to $0$ from a non-imaginary direction,  $G_t$ converges to a one-dimensional subgroup in $\PSL_2\C$ consisting of only hyperbolic elements except the identity. 
For every $i$, let $c_i$ be the $G_i$-invariant smooth curve in $\H^3$ passing $p_i$.
Then $c_i$ spirals on the boundary of $R_i$ limiting to the endpoints of $a_i$. (See \Cref{fAssymptoticallyAlmostElliptic}.)

The $\beta_{t_i}$-image of the leaf $\ell_i$ is a geodesic in $\H^3$ tangent to $R_i$ passing $p_i$. 
Then,  moreover, the geodesic $\beta_{t_i}(\ell_i)$ and the curve $c_i$ are asymptotically tangent to each other at $p_i$ as $i \to \infty$, because of the convergence of the bending map $\beta_{t_i}$ and the holonomy $\rho_{t_i}(m)$ as $i \to \infty$. 

Let $s_{i, 1}$ be the geodesic segment $\ell_i \cap F_{i, 1}$, so that $\beta_{t_i} (s_{i, 1})$ converges to a geodesic ray limiting to the fixed point of $\rho_\infty(m)$.
 Let $q_{1, i}$ be the endpoint of $s_{i, 1}$ that is on the boundary of $\ti{N}_i$, and  let $q_{2, i}$ be the other endpoint of $\ell_i \cap \ti{N}_i$.
Then  $\beta_{t_i} (q_{i, 1})$ converges to a point in $\H^3$ as $i \to \infty$. 
Then $\beta_{t_i} (\gam_i q_{i, 2})$ also converges to a point on $\H^3$. 

Since the length of each segment of $\ell_i \cap \ti{N}_i \minus \ti{m}_i$ goes to infinity, and $\beta_i(\ell_i)$ is asymptotically tangent to the curve $c_i$, therefore the distance between $\beta_{t_i} (q_{i, 1})$ and $\beta_{t_i} (q_{i, 2})$ diverges to $\infty$ as $i \to \infty$.  
This is a contradiction against the convergence of the bending map $\beta_{t_i}$ as $i \to \infty$.
\begin{figure}
\begin{overpic}[scale=.15
] {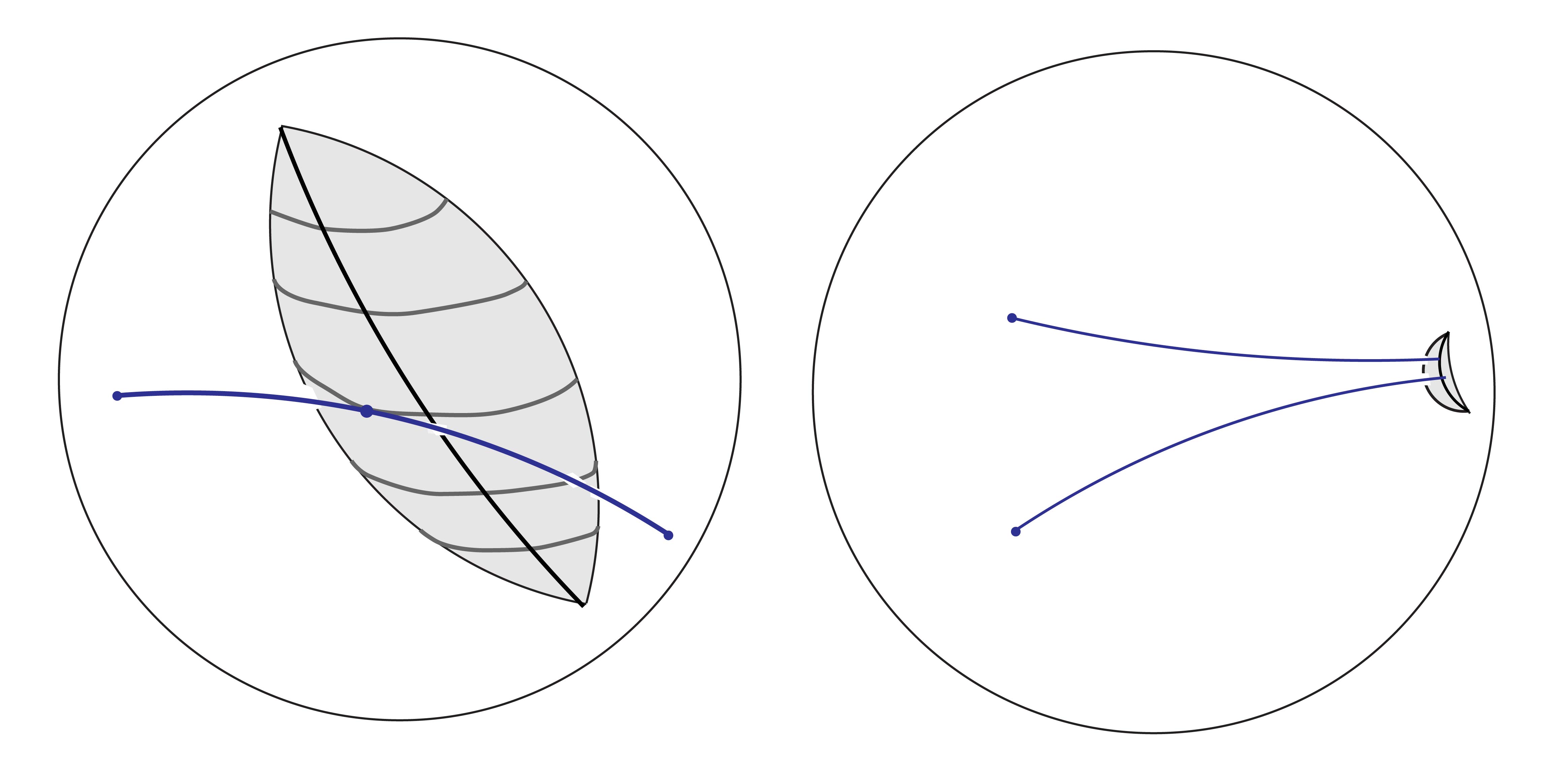} 
 \put(28 , 29 ){\contour{white}{\textcolor{darkgray}{$c_i$}}}  
 \put(22 , 19 ){\contour{white}{\textcolor{Blue}{$p_i$}} }
  \put(29 , 23 ){\tc{darkgray}{$R_i$}}
  \put(6 , 25 ){\contour{white}{\textcolor{Blue}{$\beta_{t_i} (q_{i, 1})$}}}
    \put(42 , 10 ){\contour{white}{\textcolor{Blue}{$\beta_{t_i} (q_{i, 2})$}}} 
  \put(22 , 31 ){\contour{white}{$a_{t_i}$}}
     \put(65, 12 ){\contour{white}{\textcolor{Blue}{$\beta_{t_i} (q_{i, 1})$}}}
    \put(60 , 30 ){\contour{white}{\textcolor{Blue}{$\beta_{t_i} (\gam_i q_{i, 2})$}} }   
      \end{overpic}
\caption{The left figure is the normalization of the right figure so that $p_i$ is at the center}\label{fAssymptoticallyAlmostElliptic}
\end{figure}
\end{proof}

Next we show the convergence of $\rho_t$ forces the convergence of twisting parameter along $m$. 
\begin{claim}\Label{TwistingConverges}
The Fenchel-Nielsen twisting parameter of $\tau_t$ along $m$ must converge (in $\R$) as $t \to \infty$.
\end{claim}

\begin{proof}
First,  for each non-identity element of $\PSL_2\C$,  we describe an associated foliation.
For a hyperbolic isometry or an elliptic isometry of $\H^3$, the hyperbolic planes containing its axis give a foliation on $\H^3$ minus the axis. 
For a parabolic isometry $\alpha \in \PSL_2\C$, pick a hyperbolic plane $H$ in $\H^3$ invariant under $\alpha$, which contains the parabolic fixed point.
Then there is a foliation of $\H^3$ by hyperbolic planes orthogonal to $H$ and containing the parabolic fixed point; this foliation is independent of the choice of $H$. 
For sufficiently large $t > 0$, as  $\rho_t(m)$ is not the identity (\Cref{AlmostParabolicCuspsAreHyperbolicOrElliptic}), let $\FFF_t$ denote such a foliation for $\rho_t(m)$.

Let $\ti{m}$ be a lift of $m$ to the universal cover $\ti{S}$.
Let $P_1, P_2$ be the connected components of $\ti{S} \minus \phi^{-1}(m)$ adjacent along $\ti{m}$.
For each $i = 1, 2$,   given a point $x_i$ in $P_i$ near $\ti{m}$, let $v_{\infty, i}$ be the tangent vector  at the point $\beta_\infty \circ \ti\kap_\infty(x_i)$ in $\H^3$ orthogonal its support hyperbolic plane of $x_i$ in the normal direction (\S \ref{sThurston}).
    Since the $L_t$-transversal measure along $m_t$ converges to $2\pi n > 0$,  we can pick $x_i$ so that
  $v_{\infty, i}$ is tangent to the foliation $\FFF_\infty$.
Similarly, for each $t \gg 0$, pick a point $x_{t, i}$ in $P_i$ such that, letting $v_{t, i}$ be the tangent vector of $\beta_t \circ \ti\kap_t$ at $x_{t, i}$ orthogonal to its support plane, $v_{t, i}$ is tangent to $\FFF_t$ and  $v_{t, i}$ converges to $v_{\infty, i}$ as $t \to \infty$. (See \Cref{fAlmostTwoPiTwisting}.)
 
 Let $\LLL_t$ be the circular measured lamination on $C_t$ which descends to the measured lamination of Thurston's parametrization by the collapsing map. 
 Let $e_t$ be the minimal transversal measure, given by $\LLL_t$,  of arcs connecting $x_1$ to $\rho_t(\gam_t) x_{t, 2}$.
Note that, since the isometry $\rho_t(m)$ preserves the foliation $\FFF_t$, the tangent vector  $\rho_t(\gam_t) v_{t, 2}$ at $\rho_t(\gam_t) x_{t, 2}$ is also tangent to $\FFF_t$.
 By \Cref{AlmotElliptic}, $\rho_t(m)$ is either parabolic, elliptic, or ``almost elliptic'' for $t \gg 0$.
 Therefore, for every $\ep > 0$, if $\del > 0$ is sufficiently small, then, for $t \gg 0$,  the transversal measure $e_t$ is $\ep$-close to a multiple of $2\pi$.
Thus the twisting parameter along $m$ converges modulo $2\pi$. 
By continuity, the twisting parameter of $\tau_t$ along $m$ must converge as $t \to \infty$. 

 \begin{figure}
\begin{overpic}[scale=.15
] {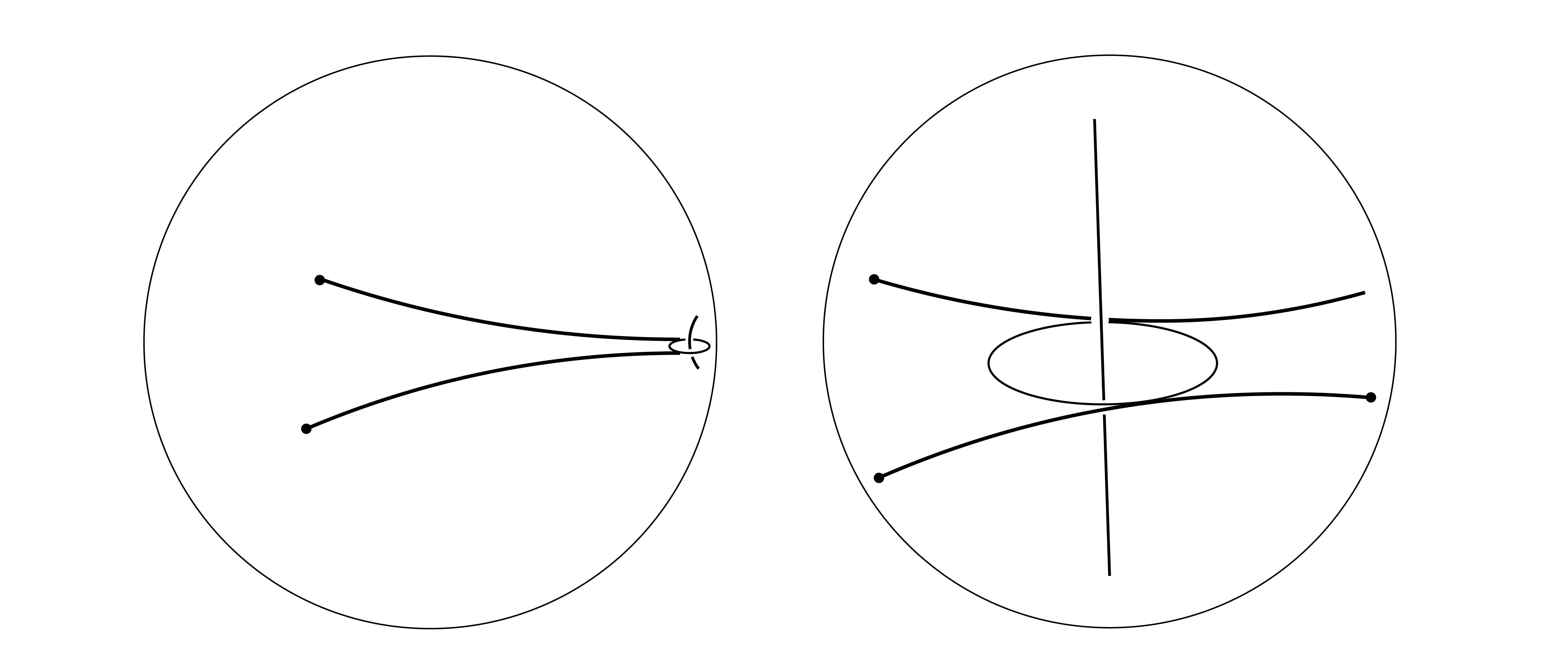} 
 \put(55.7,24){\color{black}\vector(0,1){5}}
  \put(56.5,26){\tc{black}{\contour{white}{$\rho_t(\gam_t) v_{2, t}$}}}
 \put(21.5,25.5){\tc{black}{$\rho_t(\gam_t) v_{2, t}$}}
 \put(56, 11.2){\color{black}\vector(0,1){5}}
 \put(87.45, 16.5){\color{black}\vector(0,1){5}}
 \put(19.5, 14.5){\color{black}\vector(0,1){5}}
 \put(14, 15){\tc{black}{$v_{t, 1}$}}
 \put(57, 15){\tc{black}{$v_{t, 1}$}}
 \put(88.5, 17){\contour{white}{$v_{t, 2}$}}
  \put(20.4, 24){\color{black}\vector(0,1){5}}
 \put(71.3 , 12){\textcolor{black}{$a_t$}}  
 \put(41 , 16){\textcolor{black}{$a_t$}}  
      \end{overpic}
\caption{The right figure is a normalization of the left figure so that the axis $a_t$ passes the center.}\label{fAlmostTwoPiTwisting}

\end{figure}
\end{proof}
By 
\Cref{TwistingConverges}, the Fenchel-Nielsen twisting parameter of $\tau_t$ along $m$ converges. 
For all $t > 0$, let $Q_{t, 1}$ and $Q_{t, 2}$ be the adjacent components of $\H^2 \minus \psi^{-1}(m_t)$ corresponding to $P_1$ and $P_2$, respectively, so that $Q_{t, 1}$ and $Q_{t, 2}$ are separated by the geodesic $\ti{m}_t$. 
Then, as the restriction of  $\beta_t$ of the component $Q_{t, 1}$ converges,  uniformly on compact subsets,  to the bending map of the corresponding cusp neighborhood of $C_\infty$ by \Cref{LimitPleatedSurfaceParabolic}.
Then, since the length of the geodesic loop $m_t$ goes to $0$ as $t \to \infty$, the convergence of the twisting parameter implies that the restriction of  $\beta_t$ to $Q_{t, 2}$  converges to the parabolic fixed point of $\rho_\infty (m)$ uniformly on compact subsets. 
This is a contradiction against the convergence of the bending map $\beta_t$ of $Q_{t, 2}$ uniformly on compact subsets guaranteed by \Cref{LimitPleatedSurfaceParabolic}. 
\Qed{HoroballQuotientCusp}


\begin{theorem}\Label{ParabolicDevConverges}
Suppose that $\rho_\infi(m)$ is parabolic.
Then, by an appropriate isotopy of $S$ in $t$,  $f_t \col \til{S} \to \CP^1$ converges to a $\rho_\infi$-equivariant continuous map $\t{S} \to \CP^1$ such that, for some multiloop $M$ on $S$ consisting of finitely many parallel copies of $m$,  
\begin{itemize}
\item $f_\infi$ is a local homeomorphism on $\ti{S} \minus \phi^{-1}(M)$, and
\item $f_\infi$ takes each component of $\phi^{-1}(M)$ to its corresponding parabolic fixed point.
\end{itemize}
\end{theorem}

Under the assumption of \Cref{ParabolicDevConverges}, each cusp of $C_\infty$ is a horodisk quotient by \Cref{HoroballQuotientCusp}.
Thus, by \Cref{AlmostParabolicCuspsAreHyperbolicOrElliptic}, $\rho_t(m)$ is hyperbolic for all sufficiently large $t > 0$, and it  converges to the parabolic element $\rho_\infty(m)$ as $t \to \infty$.

More generally, let $\gam_t \in \PSL_2\C, ~t \geq 0$ be a path of hyperbolic elements such that $\gam_t$ converges to a parabolic element $\gam_\infty$ in $\PSL_2\C$ as $t \to \infi$.  
Let $G_t$ be the one-parameter subgroup of $\PSL_2\C$ containing $\gam_t$ such that the cyclic group generated by $\gam_t$ is asymptotically dense in $G_t$ with respect to the path metric on $G_t$ induced by the (left) invariant metric on $\PSL_2\C$.

    Continuously conjugate $\gam_t$ by elements $\omega_t$ of $\PSL_2\C$ so that the axis of  $\omega_t \cdot \gam_t \coloneqq r_t w_t r_t^{-1}$ remains, for all $t$,  to be the geodesic in $\H^3$ which connects $0$ to $\infty$ in the ideal boundary $\CP^1 = \C \cup \{\infty\}$. 

\begin{proposition}\Label{DevConvergeenceOfThinPart}
Let $A$ be a cylinder and homeomorphically identify $A$ with $[-1, 1] \times \SS^1$, and let $\ti{A}$ be the universal cover of $A$. 
Let $A_t ~( t > 0)$ be a path of $\CP^1$-structures on a cylinder $A$, and let $f_t$ be its developing map which changes continuously in $t$, such that 
\begin{itemize}
\item the holonomy of $A_t$ is the limit holonomy isomorphism $\pi_1(S) \cong \Z \to \langle\gam_t \rangle$,
\item each boundary of $A_t$ develops onto a $G_t$-invariant curve on $\CP^1$ for all $t > 0$. 
 \item for each boundary circle $b$ of $A$, the restriction of $f_t$ to the lift $\ti{b}$ to $\ti{A}$ converges to a $G_\infty$-invariant simple curve on $\CP^1$ (which is a $G_\infty$-invariant round circle minus the parabolic fixed point). 
 \end{itemize}
 Then, by an isotopy of $A$ fixing the boundary, $\dev A_t\col \ti{A} \to \CP^1$ converges to an continuous map $f_\infty\col \ti{A} \to \CP^1$ such that
 \begin{itemize}
 \item  $f_\infty$ is equivariant via the isomorphism $\Z \to\langle \gam_\infty\rangle$;
 \item there is a multiloop $M$ consisting of loops homotopy equivalent to $A$, such that $f_\infty$ is a local homeomorphism on $\ti{A} \minus \ti{M}$;
 \item $f_\infty$ takes $\ti{M}$ to the parabolic fixed point of $\gam_\infty$. 
\end{itemize}
\end{proposition}
\begin{proof}
We construct a path of fundamental membranes $Z_t$ for the developing maps $f_t$ which give the desired limit as $t \to \infty$. 

The normalized developing map  $\omega_t \circ f_t\col \ti{A} \to \C \cup \{\infty\}$ is  identified with the restriction of $\exp\col \C \to \C^\ast$ to a bi-infinite strip $I_t$ bounded by parallel lines in $\C \cong \E^2$. 
Let $b_1$ and $b_2$ denote the boundary components of $A$.  
Regarding $b_1, b_2$ as simple closed curves, we can lift $b_1$ and $b_2$ to segments $s_1$ and $s_2$, respectively, of segments of the boundary components of $\ti{A}$. 
For each $t > 0$ and  $i = 1, 2$, let $s_{i, t}$ be the segment of the boundary line of $I_t$ such that $\omega_t \circ f_t (s_i) = \exp(s_{i, t})$.
Then $s_{1, t}$ and $s_{2, t}$ are parallel and have the same length. 
Thus $s_{2, t}$ is the Euclidean translation of $s_{1, t}$ by unique $z_t \in \C \minus \{0\}$.  
\begin{claim}\Label{FundamentalSegments}
\begin{enumerate}
\item The length of $s_{i, t}$ goes to zero as $t \to \infty$, and \Label{iLengthZero}
\item $z_t$ converges to an integer multiple of $2\pi i$ as $t \to \infty$. \Label{iTwoPiTranslation}
\end{enumerate}
\end{claim}
\begin{proof}
(\ref{iLengthZero})
As $0$ and $\infty$ are the fixed points of $\omega_t \gam_t \omega_t^{-1}$ and $\gam_t$ converges to $\gam_\infty$, both
$\omega_t^{-1}(0)$ and $\omega_t^{-1}(\infty)$ converge to the parabolic fixed point of $\gam_\infty$ as $t \to \infty$.
Since the development of $b_i$ converges to a $G_\infty$-invariant curve on $\CP^1$, clearly the development of $s_{i, t}$ converges to a simple arc contained in the $G_\infty$-invariant curve. 
Therefore, since $f_t = w_t^{-1} \exp$ on $\ti{A}$, the norm of the derivative of $f_t$ at each point on the segment $s_i$ goes to infinity as $t \to \infty$. 
Hence the Euclidean length of $s_{i, t}$ must go to zero as $t \to \infty$.

(\ref{iTwoPiTranslation})
Since the Euclidean length of $s_{1, t}$ goes to zero on $I_t \subset \C$,  translating $I_t$ by a multiple of $2\pi i$,  we may assume that $s_{1, t}$ converges to a point $p$ on $\C$. 
Let $q \in \CP^1$ be the parabolic fixed point of $\gam_\infty$.
Let $K$ be a compact subset $K$ in $\CP^1 \minus \{q\}$ and  $U_p$ be a neighborhood of $p$ in $\C$. 
Let $U$ denote the union of translates of $U_p$ by the integer multiples of $2\pi i$. 
Then, if $t$ is sufficiently large, then $\omega_t^{-1} \exp (I_t \minus U)$ is contained in $\CP^1 \minus K$. 
Therefore, as the developments of $s_{1, t}$ and $ s_{2, t}$ converge to simple arcs in $\CP^1 \minus \{q\}$, their difference $z_t$ must converge to a multiple of $2\pi i$.  
\end{proof}

Let $n$ be the integer such that $z_t$ converges to $2\pi i n$. 
Pick a polygonal fundamental domain $Z_t$ of $A_t$ in $I_t$ with following properties: 
$Z_t$ is a union of $(n + 1)$-rectangles $R_{t, 1}, R_{t, 2}, \dots, R_{t, n+ 1}$ and $n$ parallelograms $P_{t, 1}, \dots, P_{t, n}$ as in the figure (\Cref{fZigsagTwo}) so that 
\begin{itemize}
\item for each $i = 1, \dots, n, n+ 1$, a pair of edges of $R_{t, i}$ are parallel to the boundary of the Euclidean strip $I_t$ , the boundary segment $s_{1, t}$ is an edge of $R_{t, 1}$,  the boundary segment $s_{2, t}$ is an edge of $R_{t, n + 1}$, and, for each $i = 2, \dots, n-2$,   the Euclidean translation of $s_{t, i}$ by $2\pi i$ decomposes $R_{t, i + 1}$ into two isometric rectangles, and
\item for each $i = 1, \dots, n$, the parallelogram $P_{t, i}$ have edges parallel to the boundary of $I_t$ which are an edge $R_{t, i}$ and an edge $R_{t, i + 1}$.
\end{itemize}
In addition,  we take $R_{t, 1}, R_{t, 2}, \dots, R_{t, n+ 1}$ and $n$ parallelograms $P_{t, 1}, \dots, P_{t, n}$ appropriately so that
\begin{itemize}
\item the development of $P_{t, i}$ by $f_t$ converges to the parabolic fixed point of $\gam_\infty$ as $t \to \infty$; 
\item the $f_t$-images of $R_{t, 1}$ and $R_{t, n + 1}$ converge to horodisks bounded by the limit of $f_t (\ti{b})$ in the hypothesis, and the restriction of $f_t$ to $R_{t, 1}$ and $R_{t, n + 1}$ converge to a developing map of horodisk quotients; 
\item for $i = 2, \dots, n$, the restriction of $f_t$ to $R_{t, i}$ converges to a developing map of the Euclidean cylinder $(\CP^1 \minus \{p\})/ \langle \gam_\infty \rangle$
\end{itemize}
(\Cref{fCloseToParabolicLimit}).
Let $M$ be a multiloop on $A$ consisting of $n$ boundary parallel loops. 
Pick a path of regular neighborhood $N_t$ of $M$ so that $N_t$ converges to $M$ as $t \to \infty$. 
Isotope $A$ so that a fundamental domain $F$ of $\ti{A}$ maps to $Z_t$ and that $N_t$ are identified with $P_{t, 1}, \dots, P_{t, n}$.
Then we a desired convergence. 

\begin{figure}
\begin{overpic}[scale=.10
] {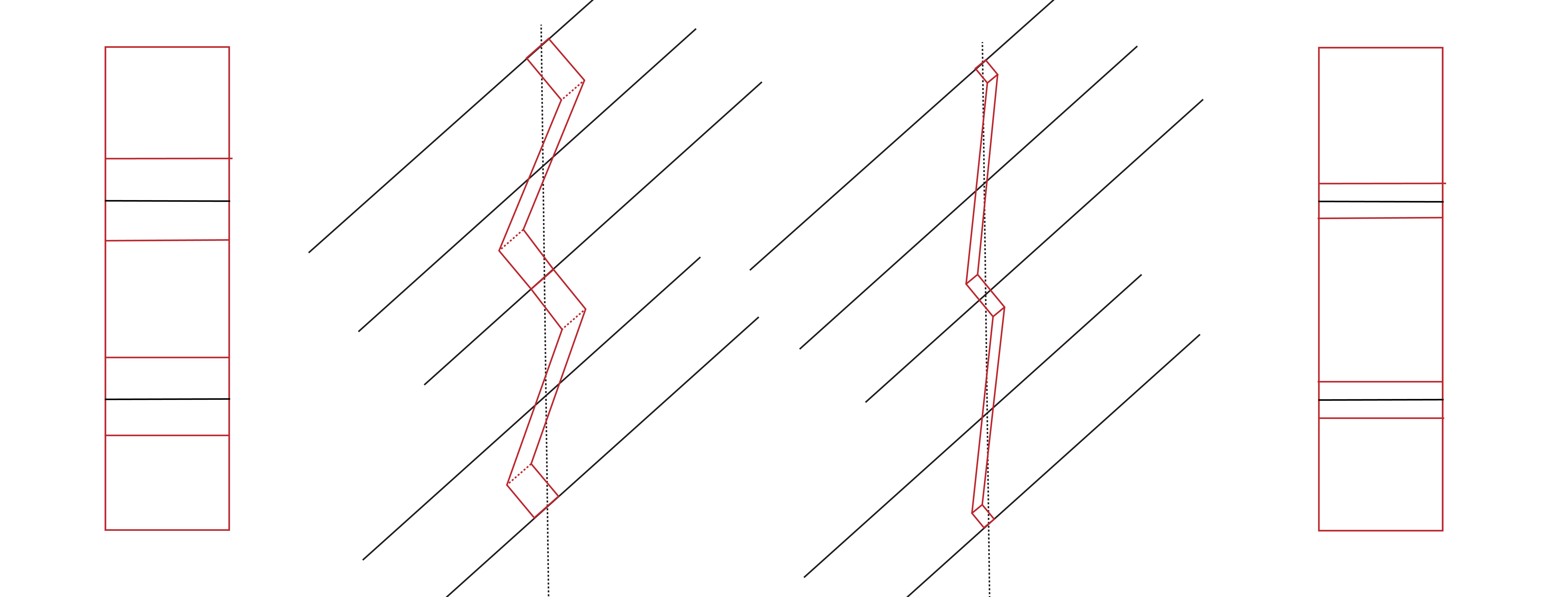} 
  \put(33 , 37){\contour{white}{$s_{1, t_1}$}}  
  \put(20 , 15){$I_{t_1}$}  
  \put(68 , 20){$I_{t_2}$}  
  \put(30 , 2){\contour{white}{$s_{2, t_1}$}  }
 \put(58 , 30){\contour{white}{$s_{1, t_2}$}  }
  \put(65 , 2){\contour{white}{$s_{2, t_2}$}  }
 \put(38 ,33 ){\contour{white}{\textcolor{red}{$R_{t_1, 1}$}}  }
 \put(37 ,25 ){\contour{white}{\textcolor{red}{$P_{t_1, 1}$}}  }
 \put(37 , 15 ){\contour{white}{\textcolor{red}{$P_{t_1, 2}$}}  }
 \put(2 , 22 ){\textcolor{red}{$N_{t_1}$}}  
 \put(8 , 0 ){\textcolor{red}{$F$}}  
 \put(94 , 12 ){\textcolor{red}{$N_{t_2}$}}  
 \linethickness{1pt}
 \put(18,20){\color{black}\vector(1,0){5}}
   \put(19 ,22 ){\contour{white}{$f_{t_1}$}  }
      \put(19 ,22 ){$f_{t_1}$}  
          \put(80 ,22 ){$f_{t_2}$}  
 \put(83,20){\color{black}\vector(-1,0){5}}
      \end{overpic}
\caption{The limiting behavior of the fundamental membrane $Z_t$ of $A_t$, where $n = 2$ and $t_1 < t_2$. }\label{fZigsagTwo}
\end{figure}

\begin{figure}
\begin{overpic}[scale=.2
] {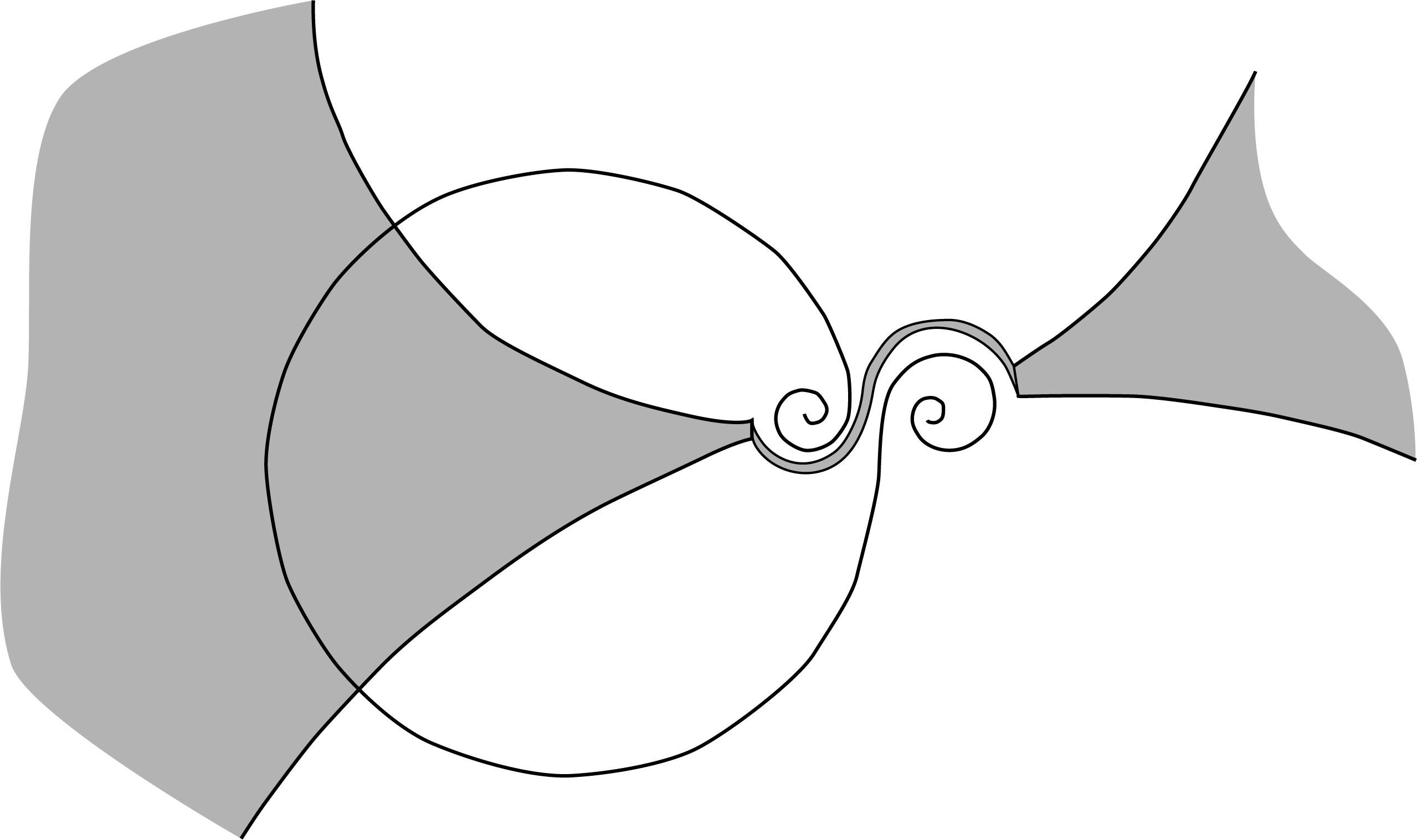} 
   \put( 20, 26){\contour{lightgray}{$f_t(R_{t, i})$}}
   \put( 60, 39){$f_t(P_{t, i})$} 
      \end{overpic}
\caption{}\label{fCloseToParabolicLimit}
\end{figure}
\end{proof}

\begin{proof}[Proof of Theorem \ref{ParabolicDevConverges}]
We already know the convergence of the developing map in every thick part by Theorem \ref{DevelopingMapWithParabolicCusp}. 
There are two cusps $c_1, c_1$ of $C_\infty$, which are horodisk quotients by \Cref{HoroballQuotientCusp}.
For each cusp $c_i$ of $C_\infty$, pick a simple closed curve $\ell_i$ which develops to a $G_\infty$-invariant simple curve on $\CP^1$. 
Then, for large $t > 0$, pick a simple closed curve $\ell_{i, t}$ on $C_t$ such that $\ell_{i, t}$ develops onto a $G_t$-invariant curve on $\CP^1$ and  $\ell_{i, t}$ converges to $\ell_i$ as $t \to \infty$. 

Let $A_t$ be the cylinder in $C_t$ bounded by $\ell_{1, t}$ and $\ell_{2, t}$. 
Then we can take such a path of cylinders $A_t$ in $C_t$ and a constant $\del > 0$ such that $A_t$ contains the $\del$-thin part of $C_t$ for all sufficiently large $t$. 
 Thus, by applying \Cref{DevConvergeenceOfThinPart} to $A_t$, we obtain a multiloop for the desired convergence property of $\dev C_t$. 
\end{proof}

\subsection{Convergence in holomorphic quadratic differential in the case of parabolic cusps}
Under the assumption that $\rho_\infty(m)$ is parabolic, we already have the limit $C_\infty$ of  $C_t$ as $t \to \infty$ where $C_\infty$ is a $\CP^1$-structure on a Riemann surface $X_\infty$ with two cusps homeomorphic to $S \minus m$. 
Moreover, each cusp of $C_\infty$ has a neighborhood which is a horodisk quotient (i.e. isomorphic, as a $\CP^1$-structure,  to a cusp neighborhood of a hyperbolic surface) by \Cref{HoroballQuotientCusp}. 
Then the holomorphic quadratic differential $\phi_\infty$ on $X_\infty$ representing $C_\infty$ has, at worst, a first order pole at each cusp. 
Therefore we have the following convergence of the differential. 
 \begin{theorem}\Label{LimitiDifferentialParabolicCusp}
Suppose that $\rho_\infi (m)$ is parabolic. 
Then $X_t$ converges to a nodal Riemann surface $X_\infi$ such that $X_\infi$ minus the node is homeomorphic to $S \minus m$ and $q_t$ converges to a  quadratic differential $q_\infi$ on $X_\infi$ such that the node is at worst first order pole.
\end{theorem}

\section{$\rho_\infi(m)$ cannot be hyperbolic}\Label{sHyperbolicNeck}

In this section, we show that $\rho_\infi(m)$ cannot be a hyperbolic element. 
\begin{lemma}\Label{HyperbolicLimit}
Suppose that $X_t$ is pinched along a loop $m$ and $\rho_\infi(m)$ is hyperbolic. 
Then 
\begin{enumerate}
\item 
$C_t$ converges to a $\CP^1$-structure $C_\infi$ on a compact surface with two punctures, which is homeomorphic to $S \minus m$, in the sense that, for every $\ep > 0$, the $\ep$-thick part of $C_t$ converges to the $\ep$-thick part of $C_\infi$ uniformly, and  \Label{iHyperbolicNeckConverenceAwayFromCusp}
\item $\rho_\infi (F)$ is non-elementary for each component $F$ of $S \minus m$. \Label{iHyperbolicNonelementary}
\end{enumerate}
\end{lemma}

\begin{proof}
(\ref{iHyperbolicNeckConverenceAwayFromCusp}) is an immediate corollary of \Cref{ConvergenceThickPart}.

(\ref{iHyperbolicNonelementary})
Let $F_\infi$ be the component of $C_\infi$ corresponding to $F$. 
Let $(\sigma, \nu)$ denote the Thurston parametrization of $F_\infi$. 
Then $\sigma$ is a hyperbolic surface with geodesic boundary, such that the lengths of the boundary components are the translation length of $\rho_\infi(m)$ (see the proof of Lemma \ref{CuspNbhdInThurstonCoordinates}). 
Let $(\til{\sigma}, \til{\nu})$ be  the universal cover of $(\sigma, \nu)$ so that $\til{\sigma}$ is a convex subset of $\H^2$ bounded by geodesics and that $\til{\nu}$ is a $\pi_1(\sigma)$-invariant lamination on $\til{\sigma}$.

Let $\alpha\col \til{\sigma} \to \H^3$ be its pleated surface equivariant by the holonomy of $F_\infi$.
Let $\ell$ be a boundary geodesic of $\til{\sigma}$. 
Then the endpoints of $\alpha (\ell)$ are in the limit set $\Lambda$ of $\Hol F_\infi$, as $\alpha(\ell)$ is the axis of the hyperbolic $\rho_\infi(m)$. 
Every component $R$ of $\ti{\sigma} \minus \ti{\nu}$ has at least three ideal points.
Then the ideal points of $\alpha(R)$ are in $\Lambda$ (see \cite[Lemma 5.1]{Baba20ThurstonParameter}). 
Thus $\rho_\infi |F$ is non-elementary. 
\end{proof}

\begin{lemma}\Label{FoliatedHyperbolicCuspNeighborhood}
For each cusp $p$ of $C_\infi$, there is a neighborhood of $p$ foliated by isomorphic admissible loops which develop to simple curves on  $\CP^1$ invariant under a one-parameter subgroup in $\PSL_2\C$ containing $\rho_t(m)$.
\end{lemma}
\begin{proof}
The developing map near a cusp neighborhood is the restriction of the exponential map $\exp\col \C \to \C^\ast$; moreover, by taking an appropriate neighborhood, one can assume that the restriction is to a half-plane bounded by a straight line in $\C$ invariant under the deck transformation corresponding to the hyperbolic element $\rho_t(m)$. 
 
 The half-plane is foliated by straight lines parallel to the boundary, and this foliation descends to a desired foliation of the cusp neighborhoods by admissible loops. 
\end{proof}

\begin{proposition}\Label{FoliatedNeck}
If $\ep > 0$ is sufficiently small, then, for every sufficiently large $t > 0$, there is a cylinder $A_t$ in $C_t$ homotopy equivalent to $m$ such that 
\begin{itemize}
\item $A_t$ changes continuously in $t \gg 0$;
\item $A_t$ is foliated by admissible loops whose developments are invariant under a one-parameter subgroup $G_t$ in $\PSL_2\C$ containing $\rho_t(m)$;
\item $A_t$ contains the conformally $\ep$-thin part of $C_t$;
\item $C_t \minus A_t$ converges to a $\CP^1$-structure on $S \minus m$ whose boundary components are admissible loops. 
\end{itemize}
\end{proposition}

\begin{proof}
Consider the cusp neighborhoods of $C_\infi$ foliated by admissible loops by Lemma \ref{FoliatedHyperbolicCuspNeighborhood}.
By the convergence of Lemma \ref{HyperbolicLimit} and the stability of the admissible loops, for $t \gg 0$, there is a cylinder $A_t$ foliated by admissible loops whose developments are invariant under $G_t$.
Then it is easy to realize other desired properties.
\end{proof}

 By Claim \ref{HyperbolicLimit} (\ref{iHyperbolicNonelementary}),  the developing map of $C_t \minus A_t$ converges uniformly on compact subsets. 
By normalizing $\rho_t$ by $\PSL_2\C$ continuously, so that, for sufficiently large $t > 0$,  we can, in addition, assume that the axis of the hyperbolic element $\rho_t(m)$ connects $0$ and $\infty$ of $\CP^1 = \C \cup \{\infty\}$.
Then the developing map of the cylinder $A_t$ is the restriction of the exponential map $\exp\col \C \to \C^\ast$ to the strip region $R_t$ bounded by parallel lines, since the boundary components of $A_t$ develop to $G_t$-invariant curves by \Cref{FoliatedNeck}.  
 Since the boundary components of $A_t$ converge to peripheral loops of $C_\infty$, by the continuity of $\dev C_t$ in $t$, the region $R_t$ converges to a strip in $\C$ with finite width. 
Therefore $A_t$ must converge as $t \to \infty$. 
 Thus $C_t$ converges to a $\CP^1$-structure on $S$--- this contradicts the divergence of $C_t$ in the deformation space. 
Hence $\rho_\infi(m)$ cannot be hyperbolic. 
\section{$\rho_\infi(m)$ cannot be  elliptic}\Label{sEllipticNeck}

In this section, similarly to the previous section (\S\ref{sHyperbolicNeck}), we show that $\rho_\infi (m)$ cannot be elliptic. 
To show this,  we assume, to the contrary, that $\rho_\infi(m)$ is elliptic and obtain a contradiction against the convergence of $\rho_t$ as $t \to \infi$.
By \Cref{ConvergenceThickPart}, we have
\begin{proposition}\Label{mEllipticConvergencOfThickPart}
Suppose that $\rho_\infi(m)$ is elliptic. 
Then $C_t$ converges to a $\CP^1$-structure $C_\infi$ on a compact surface minus two points homeomorphic to $S \minus m$, in the sense that, for every $\ep > 0$, the $\ep$-thick part of $C_t$ converges to the $\ep$-thick part of $C_\infi$.  
\end{proposition}
\begin{lemma}\Label{EllipticNeckDiscreteStabilizer}
For each component $F_\infi$ of $C_\infi$,  the stabilizer of $\rho_\infi(F_\infi)$ by conjugation is a discrete subgroup in $\PSL_2\C$.
\end{lemma}

\begin{proof}
Let $F_\infi$ be a component of $C_\infi$. 
Then let $(\sigma, \nu)$ be the Thurston parametrization of $F_\infi$, and let $(\ti{\sigma}, \ti{\nu})$ be the universal cover of $(\sigma, \nu)$.
 Then the rotation angle of the elliptic element $\rho_\infi (m)$ is, modulo $2\pi$, equal to the total weight, given by $\nu$, of the leaves ending at a puncture (Proposition \ref{CuspClassification}).
Let $\beta_\infi\col \ti{\sigma} \to \H^3$ be the equivariant pleated surface. 
Pick a leaf $\ell$ of $\nu$ whose endpoints are at cusps of $\nu$; then $\ell$ is an isolated leaf.
Let $\ti\ell$ be a leaf of $\ti\nu$ which is a lift of $\ell$. 
Then its image $\beta_\infi (\ell)$ is a geodesic in $\H^3$. 
Each endpoint of this geodesic is a fixed point of the parabolic element in the image $\rho_\infi(\pi_1(F))$ corresponding to its associated peripheral loop.

As the leaf $\ell$ is isolated,  
$\ell$ bounds a component $P$ of $\ti{\sigma} \minus \ti{\nu}$, and $P$ has at least three ideal points. 
Then,  for each ideal point $p$ of $P$, let $\gam \in \pi_1(F_\infi)$ be such that $\gam$ fixes $p$.
Then  $\beta_\infi(p)$ is fixed by the elliptic element $\rho_\infi(\gam)$. 
Therefore, the stabilizer of $\rho_\infi( F_\infi)$ is a discrete subgroup of $\PSL_2\C$.
\end{proof}
Similarly to \Cref{FoliatedNeck}, the following follows from  Lemma \ref{mEllipticConvergencOfThickPart} and Lemma \ref{EllipticNeckDiscreteStabilizer}:
\begin{proposition}\Label{EllipticNeck}
If  $\ep > 0$ is sufficiently small, then for every sufficiently large $t > 0$,
there is a cylinder $A_t$ in $C_t$ homotopy equivalent to $m$ such that
\begin{itemize}
\item $A_t$ changes continuously in $t  \gg 0$;
\item $A_t$ is foliated by loops whose developments are invariant under the one-dimensional subgroup $G_t$ of $\PSL_2\C$ containing $\rho_t (m)$, and $G_t$ converges to a one-dimensional subgroup $G_\infi$ of $\PSL_2\C$ containing $\rho_\infi (m)$;
\item $A_t$ contains the conformally $\ep$-thin part of $C_t$ homotopic to $m$;
\item $C_t \minus A_t$ converges to a $\CP^1$-structure  on $S \minus m$ such that the boundary components cover round circles on $\CP^1$.
\end{itemize}
\end{proposition}

\begin{proposition}
Suppose that $\rho_\infi(m)$ is elliptic. 
Then $C_t$ converges to a $\CP^1$-structure on $S$, which is a contradiction as desired.
\end{proposition}

\begin{proof}
Fix sufficiently small $\ep > 0$, and let $A_t$ be a cylinder given by Proposition \ref{EllipticNeck}.
Let $\phi_t\col \ti{C}_t \to C_t$ be the universal covering map.
Then the developing map of $\ti{C_t} \minus \phi_t^{-1}(A_t)$ converges uniformly on compact subsets.  
Let $\ti{A}_t$ be the component of $\phi_t^{-1}(A_t)$ invariant under $m \in \pi_1(S)$, so that $A_t$ changes continuously in $t$.
We can normalize $\dev C_t$ by $\PSL_2\C$ continuously in $t$,  such that, for sufficiently large $t > 0$,  the geodesic axis of $\rho_t(m)$ connects $0$ and $\infty$ of $\CP^1 = \C \cup \{\infty\}$.
Then, the restriction of $\dev C_t = f_t$ to $\ti{A}_t$ is the restriction of the exponential map $\exp\col \C \to \C^\ast$ to an infinite strip in $\C$.
Since $f_t$ converges on the boundary components of $\ti{A}_t$, thus the restriction of $f_t$ to $\ti{A}$ converges as $t \to \infty$. 
 Hence $A_t$ must converge as $t \to \infty$ as a $\CP^1$-structure on a cylinder with boundary.
Therefore $C_t$ converges to a $\CP^1$-structure on $S$, which is a contradiction. 
\end{proof}

\section{Limit when $\rho_\infi(m) = I $}\Label{mTrivialHolonomy}
Let $A$ be a regular neighborhood of a loop $m$ on $S$.
For $t \geq 0$, let  $(\tau_t, L_t)$ be Thurston parameters of $C_t$.
Let $\beta_t\col \H^2 \to \H^3$ be its $\rho_t$-equivariant pleated surface.
Let $\kap_t \col C \to \tau$ be the collapsing map, and $\ti{\kap}_t\col \ti{C} \to \H^2$ denote the lift of $\kap$ to the map between their universal covers. 
Let $a_t$ denote the axis of $\rho_t(m) \in \PSL_2\C$ (\Cref{axis}).

Note that a $\CP^1$-structure on $S$ is defined up to an isotopy of the base surface $S$. 
Thus the developing map $f_t\col \ti{S} \to \CP^1$ of the path $C_t$ of $\CP^1$-structures on $S$ can be modified by an isotopy $\psi_t\col S \to S$ in $t$ without changing $C_t$. 
Finally, recall that $\phi \col \ti{S} \to S$ is the universal covering map.
\begin{theorem}\Label{TrivialNeck}
Suppose that $\rho_\infi(m) = I$. 
Then the following hold:
\begin{enumerate}
\item $\rho_t(m) \neq I$ for sufficiently large $t > 0$. \Label{iNotExactlyI}
\item The Fenchel-Nielsen twisting parameter (in  $\R$) of $X_t$ along $m$ diverges to either $\infty$ or to $- \infty$. \Label{iTwistDieverges}
\item \Label{iSubsequencialConvergence}
For every diverging sequence  $0 < t_1 < t_2 < \dots$,
there is  a subsequence such that
\begin{enumerate}
\item \Label{iIAixsConverging}
the axis $a_{t_i}$ converges to a point on $\CP^1$ or a geodesic in $\H^3$, denoted by $a_\infi$;
\item there is a $\CP^1$-structure in $\PPP(S \minus m)$ such that, for every  $\ep > 0$, the $\ep$-thick part of $C_{t_i}$ converges to the $\ep$-thick part of $C_\infi$ uniformly;  \Label{iThickPartConvergence}
\item up to an isotopy of $S$ in $t$,  the restriction of $f_{t_i}$ to $\til{S} \minus \phi^{-1}(A)$ converges to a $\rho_\infi$-equivariant continuous map $f_\infi\col \til{S} \minus \phi^{-1}(A) \to \CP^1$ as $t_i \to \infi$ such that, for each component $\ti{A}$ of $\phi^{-1}(A)$, 
its boundary components map onto the ideal points of  $a_\infi$.  \Label{iBoundaryComponentAtEndpoints}
\end{enumerate}
\item   the pleated surface $\beta_{t_i}\cc \ti{\kap}_{t_i} \col \til{S} \to \H^3$ converges to a $\rho_\infi$-equivariant continuous map $\ti{S} \to \H^3 \cup \CP^1$, up to an isotopy of $S$. \Label{iIPleatedSurfaceConverging}
\end{enumerate}
\end{theorem}
Notice that, by the surjectivity in (\ref{iBoundaryComponentAtEndpoints}),
if $a_\infi$ is a geodesic, then the different boundary components of $\ti{A}$ map to the different endpoints of $a_\infi$. 

We will prove (\ref{iIPleatedSurfaceConverging})  in the next subsection  (\S \ref{smIPleatedSurface}).
In this section, we will prove the other assertions: 
(\ref{iNotExactlyI}) will be proved in Lemma
\ref{m-neq-I}; (\ref{iTwistDieverges}) will be proved in  Lemma \ref{IDivergingTwist}; (\ref{iBoundaryComponentAtEndpoints}) will be proved in Proposition \ref{PunctureAtAxis}.
The proof of 
(\ref{iThickPartConvergence}) is similar to the proof of \Cref{ConvergenceThickPart}.

Let $C_\infi \cong (\sigma_\infi, \nu_\infi)$  denote the Thurston parameterization, where $\sigma_\infi$ be a hyperbolic structure in the Teichmüller space $\TT(S \minus m)$ and $\nu_\infi$ be a measured lamination on $\sigma_\infi$. 
Then $\sigma_\infi$ has two cusps.
 At each cusp $c$ of $\sigma_\infi$, there are only finitely many leaves of $\nu_\infi$ ending at $c$ by a basic property of geodesic laminations (\cite{CanaryEpsteinGreen84}).
 Then, since $\rho_\infi(m) = I$, the total weight of those leaves is a positive $2\pi$-multiple.
\begin{lemma}\Label{IrrationalThenNonelementary}
If
$\nu_\infi$ contains an irrational sublamination, then the holonomy of $C_\infi$ is non-elementary. 
\end{lemma}
\begin{proof}
Suppose that $\nu_\infi$ contains an irrational sublamination. 
Then, there is a minimal irrational sublamination $N$ of $L$, so that every leaf of $N$ is dense in $N$.
Let $F$ be a (topologically) smallest subsurface of $S$ containing $N$, such that $F \sub N$ is a $\pi_1$-injective.
Let $\ell$ be a geodesic loop in $\sigma_\infi$ which is a good approximation of $N$. 
Let $\beta_\infi\col \H^2 \to \H^3$ be the equivariant pleated surface corresponding to $ (\sigma_\infi, \nu_\infi)$.
Then, for each  component $R$ of $F \minus \ell$,  the restriction of $\beta_\infi$ to $R$ is a quasi-isometric embedding (\cite{Baba-10}). 
Thus $\rho_\infi | \pi_1 R$ is non-elementary, immediately implying the lemma. 
\end{proof}

Using the assumption that $C_t$ is pinched along a single loop, we prove the following: 
\begin{proposition}\Label{NonelementaryComplements}
For each component $F$ of $S \minus m$,
the restriction of $\rho_\infi$ to $\pi_1 F$ is a non-trivial representation in the representation variety. 
\end{proposition}
\begin{remark}
On the other hand, the restriction $\rho_\infty \vert \pi_1(F)$ may be the trivial representation in the character variety (see \Cref{ExoticDegenerationHyperbolic}).
\end{remark}
\begin{proof}
If $\nu_\infi$ contains an irrational lamination, by Lemma \ref{IrrationalThenNonelementary},  $\rho_\infi$ is non-elementary.  
Then we can assume, without loss of generality, that $\nu_\infi$ contains only isolated leaves, and $\nu_\infi$ divides $\sigma_\infi$ into ideal polygons. 

Since each component of $\sigma_\infi$ has one or two cusps, there is a leaf $\ell$ of $\nu_\infi$ whose endpoints are at a single cusp $c$ of $\sigma_\infi$. 
Let $D$ be a small horodisk quotient neighborhood of $c$.
Then $\ell \minus D$ is a long geodesic segment, and by connecting its endpoints by a horocyclic simple arc in $\bdr D$, we obtained a simple loop $\gam$ (which is a good approximation of $\ell$); see \Cref{fNontrivialHolonomy} (Left).

Pick a lift $\ti\ell$ of $\ell$ to the universal cover $\H^2$ of $\sigma_\infty$ and fix an orientation. 
Then there is  $\alpha_\gam \in \pi_1(S)$ representing $\gam$ which takes the oriented (bi-infinite) geodesic $\ti\ell$ to an oriented geodesic starting from the endpoint of $\ti\ell$; see \Cref{fNontrivialHolonomy} (Right).
Clearly $\beta_\infi(\ti\ell)$ is an oriented geodesic in $\H^3$.
Then, by the equivariant proeprty, the holonomy along $\alpha$ takes the oriented geodesic ${\beta}_\infi(\ti\ell)$ to an oriented geodesic starting from the endpoint of ${\beta}_\infi(\ti\ell)$,  and thus $\rho_\infi(\gamma) \neq I$.
\begin{figure}[H]
\begin{overpic}[scale=.15,
] {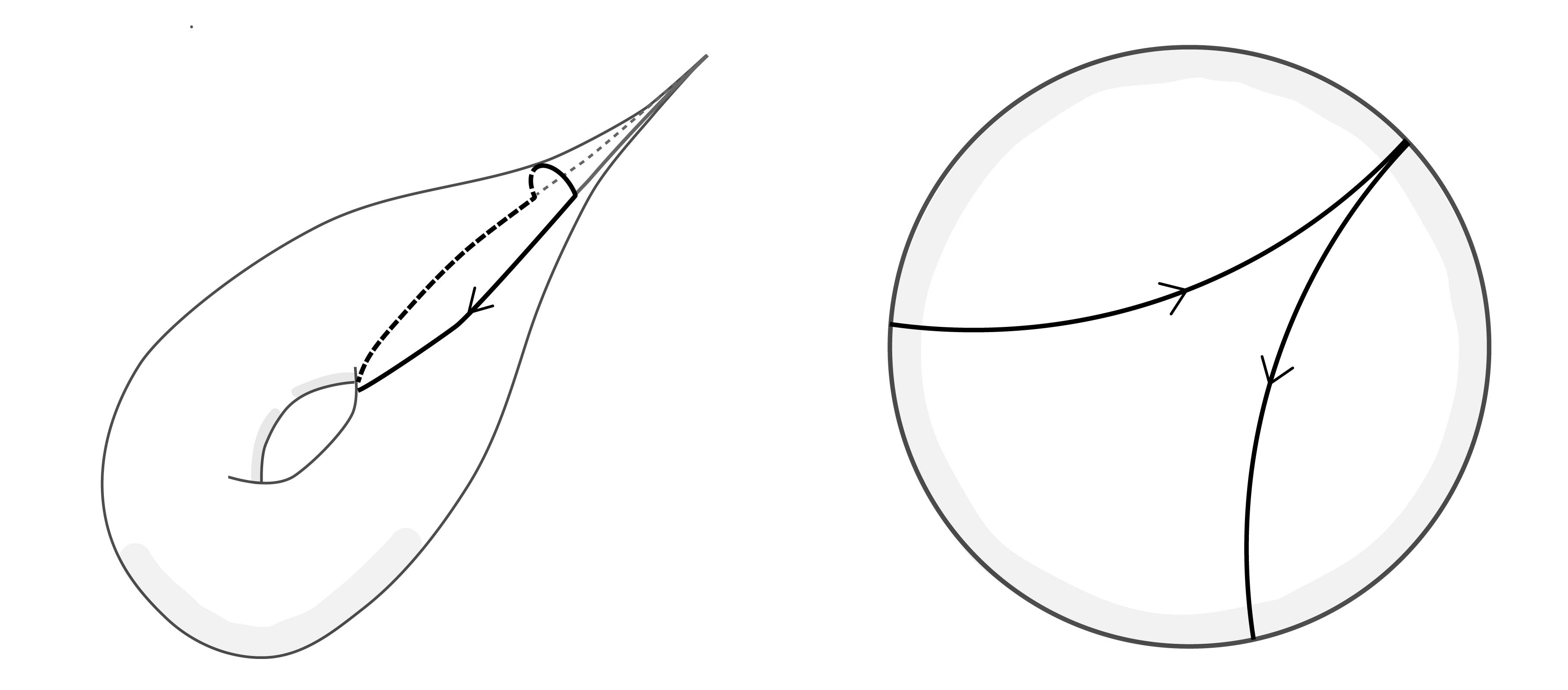} 
    \put( 10, 10){\textcolor{darkgray}{$\sigma_\infty$}}  
    \put( 26, 18){\textcolor{Black}{$\gamma$}}  
       \put(70, 18){\textcolor{Black}{$\ti\ell$}}  
              \put(81, 11){\textcolor{Black}{$\alpha_\gam\,\ti\ell$}}  
    \put(68, 10){\textcolor{darkgray}{$\H^2$}}  
      \end{overpic}
\caption{}\label{fNontrivialHolonomy}
\end{figure}
\end{proof}

\begin{lemma}\Label{ContinuousStabilizerFixedPoint}
Let $G$ be a non-trivial subgroup of $\PSL_2\C$.
Consider the (pointwise) stabilizer of the action $\PSL_2\C \curvearrowright G$ by conjugation. 
Suppose that the stabilizer is continuous. 
Then there is a set $\Lambda$ of one or two points of $\CP^1$ fixed pointwise by the action of $G$. 
\end{lemma}
\begin{proof}
Suppose that $G$ has a continuous stabilizer. Then, clearly,  $G$ is an elementary subgroup of $\PSL_2\C$. 
First suppose, in addition, that $G$ contains a hyperbolic element $h$. 
Then no element in $G$ exchanges the fixed points of $h$, as otherwise, the stabilizer cannot be continuous. 
Therefore $\Lambda$ is the fixed point set of $h$, and all elements in $G \minus \{I\}$ must be hyperbolic or elliptic elements with the same axis. 

Next suppose that $G$ contains a parabolic element $p$.
Then there is no elliptic element or hyperbolic element in $G$, as otherwise, the stabilizer cannot be continuous. 
Then $\Lambda$ must be the single fixed point of $p$, and all $G \minus \{I\}$ are all parabolic elements with the same fixed point.

Suppose that $G$ contains an elliptic element $e$ and contains no hyperbolic element. 
Then, similarly,  $\Lambda$ must be the fixed point set of $e$, and $G$ contains no parabolic element.
Moreover $G \minus \{I\}$ are all elliptic elements with a common axis. 
Then  $\Lambda$ is the set of the two endpoints of the axis.
\end{proof}
Given a $\CP^1$-surface with a cusp such that the holonomy around the cusp is trivial, its developing map continuously extends to the cusp, so that it is a branched covering map near the cusp.
\begin{lemma}\Label{PunctureNotAtAxis}
Let $F$ be a compact surface with finitely many punctures, such that the Euler characteristic of $F$ is negative.
Let $(f, \rho)$ be a developing pair of a $\CP^1$-structure $C$ on $F$ such that  
\begin{itemize}
\item $\rho\col \pi_1(F) \to \PSL_2\C$ is not the trivial representation, 
\item the holonomy around each puncture is trivial, and 
\item  the stabilizer of $\Im \,\rho$ in $\PSL_2\C$ is continuous;  thus let $\Lambda \sub \CP^1$ be the one- or two-point set in Lemma \ref{ContinuousStabilizerFixedPoint}.
\end{itemize}
Then, there is a cusp $p$ of $F$ such that $f(p)$ is not a point of  $\Lambda$. 
\end{lemma}

\begin{proof}
Notice that $\CP^1$ minus $\Lambda$ admits a complete Euclidean metric invariant under $\Im \rho$, which is unique up to scaling.
Thus, if $f$ takes all cusps of $F$ into $\Lambda$, then the surface $F$ minus finitely many points admits a complete Euclidean metric. 
This is a contradiction as the Euler characteristic of $F$ is negative.
\end{proof}

The next proposition immediately follows from \Cref{PunctureNotfixed}.
\begin{proposition}\Label{PunctureElementaryHolonomy} 
Let $F$ be a compact connected surface with two punctures, such that the Euler characteristic of $F$ is negative. 
Let $C = (f, \rho)$ be a $\CP^1$-structure on $F$, such that 
\begin{itemize}
\item $\Im\, \rho$ has a continuous stabilizer in $\PSL_2\C$;
\item  the holonomy around each puncture is trivial;
\item the degrees around the two punctures are the same.
\end{itemize}
Then no cusp of $C$ maps to a point of $\Lambda$  by $f$, where $\Lambda$ is as in Lemma \ref{ContinuousStabilizerFixedPoint}.
\end{proposition}

Let $\til{m}$ be a lift of $m$ to $\til{S}$. 
Let $Q$ and $R$ be the adjacent components of $\til{S} \minus \phi^{-1}(m)$ across $\til{m}$.
Let $\operatorname{Stab}{Q}$
 and $\operatorname{Stab}{R}$ denote the subgroups in $\pi_1(S)$ which setwise preserve $Q$ and $R$, respectively. 
 Let $C_\infty^Q, C_\infty^R$ denote the component of $C_\infty$ corresponding $Q, R$ (if $m$ is non-separating, $C_\infty^Q = C_\infty^R$). 
 
We first prove (\ref{iNotExactlyI}) in Theorem \ref{TrivialNeck}.
\begin{lemma}\Label{m-neq-I}
For sufficiently large $t > 0$, $\rho_{t}(m) \neq I$.
\end{lemma}
\begin{proof}
Suppose, to the contrary, that there is a diverging sequence $0 \leq t_1 < t_2 < \dots$ such that $\rho_{t_i}(m) = I$ for each $i$. 
We may, in addition, assume that $C_{t_i}$ converges to $C_\infi$ as $i \to \infi$ uniformly on compact subsets as $i \to \infi$. 
Then, as $\rho_{t_i} (m) = I$ and $C_t$ is pinched along $m$,
for $i \gg 0$, there is a cylinder $A_i$  in $C_{t_i}$ homotopic to $m$ such that 
\begin{itemize}
\item  $A_i$ is bounded by round circles (i.e. the development of each boundary component is a round circle on $\CP^1$), 
\item $\Mod A_i \to \infty$, and
\item $C_{t_i} \minus A_i$ converges to $C_\infi$ minus cusp neighborhoods bounded by round circles (in other words, for every $\ep > 0$, if $i$ is sufficiently large, then $A_i$ is contained in $\ep$-thin part of $C_{t_i}$ ).
\end{itemize}
We can normalize $\rho_{t_i}$ so that $\rho_{t_i} \vert \Stab R$ converges as $i \to \infty$ and the developing map $f_{t_i} \vert R$ also converges  to a developing map of $C_\infty^R$ as $i \to \infty$. 
Then the development of $\ti{m}$ converges to a point $p$ on $\CP^1$.

First suppose that the stabilizer of $\rho_\infty \vert \Stab\, Q$ is discrete. 
Then, there are elements $\alpha_1, \alpha_2$ of $\Stab \, Q$ with disjoint fixed point sets on $\CP^1$. 
Pick a sequence $\gam_i \in \PSL_2\C$ such that the restriction of  the conjugation $\gam_i \rho_{t_i} \gam^{-1}_i \eqqcolon \rho_{t_i}'$ to $\Stab Q$ converges as $i \to \infty$. 
 Therefore, the properties of $A_i$ imply that  $\gam_i$ must leave every compact in $\PSL_2\C$.
 As  $\alpha_1, \alpha_2$ have disjoint fixed point sets in $\CP^1$, one of the fixed point sets does not contain the puncture point of $C^Q_\infty.$
 Therefore either $\rho_{t_i} (\alpha_1)$ or  $\rho_{t_i} (\alpha_2)$ diverges to infinity in $\PSL_2\C$ as $i \to \infty$ against the hypothesis. 

Next suppose that the stabilizer of $\rho_\infty \vert \Stab\, Q$ is continuous. 
Then,  by Proposition \ref{PunctureElementaryHolonomy} and Lemma \ref{PunctureNotAtAxis}, with respect to the normalization $\rho_{t_i}'$, no cusp of $C_\infty^Q$ develops to a point of $\Lambda$ for $C_\infty^Q$.
Let $\omega \in \Stab Q$ such that $\rho_\infty(w)$ is non-trivial (\Cref{NonelementaryComplements}). 
Then, by the properties of $A_i$, 
  $\rho_{t_i}  (\omega)$ must diverges to $\infty$ since the continuous stabilizer preserves $\Lambda$.

This is a contradiction against the convergence of $\rho_t$. 
\end{proof}
\begin{lemma}\Label{IDivergingTwist}
The Fenchel-Nielsen twist coordinate along $m$ must diverge to $\infty$ or $-\infty$ as $t \to \infty$.  
\end{lemma} 
\begin{proof}
The proof is similar to that of Lemma \ref{m-neq-I}. 
    Suppose to the contrary that there is a sequence $t_1 < t_2 < t_3 < \dots$ such that the Fenchel-Nielsen twist parameter of $C_{t_i}$ along $m$ converges as $i \to \infty$. 
    We normalize $\rho_{t_i}$ so that $\rho_{t_i} \vert \Stab R$ converges as $i \to \infty$ the developing map $f_{t_i} \vert R$ also converges to a developing map of $C_\infty^R$ as $i \to \infty$.
    Then, similarly to the proof of Lemma \ref{m-neq-I},  one can show that $\rho_{t_i} \vert \Stab \,Q$ diverges to infinity, since the cylinder $A_i$ becomes longer and longer and it pushes  $\rho_{t_i} \vert \Stab \,Q$ farther and farther away;  this contradicts the convergence of $\rho_t$ as $t \to \infty$.
\end{proof}

Then, for each $t > 0$,  let $\iota_t $ be some power of the Dehn twist of $S$ along $m$ such that the twist coordinates of $\iota_t C_t$ along $m$ is bounded from above and below in $\R$ uniformly in $t > 0$.
Then,  by Lemma \ref{IDivergingTwist},  the power must diverge to either $\infty$ or $- \infty$ as $t \to \infi$.

There is a diverging sequence $0 \leq t_1 < t_2 < \dots$ such that $C_{t_i} \to C_\infi$ as $i \to \infi$ uniformly on compact.
Let $F$ be a component of $S \minus m$.
Let $\ti{F}$ be the universal cover of $F$.

First suppose that $\rho_\infi | F$ has a discrete stabilizer (in $\PSL_2\C$). 
Let $F_\infi$ be a component of $C_\infi$ which corresponds to $F$. 
Then $\dev F_\infi$ is the limit of $f_{t_i} | \til{F}$, so that $\lim_{i \to \infi} f_{t_i}$ takes each boundary component of $\ti{F}$ to a single point corresponding to a cusp of $C_\infi$.

   Pick a fundamental domain $D_i$ in $\ti{F}$ with an arc $s_i$ on $\bdr D_i \cap \bdr \ti{F}$ such that 
   $s_i$ descends to a loop $m_i$ isotopic  to $m$, the loop $m_i$ is contained in the $\ep_i$-thin part of $C_{t_i}$ with $\ep_i \searrow 0$ as $i \to \infi$, and the development of $m_i$ is invariant under a one-dimensional subgroup $G_i$ of $\PSL_2\C$ containing $\rho_i (m)$. 
   As $\rho_{t_i} (m) \to I$,  the image of $s_i$ becomes more and more like a round circle $c_i$ as $i \to \infi$. 

Next suppose that $\rho_\infi(F)$ has a continuous stabilizer.
Then $\rho_\infi(F)$ is elementary, and the restriction of $f_{t_i}$ to $\ti{F}$ may not converge to a local homeomorphism, even up to a subsequence. 
Nonetheless, as $C_{t_i}$ converges to $C_\infty$ in $\PP(S \minus m)$,  clearly we can normalize $\rho_{t_i}$ for  the convergence of developing pairs: 
\begin{lemma}\Label{NormalizingAndredegenerating}
 Suppose that there is no subsequence of $t_i$ such that $f_{t_i}| \ti{F}$ converges to a developing map of $F_\infi$.
 Then there is a sequence $\gam_i$ of $\PSL_2\C$  such that, up to a subsequence,  $\gam_i (f_{t_i} | \ti{F}, \rho_{t_i} | \pi_1 F)$ converges to a developing pair $(h_\infi, \zeta_\infi)$ of $F_\infi$. 
 \end{lemma}

Next, without normalization,  we show a convergence of the developing map as a continuous map. 
\begin{proposition}\Label{DevConvergesSubsurfaces}
Suppose that there is no subsequence such that the restriction $f_{t_i} | \ti{F}$ converges to a developing map of $F_\infi$ as $i \to \infty$.
Then $f_{t_i} | \ti{F}$ converges to a $\rho_\infi |\pi_1 F$-equivariant continuous map $f_{F, \infi}\col \ti{F} \to \CP^1$ uniformly on compact subsets, such that each boundary component of $\ti{F}$ maps to a single point. 
Moreover,  either $f_{F, \infi}$ is a constant map to a fixed point of $\rho_\infi | F$ or there are open disks $D_1, \dots, D_n$ on $F$ such that $f_{F, \infi}$ takes $\til{F} \minus \phi^{-1}(D_1 \sqcup \dots \sqcup D_n)$ to a fixed point $p$ of $\rho_\infi(F)$  and  each lift $\ti{D}_i$ of $D_i$ to $\CP^1 \minus \{p\}$ for all $i = 1, \dots, n$. 
\end{proposition}
\begin{proof}
Let $\gam_i \in \PSL_2\C$ be the sequence and $(h_\infty, \zeta_\infty)$ be the normalized limit obtained by Lemma \ref{NormalizingAndredegenerating}.
By the non-subconvergence hypothesis, $\rho_\infty (\pi_1 F)$ is an elementary representation. 
We divide the proof into cases depending on the types of elementary subgroups.

First suppose that $\rho_\infi(\pi_1 F)$ contains a loxodromic or elliptic element. 
Then, let $\ell$ be the axis of the loxodromic or the elliptic element. 
Then, there is a corresponding loxodromic or elliptic element in $\Im h_\infty$, and let $\ell'$ be its axis. 
By the non-subconvergence hypothesis,  there is $\omega \in \pi_1 F$ such that $h_\infty (\omega)$ is a parabolic element but $\rho_\infty (\omega)$ is the identity in $\PSL_2\C$.
Thus
$\gam_i$ must be a hyperbolic element for sufficiently large $i$ such that as $i \to \infty$, the translation length of $\gam_i$ diverges to infinity. 
 In addition,  $\Axis(\gam_i)$ converges to the  $\ell'$ in $\H^3$. 
Let $p$ be the limit of the repelling fixed point of $\gam_i$, and let $q$ be the limit of the attracting fixed point of $\gam_i$, so that $\{p, q\}$ are the endpoints of $\ell'$. 
Note that as $\rho_\infty (\pi_1(F))$ is elementary, $\rho_\infty \pi_1(F)$ preserves $p$ and $q$ point-wise.

Take a connected compact fundamental domain $Q$ in $\ti{F}$.
We can assume that $Q \cap \bdr \ti{F}$ is disjoint from $q$, by perturbing the loop $m_i$ on $C_{t_i}$ if necessary.
For simplicity, we first suppose that $h_\infi(Q)$ is disjoint from $q$. 
Then, letting $f_i = f_{t_i}$,  the restriction $f_i | Q$ converges to the constant map to $p$ uniformly, as $i \to \infi$,  and thus $f_i\col \ti{F} \to \CP^1$ converges to the constant map to $p$ uniformly on compact subsets. 

Suppose that $h_\infi(Q) \cap \{q\} \neq \emptyset$.
Then, by the compactness of $Q$,  there are finitely many points of $h_\infi^{-1}(q)$ in the interior of $Q$. 
Pick small disjoint open disk neighborhoods of the points in $h_\infi^{-1}(q)$ in $Q$. 
Then, as the disks are contained in a fundamental domain,   their images $D_1, \dots, D_n$ in $F$ are disjoint. 
Then, as $\zeta_\infty$ preserves $q$,  the restriction of $f_i$ to $\ti{F} \minus \phi^{-1} (D_1 \sqcup \dots \sqcup D_n)$ converges to the constant map  to $p$ uniformly on compact subsets.
Moreover, for each lift $\ti{D}_i$ of $D_i$ to $\ti{F}$, $\ti{D}_i$ contains a unique point mapping to $q$. 
Thus up to an isotopy of $S$, we can in addition assume that $f_i |  D_i$ converges to a homeomorphism to $\CP^1 \minus \{p\}$, as desired. 
By Lemma \ref{PunctureAtDomain} and Proposition \ref{PunctureNotfixed}, the boundary components of $\ti{F}$ all map to $p$. 

Next, suppose that $\rho_\infi(F)$ contains a (non-trivial) parabolic element but no hyperbolic and elliptic element. 
Let $\omega \in \pi_1 F$ such that $\rho_\infty (\omega)$ is also a non-trivial parabolic element. 
Therefore $\rho_\infty$ and $\rho_\infty'$ are conjugate to each other, and $(f_{t_i}, \rho_{t_i} \vert \pi_1 F)$ converges to a developing pair of $F_\infty$.
This contradicts the non-subconvergence hypothesis.  
   
Last, suppose that $\rho_\infi(\pi_1 F)$ is the trivial representation. 
This case will be similar to the case when $\rho_\infty (\pi_1 F)$ contains an elliptic or a hyperbolic element. 
Then the normalized holonomy  $\zeta_\infi$ is a parabolic representation. 
Let $p$ be the parabolic fixed point of $\zeta_\infi$. 
We can assume that $\gam_i$ is a hyperbolic element for $i$ large, and the axis of $\gam_i$ converges to a geodesic $\ell$ starting from $p$. 
Let $q$ be the other endpoint of $\ell$.
Pick a  connected fundamental domain $Q$ in $\ti{F}$ so that $h^{-1}(p)$ is disjoint from  $\bdr Q$. 
Suppose in addition that no point of $Q$ maps to $p$. 
Then, up to a subsequence,  $f_i | \ti{F}$ converges to a constant map to $q$ uniformly on compact subsets. 
Suppose there are (finitely many) points of $Q$ which map to $p$. 
Then, similarly to the case of a hyperbolic and an elliptic representation, take disjoint open ball neighborhoods of those points in $Q$, and let $D_1, D_2 \dots D_n$ be disjoint disks on $F$ which lift to those open balls.
Then the desired convergence follows similarly.  
\end{proof}

By Proposition \ref{DevConvergesSubsurfaces}, the restriction of $f_i$ to $\ti{S} \minus \phi^{-1} (A)$ converges to a $\rho_\infi$-equivariant map $f_\infi \col \ti{S} \minus \phi^{-1}(A) \to \CP^1$.
We next prove the convergence of the boundary components to complete the proof of (\ref{iBoundaryComponentAtEndpoints}).
\begin{proposition}\Label{PunctureAtAxis}
For each component $\ti{A}$ of $\phi^{-1}(A)$, let $\gam \in \pi_1(S)$ be the representative of $m$ preserving $\ti{A}$.
 Then, by taking a subsequence so that  $\Axis(\rho_{t_i}(\gam) ) \eqqcolon a_i$ converges to a subset $a_\infi \in \H^3 \cup \CP^1$, which is either a point on $\CP^1$ or a geodesic in $\H^3$, then $f_\infi$ takes the boundary components of $\ti{A}$  onto the ideal points of $a_\infi$.
\end{proposition}
  
\proof
By Lemma \ref{m-neq-I}, $\rho_{t_i}(m) \neq I$ for sufficiently large $i \in \Z_{> 0}$.
Thus, by taking a subsequence,  we may in addition assume that $\rho_{t_i}(\gam)$ converges to $I$ tangentially to a unit tangent vector of $\PSL_2\C$ at $I$.  
Let $G_i$ be the one-parameter subgroup of $\PSL_2\C$ which contains $\rho_{t_i} (\gam)$, such that the cyclic group generated by $ \rho_{t_i} (\gam)$ is asymptotically  dense in $G_i$  with respect to the intrinsic metric on $G_i$.
Then the trajectories of  $G_i$ yields a unique foliation of $\H^3$ except that,  if $\rho_{t_i} (\gam)$ is elliptic, only  of $\H
^3 \minus a_i$ (Figure \ref{fInvariantFoliationOfH}). 
\begin{figure}
\begin{overpic}[scale=.3
] {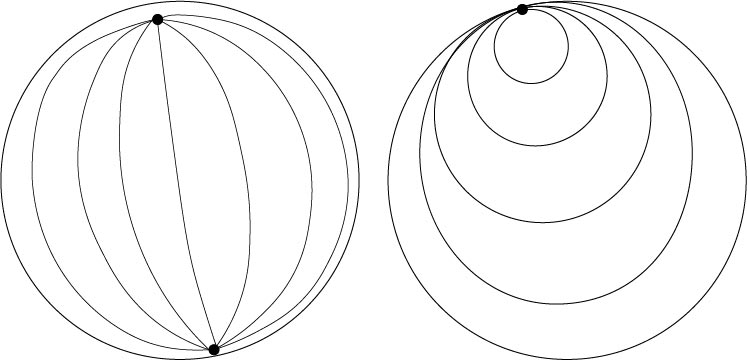} 
      \end{overpic}
\caption{}\Label{fInvariantFoliationOfH}
\end{figure}
We have chosen a subsequence $t_i$ so that  $C_{t_i} \to C_\infi$ uniformly on every thick part and the axis $a_i$ converges to a closed subset $a_\infi$ of $\ol{\H^3}$.
Let $P, Q$ be the components of $\ti{S} \minus \phi^{-1}(A)$ adjacent across $\ti{A}$. 
\begin{claim}\Label{PuncturesAtAxes}
Let $\ell$ be the common boundary component of $P$ and $\ti{A}$. 
Suppose, to the contrary, that $\lim f_{t_i}(\ell)$ is not a point, in $\CP^1$, of the limit axis $a_\infi$.
Then $\rho_{t_i} | Q$ diverges to $\infty$ in $\rchi$.
\end{claim}

\begin{proof}
Let $\iota_i$ be some power of the Dehn twist of $S$ along $m$ so that the Fenchel-Nielsen twist parameter of the remarked Riemann surface $\iota_i X_{t_i}$ along $m$ is bounded from above and below uniformly in $i$.

Let $\ell'$ be the common boundary component of $\ti{A}$ and $Q$. 
By Proposition \ref{NonelementaryComplements}, there is $\gamma \in \pi_1 (S)$ belonging to $\Stab Q$ such that $\rho_\infi  (\gamma)$ is {\it not} the identity matrix.
We may in addition assume the axis of  $\rho_{t_i}(\iota_i\cdot \gamma)$ converges to the point $f_{F, \infi}(\ell)$ on $\CP^1$  (if $\rho_\infi(\Stab Q)$  is elementary, use Lemma \ref{PunctureNotAtAxis} and Proposition \ref{PunctureElementaryHolonomy}).
By the tangential convergence of $\rho_i(m) \to I$, the $G_i$-invariant foliation $\mathcal{F}_i$ of $\H^3$converges to a foliation $\mathcal{F}_\infi$ of $\H^3$. 
If $f_{P, \infi}(\ell)$ is not the ideal point of $a_\infi$, $\Axis(\rho_i (\iota_i \cdot \gamma ))$ be eventually disjoint from every compact subset in the space of the leaves of $\FFF_\infi$.
Therefore, since the $G_i$-invariant  foliations $\mathcal{F}_i$ limit to  $\FFF_\infty$, $\Axis(\rho_i(\gamma))$ also leaves every compact subset of the leaf space of $\FFF_\infi$ . 
 Hence $\rho_i(\gamma)$ must diverge to $\infty$ in $\PSL_2\C$, which is a contradiction. 
(Figure \ref{fPunctureAndAxis}.)
\end{proof}
This claim completes the proof.
\Qed{PunctureAtAxis}

 It remains only to prove the surjectivity in \Cref{TrivialNeck}
(\ref{iBoundaryComponentAtEndpoints}):
\begin{lemma}\Label{DifferentEndpoints}
Suppose that $a_\infi$ is a geodesic in $\H^3$. 
Then $f_\infi(\ell)$ and $f_\infi(\ell')$ are the different endpoints of $a_\infi$.
\end{lemma} 

\begin{proof}
By Claim \ref{PuncturesAtAxes}.
$f_i  | \ell$ converges to the constant map to an endpoint of $a_\infi$. 

Let $n_i \in \Z$ be the power of the Dehn twist along $m$ which gives  $\iota_i \in \MCG(S)$. 
Thus $\rho_i(\gam^{n_i})$ is a hyperbolic element whose axis $a_i$ converges to $a_\infi$, and its translation length diverges to infinity as $i \to \infty$. 
Then the attracting fixed point of $\rho_i(\gam^{n_i})$ converges to the endpoint of $a_\infi$ which is not $f_\infi(\ell)$. 
Thus $f_\infi(\ell')$ must be at the other endpoint. 
\end{proof}

\begin{figure}
\vspace{5mm}
\begin{overpic}[scale=.35,
] {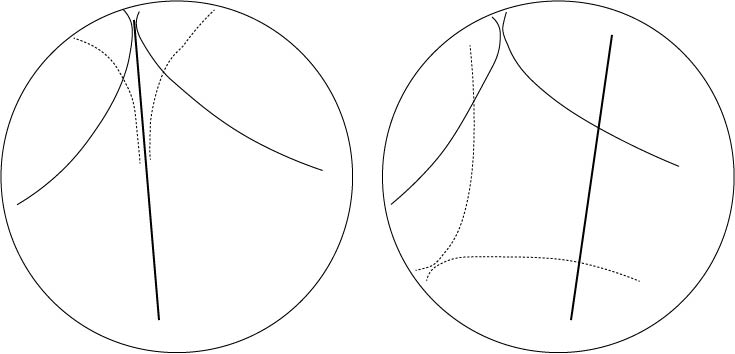} 
   \put(18 ,50 ){\textcolor{black}{\tiny The axis at the puncture.}}  
    \put(65 ,50 ){\textcolor{black}{\tiny The axis off the puncture; diverging.}}  
      \end{overpic}
\caption{}\Label{fPunctureAndAxis}
\end{figure}

\begingroup
 \color{RubineRed}
    \fontsize{12pt}{12pt}\selectfont

   \endgroup
\Qed{PunctureAtAxis}
\subsubsection{Convergence of pleated surfaces when $\rho_\infi(m) = I$}\Label{smIPleatedSurface}
First we compare developing maps of $\CP^1$-structures and the exponential map $\exp \col \C \to \C^\ast$. 
Let $\ell$ be the geodesic in $\H^3$ connecting $0$ to $\infty$ of $\bdr \H^3 = \C \cup \{\infty\}$. 
Let $\Psi \col \C^\ast \to \ell$ be the continuous extension of the nearest point projection $\H^3 \to \ell$.
Then, the composition is $\Psi \cc \exp\col \C \to \H^3$ is the Epstein map of the $\CP^1$-structure on $\C$ given by $\exp$. 

Recall that, given a $\CP^1$-structure $C = (X, q)$,  for $x \in \C$, $~d(x)$ is the Euclidean distance from $x$ to the set of the zeros of the holomorphic differential $q$.
Note that, if $d(x)$ is large, then we can naturally embed a large neighborhood of $x$ into $\C (\cong \E^2)$ by an isometric map onto its image, so that vertical leaves map into horizontal lines, and horizontal leaves map into vertical lines. 
\begin{proposition}\Label{EpsteinMapAndBendingMap} 
For every $\ep > 0$, there is $R > 0$, such that,  if $x \in C$ satisfies $d(x) > R$,  then the Epstein map $\Sigma\col \ti{C} \cong \til{S} \to \H^3$ is  $\ep$-close, in the $C^1$-topology, to the composition of the collapsing map $\ti\kap\col\til{S} \to \H^2$ and  the bending map $\beta \col \H^2 \to \H^3$ at every lift $\ti{x}$ of $x$.  
\end{proposition}
\proof[Proof of Proposition \ref{EpsteinMapAndBendingMap}]

\begin{lemma}
For every $\ep > 0$, there is $R > 0$, such that, if  $z \in \til{C}$  satisfies $d(z) > R$, then the maximal ball centered at $z$  is $\ep$-close to the maximal ball of the corresponding exponential map.
\end{lemma}

\begin{proof}
As the Epstein map
 of $C$ and $\exp$ are close, their developing maps are also close. 
This implies the closeness of their maximal balls centered at $z$ and their ideal points. 
\end{proof}
The proposition follows from the above lemma, and \Cref{AlmostExponentialMap}.
\Qed{EpsteinMapAndBendingMap}

Recall that we have already proved Theorem \ref{TrivialNeck} (\ref{iNotExactlyI}), (\ref{iTwistDieverges}), (\ref{iSubsequencialConvergence}) regarding the asymptotic behavior of $C_t$ using the decomposition of $C_t$ into the restriction of $C_t$ to the thin part $A$ and its complement. 
We prove additional compatibility of the corresponding bending map. 
\begin{proposition}\Label{TrivialNeckPleatedSurface}
Suppose that $\rho_\infi(m) = I$. 
Then for every diverging sequence $t_1 < t_2 < \dots$, up to taking a subsequence,  there are a sequence of diffeomorphisms $\iota_i \col S \to \tau_{t_i}$ representing the marking of $C_{t_i}$ and a path of cylinders $A_i$  in $C_{t_i}$ homotopy equivalent to $m$, such that 
in addition to Theorem \ref{TrivialNeck} (\ref{iNotExactlyI}), (\ref{iTwistDieverges}), (\ref{iSubsequencialConvergence}), the following holds:
\begin{enumerate}
\item  $A$ maps to $A_i$  by $\iota_i$; \Label{iNPreserved}
\item $\beta_{t_i}\cc \ti{\kap}_{t_i}\col \ti{S} \to \H^3$ converges to a $\rho_\infi$-equivariant continuous map $\ti{S} \to \H^3 \cup \CP^1$ uniformly on compact subsets; 
\item 
for each connected component $F$ of $\ti{S} \minus \phi^{-1}(A)$, the restriction of $\beta_{t_i}\circ \kappa_{t_i}$ to $F$ converges to the pleated surface of the corresponding component of $C_\infi$ or the constant map to an ideal point of  $a_\infi$ (in \Cref{PunctureAtAxis}); \Label{iThickPartPleastedSurface}

\item letting $\til{A}$ be a connected component of  $\phi^{-1}(A)$  in $\til{S}$, then $\beta_{t_i} \cc \ti\kap_{t_i}| \til{A}$ converges  to a map onto $a_\infi$
uniformly on compact subsets in $\til{A}$ with respect to a fixed closed disk metric on $\H^3 \cup \CP^1$. \Label{iCylinderToAxis}
\end{enumerate}
\end{proposition}
\proof
For $t \gg 0$, there is a one-parameter family of loops homotopic to $m$ such that their developments are invariant under a unique one-dimensional subgroup $G_t$ of $\PSL_2\C$ which contains $\rho_t(m)$ (as in the proof of \Cref{DevConvergesSubsurfaces}). 
Then we can pick a cylinder $A_t$ in $C_t$ homotopy equivalent to $m$, such that 
\begin{itemize}
\item $A_t$ is foliated by loops whose developments are invariant under $G_t$ for each $t \gg 0$, 
\item $C_t \minus A_t$ converges to $C_\infi$ as $t \to \infi$, and 
\item $\Mod A_t \to \infi$ as $t \to \infi$.
\end{itemize}
By the second property, $A_t$ is contained in a thinner and thinner part of $C_t$ as $t \to \infty$. 
Then, the developing map of $A_t$ is the restriction of $\exp\col \C \to \C^\ast$ to a {\it bi-infinite strip} $T_t$, i.e. a region in $\C$ bounded by a pair of parallel lines. 
Then its deck transformation group ($ \cong \Z$) is generated by a translation of $T_t$.
Then  $A_t$ has a natural Euclidean metric by identifying $\C$ with $\E^2$.

Recall that $A$ is a cylinder in $S$ homotopic to $m$, and fix a finite volume Euclidean structure on $A$ with geodesic boundary  (by picking a homeomorphism $A \to \s^1 \times [-1,1]$). 
We can easily pick a marking $\iota_t \col S \to C_t$ such that 
\begin{itemize}
\item $\iota_t$ takes $A$ to $A_t$ ((\ref{iNPreserved}));
\item the restriction of $C_t$ to $\iota_t(S \minus A)$ converges to $C_\infi$;
\item $\iota_t | A $ is linear with respect to the Euclidean structures on $A$ and $A_t$. 
\end{itemize}

Given a component $F$ of $\ti{S} \minus \phi^{-1} (A)$, suppose that $f_{t_i} |F$ converges to a developing map of the component of $C_\infi$. 
Then, clearly $\beta_i \cc \ti\kap_i  | F$ converges to a pleated surface for the corresponding component of $C_\infi$.
By Proposition \ref{DevConvergesSubsurfaces}, if  $f_{t_i} |F$ does not converge to a developing map, then $\beta_i \cc \ti\kap_i | F$ converges to the constant map to an ideal point of the axis limit $a_\infty$.  
Thus we have  (\ref{iThickPartPleastedSurface}).

Last we prove (\ref{iCylinderToAxis}).
As we have already shown the convergence of the developing map in the thick part, we need to show that the convergence extends to the convergence on the neck. 
As the developing map of some components of $S \minus m$ may degenerate as described in \Cref{DevConvergesSubsurfaces}, accordingly one needs to be careful about the behavior of $\beta_i \circ \kap_i$  on the neck.

By Theorem \ref{TrivialNeck}(\ref{iTwistDieverges}), the Fenchel-Nielsen twisting parameter of $C_t$ along $m$ diverges to either $\infty$ or $- \infty$ as $t \to \infty$. 
We can assume that the twisting of $C_t$ along $m$ occurs in $A_t$ by isotopy of $S$.

(Case One) Suppose that $a_\infi :=\lim_{i \to \infi} \Axis \rho_{t_i}(m)$ is a bi-infinite geodesic. 
Then $\rho_{t_i} (m)$ is hyperbolic if $i$ is large enough, and the translation length of $\rho_{t_i} (m)$ times the number of twist goes to infinity as $i \to \infty$.
For $r > 0$, let $U_i(r)$ be the $r$-neighborhood of $a_i$ in $\H^3$. 
Clearly $U_i(r)$ is invariant under $\rho_i(m)$. 
Let $(\tau_i, L_i) \in \TT \times \ML$ be the Thurston parameters of $C_i$ for each $i$. 
Pick $\ep > 0$ less than the Bers' constant, and let $N_i = N_i^{\ep}$ be the $\ep$-thin part of $\tau_i$. 
Let $\ti{N_i}$ be the lift of $N_i$ to the universal cover $\H^2$ of $\tau_i$ invariant under the fixed representative the loop $m$ in $\pi_1(S)$. 
Let $\ell_{i, 1}, \ell_{i, 2}$ denote the boundary components of $\ti{N}_i$, which connect the endpoints of the geodesic $a_i$
\begin{lemma}\Label{ThinPartImageBoundedFromGeodesic}
If $r > 0$ is sufficiently large, then $\beta_i(\ti{N}_i)$ is contained in $U_i(r)$ for sufficiently large $i$.
\end{lemma}
\begin{proof}
Let $\ti{A}$ be the lift of $A$ to $\ti{S}$ which is invariant under $m \in \pi_1(S)$. 
Let $P_1$ and $P_2$ be the components of $\ti{S} \minus \phi^{-1}(A)$ adjacent across a lift $\ti{A}$. 
Suppose, to the contrary, that for every $r > 0$, the image $\beta_i(\ti{N}_i)$ is {\it not} eventually contained in $U_i$ as $i \to \infty$.
Then, either
\begin{enumerate}[(i)]
\item for every $r > 0$, if $i$ is sufficiently large, then $\beta_i(\ell_{i, 1})$ and $\beta_i(\ell_{i, 2})$ are both not contained in $U_i$, or
\item for every large $r > 0$, if $i$ is sufficiently large, then one of $\beta_i(\ell_{i, 1})$ and $\beta_i(\ell_{i, 2})$ is contained in $U_i$ but the other is not. 
\end{enumerate}
First, suppose (i).  
Then, let $\phi_i \col \H^2 \to \tau_i$ be the universal covering map. 
Let $P_{i, 1}'$ and $P'_{i, 2}$ be the component of $\H^2 \minus \phi_i^{-1}(N_i)$.
For each $i = 1, 2, \dots$ and $j = 1, 2$, pick compact fundamental domains $D_{i, j}$ of $\Stab P_j \curvearrowright P_{i, j}'$, such that $D_{i, j}$ converges to a fundamental domain of the $\ep$-thick components of $\tau_\infty$.
Recall that $U_i$ is invariant under $\rho_i(m)$. 
Then,  for every $r > 0$, if $i$ is sufficiently large, both fundamental domains of $P_{i, 1}'$ or $P_{i, 2}'$ map to outside $U_i$ by $\beta_i$. 
Therefore, it follows from \Cref{NonelementaryComplements} and \Cref{PunctureElementaryHolonomy} that  $\rho_i | \Stab P_1$ or $\rho_i | \Stab P_2$ must diverge to $\infty$ up to a subsequence, against to the convergence of $\rho_i$.

Next we suppose (ii).
Without loss of generality, we can assume that $\beta_i(\ell_{i, 1})$,  not contained in $U_i$ but  $\beta_i(\ell_{i, 2})$ is contained in $U_i$ for sufficiently large $i$.
Then, for every $r > 0$, similarly, the fundamental domain $P_{i, 1}$ of $P_1'$ maps to outside $U_i$ by $\beta_i$ if $i$ is sufficiently large.
 Then, by the assumption of $\beta_i(\ell_{i, 2})$ being contained in $U_i$,  one can similarly show  $\rho_i | \Stab P_1$ diverges to $\infi$, up to a subsequence.
  \begin{figure}
\begin{overpic}[scale=.2
] {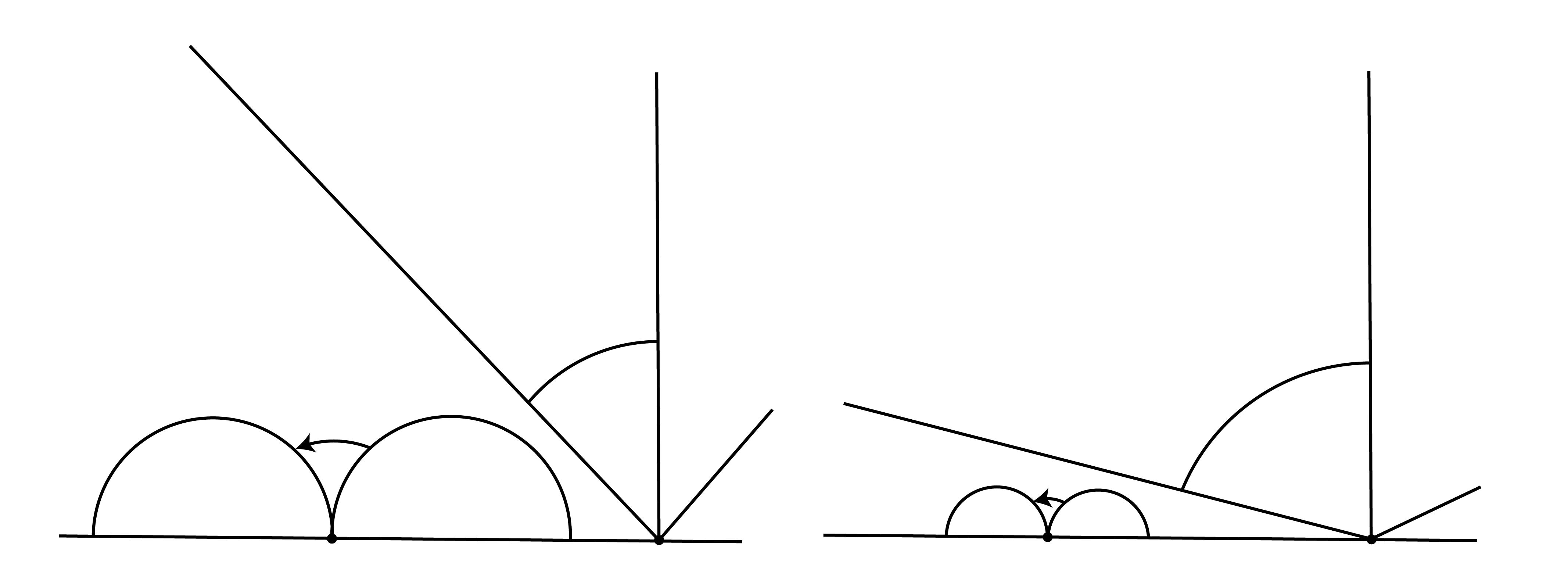} 
\put(70 , 20 ){\textcolor{black}{$U_i'(r')$}}  
\put(80 , 14 ){\textcolor{black}{$r'$}}  
\put(89 , 20 ){\textcolor{black}{$a_i'$}}  
\put(43 , 20 ){\textcolor{black}{$a_i$}}  
\put(37 , 16 ){\textcolor{black}{$r$}}  
\put(30 , 25 ){\textcolor{black}{$U_i(r)$}}  
\put(19 , 12 ){\textcolor{black}{$\rho_i(\omega)$}}  
      \end{overpic}
\caption{This figure illustrates the divergence of   $\rho_i | \Stab P_j$ in the upper half space model of $\H^3$.  The arrows indicate how the action by an element $\omega$ in $\Stab P_j$ changes, and it diverges as $i$ increases in $\PSL_2\C$, where $r' > r$ and $i' >  i$. }\label{}
\end{figure}
\end{proof}

If follows from \Cref{ThinPartImageBoundedFromGeodesic} that,
for every $\ep' > 0$, by taking $\del > 0$ sufficiently smaller than $\ep > 0$ above,  similarly letting $\ti{N}_i^{\del}$ be the $\rho_i(m)$-invariant lift of $N_i^\del$ to the universal cover $\H^2$, the image $\beta_i( \ti{N}_i^{\del})$ is $\ep'$-close to the axis $a_i$ for sufficiently large $i$. 

Recall that we have a convergence of $\beta_i \circ \kap_i$ on $P_1, P_2$ so that, in the limit, the boundary components of $\ti{A}$ map to the endpoints of $a_\infty$. 
Therefore, by taking an appropriate isotopy of $S$,  $\beta_i \circ \kap_i$ converges to a continuous map, up to a subsequence, such that $\ti{N}$ maps to $a_\infty$.

(Case Two)
Suppose that $a_\infi$ is a single point on $\CP^1$. 
  Pick any horoball  $B$ in $\H^3$ tangent at $a_\infi$.
For each $i$, pick a subset $U_i \sub \H^3$ converging to $B$ uniformly on compact subsets as $i \to \infty$,  such that, if $ \rho_i (\gam)$ is either hyperbolic or elliptic, then $U_i$ is an $r_i$-neighborhood of $a_i$ for some $r_i > 0$, and if $\rho_i(\gam)$ is parabolic, then $U_i$ is a horoball centered at the parabolic fixed point of $\rho_i(\gam)$. 

For sufficiently large $i$, Let $N_i$ be the  $\ep$-thin part of $\tau_i$ homotopy equivalent to $m$.  
Let $\ti{N}_i$ be a component of $\psi_i^{-1}(N_i)$.
\begin{lemma}
If $\ep > 0$ is sufficiently small, then $\beta_i (\ti{N}_i)$ is eventually contained in $U_i$ as $i \to \infi$. 
Therefore, $\beta_i \cc \kap_i | \ti{N}$ converges to the constant map to  the point $a_\infi$.
\end{lemma}

\begin{proof}
\begingroup  
 \color{dblue}
    \fontsize{13pt}{12pt}\selectfont

   \endgroup 
   
Let $P_1$ and $P_2$ be the components of $\ti{S} \minus \phi^{-1}(A)$ adjacent across the lift $\ti{A}$ of $A$ invariant by $\rho_i(m)$. 
Suppose, to the contrary, for every $\ep > 0$,  the image $\beta_i(\ti{N}^\ep_i)$ is not eventually contained in $U_i$. 
Then, at least one of $\beta_i(\ell_{i, 1})$ or $\beta_i(\ell_{i, 2})$ is not contained in $U_i$ for sufficiently large $i$. 
 Therefore, it follows from using \Cref{NonelementaryComplements} and \Cref{PunctureElementaryHolonomy} that either $\rho_i | \Stab P_1$ or $\rho_i | \Stab P_2$  diverges to $\infi$, up to a subsequence.
\end{proof}
\Qed{TrivialNeckPleatedSurface}

\subsubsection{Convergence of holomorphic quadratic differentials when $\rho_\infi(m) = I$}\Label{sLimitDifferentialIdentityHolonomy}
We next describe the limit quadratic differential. 
In the case that $\rho_\infi(m) = I$, the singular Euclidean structure $E_{t_i}$ contains a flat cylinder $A_t$ homotopic to $m$, such that $\Mod A_t \to \infi$ and the complex length of its circumference converges to a positive multiple of $\pi/ \sqrt{2}$, by  Proposition \ref{CuspClassification}. 
Therefore
\begin{proposition}\Label{Residue}
Let $C_\infi$ be the limit of $C_t$ in Theorem \ref{TrivialNeck} (\ref{iThickPartConvergence}).
Then,  the Schwarzian parameters of $C_\infi$ consist of a  Riemann surface with two punctures homeomorphic to $S \minus m$ and a holomorphic quadratic differential $q_\infi$, such that both punctures are a pole of order two and their residues are the same non-zero integer multiple of $\sqrt{2} \pi$.
\end{proposition}

\subsection{Non-discreteness of holonomy}\Label{sNondiscrete}
We in addition show the non-discreteness of the holonomy representation $\rho_t$ for large $t$.
\begin{theorem}\Label{EventuallyNondiscrete}
Suppose that $\rho_\infi(m) = I$.
Then $\Im \rho_t \sub \PSL_2\C$ is a non-discrete subgroup for sufficiently large $t > 0$.
\end{theorem}
\begin{proof}
Recall that $\rho_t(m) \to  I$ but $\rho_t(m) \neq I$ (Theorem \ref{TrivialNeck}(\ref{iNotExactlyI})).
For each component $F$ of $S \minus m$, $\rho_t( \pi_1(F))$ is nontrivial for sufficiently large $t \gg 0$ (Proposition \ref{NonelementaryComplements}). 
Recall from  \Cref{PunctureAtAxis} that, if $C_{t_i}$ converges to a $\CP^1$-structure of a punctured surface homeomorphic to $F$ for a diverging sequence $t_1 < t_2 < \dots$,  then, in the limit, its cusp point develops to an endpoint of the limit of the axis of $\rho_{t_i}$.
Therefore the subgroup of $\Im \rho_t$ generated by $\{\,\rho_t(m) \gamma \rho_t(m)^{-1} \, |\, \gamma \in \rho_t(F)\,\}$ is non-elementary since the endpoint  in $\CP^1$ is not preserve by some non-trivial element in  $\rho_t( \pi_1(F))$ by (Lemma \ref{PunctureNotAtAxis} and Proposition \ref{PunctureElementaryHolonomy}). 
As $\rho_t(m) \to I$, by the  Margulis lemma, $\Im \rho_t$ cannot be discrete. 
\end{proof}
\section{Examples of exotic degeneration}\Label{sExoticDegeneration}
We construct examples of a path $C_t = (f_t, \rho_t)$ of $\CP^1$-structures on $S$ asymptotically pinched along a loop $m$ as $t \to \infty$ such that $\rho_\infi(m) = I$ and  $[\rho_t]$ converges in $\rchi$ as $t \to \infi$, as in the second case of \Cref{LImitInDevelopingMap}. 
We construct two examples: one with $\rho_t(m)$ hyperbolic and one with $\rho_t(m)$ elliptic for all sufficiently large $t > 0$.
\subsection{Hyperbolic $\rho_t(m)$ converging to $I$.}

Let $E$ be the singular Euclidean surface obtained from an $L$-shaped polygon by identifying the opposite edges (Figure \ref{fLSahped}).
Then $E$ has exactly one cone point, and its cone angle is $6\pi$.
Let $F$ be the underlying topological surface of $E$, which is a closed surface of genus two.   
Let $E'$ denote  $E$ minus the cone point, and let $F'$ denote the underlying topological surface of $E'$. 
\begin{figure}
\begin{overpic}[scale=.05
] {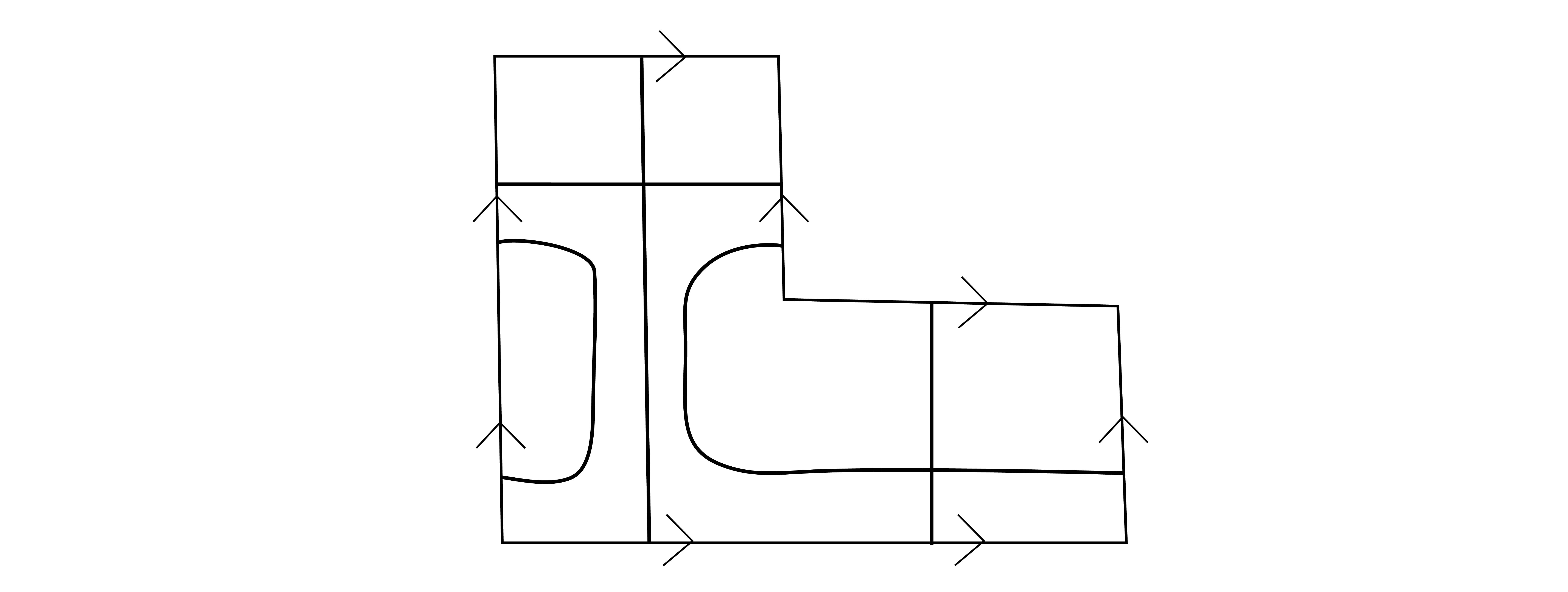} 
 \put(37.9 , 30 ){$a_1$}  
\put(60 , 13 ){$a_2$}  
 \put(45 , 27 ){$b_1$}  
 \put(46 , 10 ){$b_2$}   
      \end{overpic}
\caption{}\label{fLSahped}
\end{figure}
Let $\ell_p$ be the (oriented) peripheral loop around the removed cone point. 
Let $\xi \col \pi_1(F') \to \PSL_2\C$ be the holonomy of $E'$.
Then, as $F'$ has a Euclidean structure, the image of $\xi$ consists of parabolic elements, and we can assume that its image consists of upper triangular matrices with 1's on the diagonal. 
In particular $\xi(\ell_p) = I$ (as before, by abuse of notation, we regard $\ell_p$ also as a fixed element of $\pi_1(S)$ by picking a basepoint of $\pi_1(S)$ on $\ell$.)
Notice that there is a point in the universal cover $\ti{E}$ of $E$ corresponding to  $\ell_p \in \pi_1(S)$. 
(Namely, by lifting $\ell_p$ to a loop in the universal cover $E$ starting from the base point, there is a unique cone point of $\ti{E}$ in the disk region bounded by the lift.) 
\begin{proposition}\Label{PuncturedSurfaceWithHyperbolicCuspHolonomy} 
There is a path of $\CP^1$-structures, $D_t = (h_t, \xi_t)$,  on $F'$ converging to $E' = (h, \xi)$ as $t \to \infi$,  such that  $\xi_t(\ell_p)$ is a hyperbolic translation whose axis converges to a geodesic connecting the global (parabolic) fixed point of $\xi$ and the $h$-image of the corresponding singular point of $\til{E}$.
\end{proposition}
\proof
Note that elements of $\Im \xi$ are translations of $\C$.
Pick  non-separating simple closed curves  $a_1, b_1, a_2, b_2$ on $E$ as in Figure \ref{fLSahped} forming a standard generating set of $\pi_1(F)$ so that 
\begin{itemize}
\item for each $i = 1, 2,$  $a_i$ and $b_i$ intersect in a single point, and $[a_1, b_1][a_2, b_2] = I$,
\item the translation directions of $a_1$ and $a_2$ are the same and the translation direction of $b_1$ and $b_2$ are the same, and 
\item  the translation directions of $a_i$ and $b_i$ are orthogonal for each $i = 1, 2$.
\end{itemize}

Let $c$ be a separating loop on $E$ which separates $\{a_1, b_1\}$ and $\{a_2, b_2\}$. 
Then, let $F_1$ and  $F_2$ be the components of $F \minus c$ which are homeomorphic to a torus minus a disk. 
\begin{lemma}\Label{DeformPuncturedTorus}
Let $q_i$ be any geodesic in $\H^3$ starting from the global fixed point $p \in \CP^1$ of $\Hol E$, and let $H_i$ be the hyperbolic plane, in $\H^3$, containing an $\langle a_i \rangle$-orbit of $q_i$. 
For each $i = 1, 2$, 
given any path $h_{i, t}\, (t \geq 0)$ of hyperbolic elements in $\PSL_2\C$ such that 
\begin{enumerate}
\item the axis of $h_{i, t}$ is orthogonal to $H_i$ at a point in $q_i$ for all $t \geq 0$, and
\item $h_{i, t} \to I$ as $t \to \infi$.
\end{enumerate} 
Then, there is a path $\zeta_{i, t} \col \pi_1(F_i) \to \PSL_2\C$ of homomorphisms which converges to the restriction of $\Hol(E)$ to $\pi_1(F_i)$  as $t \to \infty$ such that $\zeta_{i, t}(c)= h_{i, t}$.
\end{lemma} 

\begin{proof}
The point $p$ is contained in the ideal boundary of $H_i$. 
Let $r_i$ be a geodesic in $H_i$, such that $R(r_i) R(q_i) = \xi(a_i)$, where  $R(r_i)$ and $R(q_i)$ are the $\pi$-rotations of $\H^3$ about $r_i$ and $q_i$, respectively (Figure \ref{fPantsHolonomy}, Right). 

 \begin{figure}
\begin{overpic}[scale=.5,
] {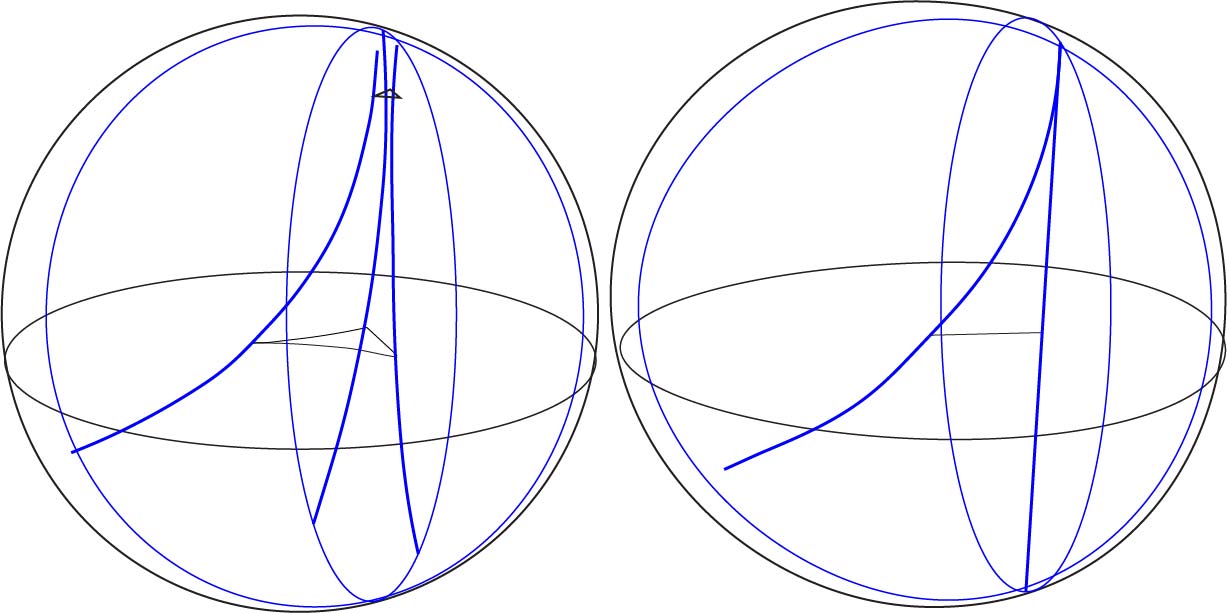} 
 \put(22 ,15 ){\textcolor{Blue}{$q_{t, i}$}}  
\put(33,16 ){\textcolor{Blue}{$q_{t, i}'$}} 
  \put(12 , 20 ){\textcolor{Blue}{$r_{t, i}$}}  
 \put(86 ,24 ){\textcolor{Blue}{$q_i$}}  
  \put(69 , 20 ){\textcolor{Blue}{$r_i$}}  

 \put(32 , 50 ){\textcolor{black}{$p$}}  
  \put(60 , 30 ){\textcolor{Blue}{$H_i$}}  
\put(78 , 16){\textcolor{Blue}{$H_i^\perp$}}  
      \end{overpic}
\caption{}\label{fPantsHolonomy}
\end{figure}

Let $H_i^\perp$ be the hyperbolic plane in $\H^3$ orthogonal to the hyperbolic plane $H_i$ along the geodesic $q_i$. 
As $\Axis(h_{i, t})$ is in $H_i^\perp$ and orthogonal to $q_i$,  we let $q_{i, t}$ and $q_{i, t}'$ be continuous paths of  geodesics in $H_i^\perp$ such that 
$R(q_{i, t}) R(q_{i, t}') = h_{i,t}$, the geodesics  $q_{i, t}$ and $q_{i, t}'$ converge to $q_i$ as $t \to \infi$ uniformly on compact subsets, and  the $\pi$-rotation $R(q_i)$ exchanges $q_{i, t}$ and $q_{i, t}'$. 
By this symmetry, there is a path of geodesics $r_{i, t}$ in $H_i$ such that, for all $t \gg 0$, 
\begin{itemize}
\item there is a hyperbolic plane intersecting $r_{t, i}, q_{t, i}, q_{i, t}'$ orthogonally, and
\item $d_{\H^3}(r_{i, t}, q_{i, t}) = d_{\H^3}(r_{i, t}, q_{i, t}')$.
\end{itemize}
Thus  by the symmetry, $\tr R(q_{i, t}) R(r_i) = \tr R(q_{i, t}')R(r_i) \in \R \minus [-2,2]$. 

The surface $F_i \minus a_i$ is a pair of pants, and two of its boundary components correspond to $a_i$.  
Consider the  path of homomorphisms $\zeta_{i, t} \col \pi_1 (F_i \minus a_i) \to \PSL_2\C$ for $t > 0$, such that the two boundary components corresponding to $a_i$ map to  $R(q_{i, t}) R(r_i)$ and $ R(r_i)R(q_{i, t}')$--- thus the other boundary corresponding to  $\bdr F_i $ maps to $R(q_{i, t})R(q_{i, t}') = h_{i, t}$ (see \cite{Goldman09FrickeSpaces}).
Then, by Theorem \ref{HolImmersion},  there is a path of $\CP^1$-structures on $F_i \minus a_i$ with holonomy $\zeta_{i, t}$ which converges to the component of $E \minus (c \cup a_i)$ as $t \to \infi$ corresponding to $F_i \minus a_i$.
As the holonomies along the two boundary components are conjugate, for large enough $t > 0$, there is a path of $\CP^1$-structures $\Sigma_{i, t}$ on $F_i$ which converges to the component of $E \minus c$, so that $\Hol \Sigma_{i, t} | \pi_1 F_i  = \zeta_{i, t}$.
In particular, the holonomy  of $\Sigma_{i, t}$ around the puncture is the hyperbolic element $R(q_{i, t}) R(q_{i, t}') = h_{i, t}$.
\end{proof}

\begin{figure}

\vspace{5mm}
\begin{overpic}[scale=.07
] {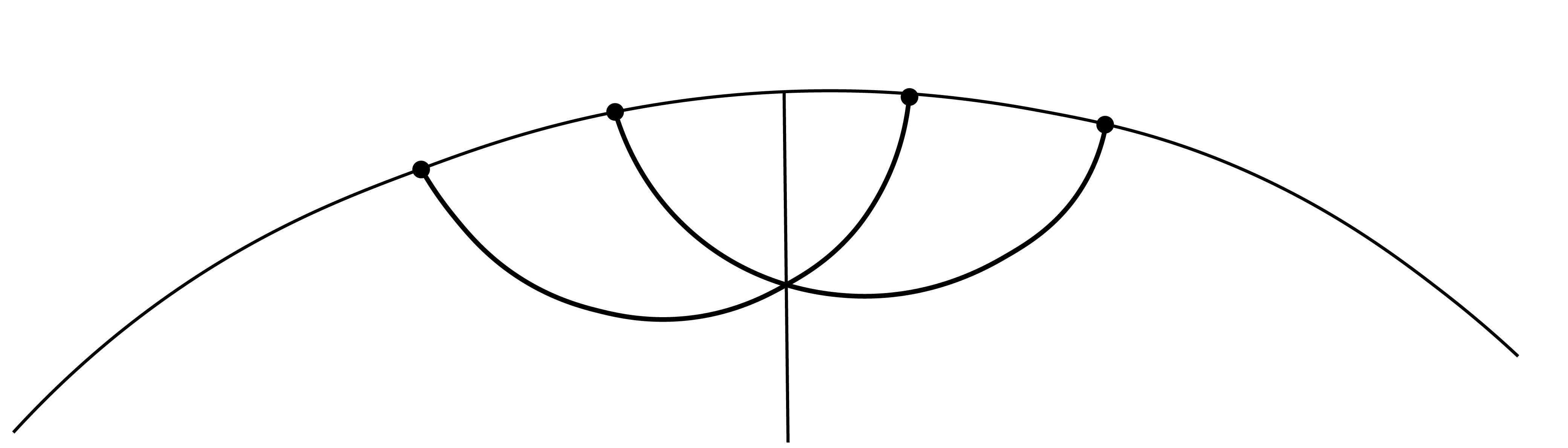} 
   \put(37 , 25 ){\contour{white}{$v_{1, t}$}}
   \put(71 , 22 ){$u_{1, t}$}
   \put(57 ,   25 ){$v_{2, t}$}
   \put(22 , 21 ){\contour{white}{$u_{2, t}$}}
    \put( 51, 2){$q$}
      \end{overpic}
\caption{}\label{fAxesOfh}
\end{figure}

Notice that $H_1$ and $H_2$ are totally geodesic hyperbolic planes in $\H^3$ tangent at $p$. 
Therefore we can, in addition, assume that $H_1$ and $H_2$ are different and $H_1^\perp = H_2^\perp \eqqcolon H$. 
Pick a geodesic $q$  in $H$ initiating from $p$ contained in the region bounded by the geodesics $q_1 = H \cap H_1$ and $q_2 = H \cap H_2$. 
\begin{proposition}\Label{AxisOfCompsitionOfSmallHyperbolics}
 We can choose, the path of the hyperbolic isometries  $h_{1, t}, h_{2, t}$ (given by Lemma \ref{DeformPuncturedTorus})  so that their composition  $h_{1, t} h_{2, t}$ is eventually a hyperbolic element whose axis converges to $q$ as $t \to \infi$. 

\end{proposition}
\begin{proof}
Pick $h_{1, t}$ and $h_{2, t}$ such that their axes converge to the parabolic fixed point $p$.
Since $h_{1, t}$ and $h_{2, t}$ converge to $I$, their product $h_{1, t} h_{2, t}$ also converges to $I$ in $\PSL_2\C$.

For each $i = 1, 2$, let  $u_{i, t}$ be the attracting fixed point,  and let $v_{i, t}$ be the repelling fixed point of  $h_{i, t}$. 
We may first assume that the endpoints of $\Axis h_{1, t}, \Axis h_{2, t}$ lie on $\bdr H$ in this cyclic order $u_{2, t}, v_{1, t}, v_{2, t}, u_{1, t}$  (Figure \ref{fAxesOfh}).
The composition $h_{1, t} h_{2, t}$ fixes a point on the arc in $\bdr H$ between $v_{1, t}$ and $v_{2, t}$ for each $t > 0$.
Note that the segment contains $p$.
Then as $\Axis(h_{1, t}), \Axis(h_{2, t})$ converge to the parabolic fixed point $p$, there is a fixed point of   $h_{1, t} h_{2, t}$ converging to $p$. 
Moreover,  as $h_{2, t} \to I$, one can continuously adjust the translation length of $h_{1, t}$ so that  $h_{1, t} h_{2, t}$ also fixes  the other endpoint of $q$  for sufficiently large $t > 0$. 
Let $s$ be the endpoint of the geodesic $q$ which is {\it not} $p$.
Then, after this adjustment, clearly $h_{1, t} h_{2, t} (s) = s$  holds for all large $t > 0$ and $h_{2, t} (s) \to s$  as $t \to \infty$.
Since the axis of the hyperbolic element $h_{1, t}$ converges to the ideal point $p (\neq s)$, the translation length of $h_{1, t}$ must converge to zero; thus $h_{1, t}$ converges to the identity (so that the condition (2) in \Cref{DeformPuncturedTorus} remains satisfied after the modification). 

Clearly the composition $h_{1, t} h_{2, t}$ does {\it not} fix the endpoints of the axes of the hyperbolic elements $h_{1, t}$ and  $h_{2, t}$ for all large $t > 0$.
Therefore $h_{1, t} h_{2, t}$ is a hyperbolic element with the axis $q$ for sufficiently large $t > 0$, which is not the identity. 
\end{proof}

Let $h_{1, t}, h_{2, t} \in \PSL_2\C$ be the paths given by Proposition \ref{AxisOfCompsitionOfSmallHyperbolics}.
Then,  by Lemma \ref{DeformPuncturedTorus}, for each $i = 1, 2$,  we have a path of homomorphisms $\zeta_{i, t}\col \pi_1(F_i)  \to \PSL_2\C$ such that $\zeta_{i, t}(\ell) = h_{i, t}$ for $t \gg 0$.
Then there is a unique path $\zeta_t \col \pi(F') \to \PSL_2\C$ so that $\zeta_t | \pi_1(F_i) = \zeta_{i, t}$ for $i =1, 2$; thus $\zeta_t (\ell_p) = h_{1, t}h_{2, t}$.
Then, by the holonomy theorem (\Cref{HolImmersion}),  there is a path $D_t$ of $\CP^1$-structures on $F'$ with holonomy $\zeta_t$ for $t \gg 0$ such that $D_t$ converges to $E'$ as $t \to \infi$.  \Qed{PuncturedSurfaceWithHyperbolicCuspHolonomy}
    \begin{remark}\Label{rConveringToTrivialOne}
    Since $\xi_t$ converges to the parabolic representation $\xi$ and the axis of the hyperbolic element $\rho_t(\ell_p)$ converges to a geodesic starting from the parabolic fixed point of $\xi$ as $t \to \infty$, 
by normalizing by an appropriate power $r_t$ of isometries $\xi_t(\ell_p)$, the conjugation $\xi_t(\ell_p)^{r_t}\cdot \xi_t \cdot \xi_t(\ell_p)^{-r_t}$ converges to the trivial representation, and the developing map $\xi_t(\ell_p)^{r_t} h_t$ converges to the constant map to the endpoint of $q$ which is not $p$. 
\end{remark}   
\subsubsection{Constructing a closed surface from punctured surfaces}
To make a desired example of exotic degeneration, 
we take two copies $D_t$ of $\CP^1$-surfaces with a single puncture from \Cref{PuncturedSurfaceWithHyperbolicCuspHolonomy}, and glue them together with many twists. 
\begin{theorem}\Label{ExoticDegenerationHyperbolic}
There is a path of $\CP^1$-structures  $C_t = (f_t, \rho_t)$ on a closed surface $S$ of genus four with following properties:
\begin{itemize}
\item The conformal structure $X_t$ is pinched along a separating loop $m$ as $t \to \infty$; let $F_1$ and $F_2$ be the connected components of $S \minus m$. 
\item $\rho_t\col \pi_1(S) \to \PSL_2\C$ converges in the representation variety as $t \to \infi$. 
\item Pick an element $\gam \in \pi_1(S)$ whose free homotopy class is $m$. Then  $\rho_\infi(\gam) = I$, and, for all  $t > 0$, the holonomy  $\rho_t(\gam)$ is a hyperbolic element such that its axis $a_t$ converges to a geodesic $a_\infi$ in $\H^3$ as  $t \to \infi$. 
\item Let $\ti{F}_1, \til{F}_2$ be the connected components of $\ti{S} \minus \phi^{-1}(m)$ which are adjacent across the lift $\ti{m}$ of $m$ preserved by $\gam \in \pi_1(S)$ .
Then 
$C_t  | \til{F}_1$ converges to the developing map of a $\CP^1$-structure on a genus two surface minus a point such that the cusp maps to an endpoint of $a_\infi$ as $t \to \infi$. 
\item $f_t | \til{F}_2$ converges to the constant map to the other endpoint of $a_\infi$ uniformly on compact subsets,  and $\rho_\infi | \pi_1(F_2)$ is the trivial representation.
\end{itemize}
\end{theorem}
\begin{remark}
In fact, $\Im \rho_\infty$ consists of parabolic elements with a global fixed point on $\CP^1$, and
therefore the limit representation $\rho_\infty$ is identified with the trivial representation in the character variety $\rchi$.
In other words, the frontier of  $\PSL_2\C$-orbit of $\rho_\infty$  contains the trivial representation. 
Thus, there is a path $\alpha_t \,(t > 0)$ in $\PSL_2\C$ such that $\alpha_t \rho_t \alpha_t^{-1}$ converges to the trivial representation. 
\end{remark}
\begin{proof}
For sufficiently large $t > 0$, the $\CP^1$-structure $D_t$ with a single puncture from \Cref{PuncturedSurfaceWithHyperbolicCuspHolonomy} has a cusp neighborhood $N_t$ foliated by admissible loops whose developments are invariant under the one-dimensional subgroup $G_t$ of $\PSL_2\C$ containing $\xi_t(\ell_p)$. 
We can assume that $N_t$ changes continuously in $t$ and is asymptotically the empty set on $E'$ as $t \to \infi$. 
Note that as $G_t$ is a one-dimensional subgroup $\rho_t(m)$ of $\PSL_2\C$,  integer powers $\rho_t(m)^n$ for $n \in \Z$ continuously extends to real powers. 

First take two copies $\Sigma_{1, t}, \Sigma_{2, t}$ of $D_t \minus N_t$ and,   since the boundary of $N_t$ are invariant by the one-parameter subgroup $G_t$, 
 glue them together along their boundary components without adding a twist.
Let $C_t' = (f_t', \rho_t')$ be the resulting developing pair. 
Then we can normalize by $\PSL_2\C$ so that the axis of the hyperbolic element $\rho_t'(m)$ is the geodesic $q$ for all $t$.
 In addition, we can renormalize the developing pair by $\PSL_2\C$ so that the restriction of $f_t'$ to $\ti{F}_1$  and the restriction of $\rho_t'$ to the stabilizer $\Stab \ti{F}_1$ of $\ti{F}_1$ in $\pi_1(S)$ converges to a developing pair of $E'$ as $t \to \infty$.
Then, as $N_t$ converges to the empty set, the restriction of $\rho_t'$ to $\Stab \ti{F_2}$ leaves every compact in the representation variety, and the restriction of $f_t'$ to $\ti{F}_2$ does not converge to a continuous map as $t \to \infty$. 

Recall that the holonomy $\rho_t'$ along $m$ is a hyperbolic element with axis $q$, and the translation length of $\rho_t'(m)$ goes to zero as $t \to \infty$.
Therefore, when we glue $\Sigma_{1, t}, \Sigma_{2, t}$ of $D_t \minus N_t$, we can continuously add more and more twists along $m$,  which conjugates the structure on $F_2$ by $\rho_t'(m)$ raised to the power of the amount of twist along $q$, so that

\begin{itemize}
\item the restriction of $f_t'$ to $\ti{F}_1$  and the restriction of $\rho_t'$ to $\Stab \ti{F}_1$ still converges to a developing pair for $E'$, and 
\item the restriction of $\rho_t'$ to $\Stab \ti{F}_2$ converges to the trivial representation, and the restriction of $f_t'$ to $\ti{F}_2$ converges to the constant map to the other endpoint of $q$ (by \Cref{rConveringToTrivialOne}) as $t \to \infty$.
\end{itemize}
We obtained a desired path $C_t'$. 
   \end{proof}
\subsection{Elliptic $\rho_t(m)$ converging to the identity}\Label{sExoticDegenerationSmallElliptic}
In this section, we construct an example of  $C_t = (f, \rho_t)$ in Theorem \ref{LImitInDevelopingMap} (ii) such that  $\rho_t(m)$ is an elliptic element for all sufficiently large $t > 0$ and it converges to $I$ as $t \to \infi$.

Given an elliptic element $e \in \PSL_2\C$, normalize the unit disk model $\D^3 \sub \R^3$ of $\H^3$ centered at the origin, so that $\Axis(e)$ is contained in the axis of the third coordinate. 
Let $\zeta \in (0, 2\pi)$ be the rotation angle of $e$. 
Then, define $b_e \col \R \to \bdr \H^3$ by $x \mapsto (\cos (\zeta x )\sin x, \sin (\zeta x) \sin x, \cos x)$ which is equivariant under $\Z \to \langle e \rangle$.

\begin{lemma}\Label{EllipticPath}
Let $r$ be a geodesic in $\H^3$. 
Pick a parallel vector field $V \sub T \H^3$ along $r$ such that $V$ is orthogonal to $r$. 
Then,  there are a path of (nontrivial) elliptic elements $e_t \in \PSL_2\C$ and a continuous function $\theta_t \in \R_{\geq 0}$ in $t > 0$ which satisfies the following: 
\begin{itemize}
\item $e_t \to I$ as $t \to \infi$.
\item $\Axis (e_t)$ orthogonally intersects $r$, and $\Axis(e_t)$ converges to an endpoint of $r$ on $\CP^1$ as $t \to \infi$.
\item Letting $\theta_t \in \R$ be a continuous function such that the angle between $\Axis e_t$ and $V$ is $\theta_t  \mod 2 \pi$, when an orientation of $Axis(e_t)$ is fixed continuously in $t$.
\item Let $u_t = 2 \theta_t$. Then the rotation angle of $e_t^{u_t}$ is $\pi$ for all $t \geq 0$, so that $e_t^{u_t}$ takes $r$ to itself, reversing the orientation. 
\end{itemize} 
\end{lemma}
\begin{proof}
It is easy to construct an example satisfying the first three conditions. 
Then adjust the rotation angle of $e_t$ so that it also satisfies the last condition. 
\begingroup  
 \color{dblue}
    \fontsize{13pt}{12pt}\selectfont
   \endgroup 
\end{proof}

\begin{lemma}\Label{EllipticCylinder}
Let $e_t$ be as in Lemma \ref{EllipticPath}.
Let $p$ be the endpoint of $r$ to which $\Axis(e_t)$ converges. 
Pick a round disk $D$ in $\CP^1$ containing $p$ such that the hyperbolic plane in bounded by the boundary of $D$ is orthogonal to the geodesic $r$.
Then, there is a path $A_t$ of $\CP^1$-structures on an  annulus $A$ with smooth boundary for sufficiently large $t \gg 0$, such that 
\begin{itemize}
\item $A_t$ converges to the once-punctured disk $D \minus \{p\}$ as $t \to \infty$ as a $\CP^1$-structure, and
\item the developments of the both boundary components of $A_t$ are curves equivalent to $b_{e_t}$ by elements of $\PSL_2\C$. 
\end{itemize}
\end{lemma}
\begin{proof}
For sufficiently large $t > 0$, one can easily construct the fundamental membrane for $A_t$ for sufficiently large $t > 0$ (Figure \ref{fMenbrance}). 
\begin{figure}
\begin{overpic}[scale=.18
] {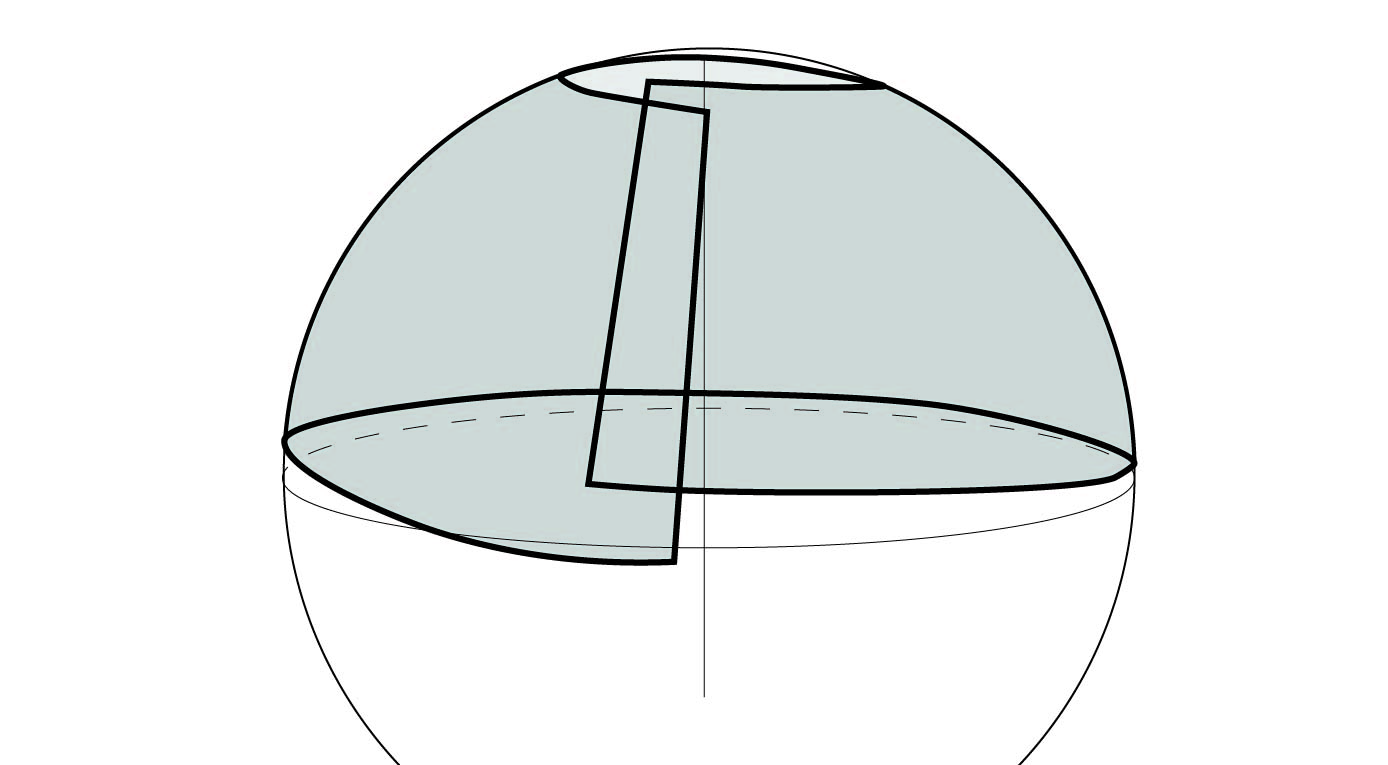} 
      \end{overpic}
\caption{}\Label{fMenbrance}
\end{figure}
\end{proof}

\begin{proposition}\Label{CylinderWithParabolicPuncture}
Let $P$ be a pair of pants, and pick a boundary component $\ell$ of $P$.
Let $\ti{\ell}$ be a lift of $\ell$ to the universal cover of $P$.
Consider a (flat) Euclidean cylinder with geodesic boundary, and let $P_\infty$ be the surface obtained by removing an interior point $p$ of $P_\infi$; regard $P_\infi$ as a $\CP^1$-structure  on $P$, and let $(h, \xi)$ be its developing pair, so that $h$ takes $\ti\ell$ to a single point $v$ on $\CP^1$.

Let $r$ be the geodesic  in $\H^3$ connecting $v$ and the parabolic fixed point of $h$, and 
let $e_t \in \PSL_2\C$ be a path of (non-trivial) elliptic elements given by Lemma \ref{EllipticPath} for $r$.

Then, there is a path of $\CP^1$-structures $P_t = (h_t, \xi_t)$ on $P$ satisfying the following:
\begin{enumerate}
\item  For all $t > 0$,
$\xi_t(\ell) = e_t$.
\item $P_t$ converges to $P_\infi$ as $t \to \infi$.
Let $\gam_t \in \PSL_2\C$ be a path of hyperbolic elements with the axis $r$, such that $\gam_t \Axis(e_t)$ converges to a geodesic $g_\infi$ in $\H^3$ orthogonal to $r$ as $t \to \infi$ (so that $\gam_t$ is a large hyperbolic translation towards $v$ for $t \gg 0$). 
Let $H \sub \H^3$ be the totally geodesic hyperbolic plane orthogonal to $r$ and containing $g_\infi$.
Then, the developing pair $\gam_t (h_t, \xi_t)$ normalized by $\gam_t$  converges to a developing pair for a round disk minus a point, where the removed point is $v$ and the disk is the component of $\CP^1 \minus \bdr H$ containing $v$.
\Label{iNearAxisAssympotiticallyCylinder}
\item Let $\ell_t$ be the boundary component of $P_t$ corresponding to $\ell$. 
Then $\dev P_t$ along a lift of $\ell_t$ is  $b_{\ell_t}$ (up to $\PSL_2\C$).  \Label{iNiceBoundary}
\item Let $\alpha$ be a boundary component of $P$ not equal to $\ell$.
Then $\xi_t(\alpha)$ is a hyperbolic element for all $t \gg 0$ (converging to a parabolic element as $t \to \infi$).
\end{enumerate}
\end{proposition}

\begin{proof}
First we construct an appropriate path of representations $\xi_t\col \pi_1(P) \to \PSL_2\C$.
Let $a_t$ denote  $\Axis(e_t)$.
Pick a pair of geodesics $q_t, q_t'$ in $\H^3$ for each $t > 0$ such that 
\begin{itemize}
\item  $R(q_t) R(q_t') = e_t$, where $R(q_t), R(q_t') \in \PSL_2\C$ are the $\pi$-rotations of $\H^3$ about $q_t, q_t'$, respectively;
\item $q_t$ and $q_t'$ change continuously in $t > 0$; 
\item $q_t$ and $q_t'$ intersect at the intersection $a_t \cap r$;
\item $q_t$ and $q_t'$ are symmetric about $r$;
\item $q_t$ and $q_t'$ are orthogonal to $a_t$;
\item $q_t$ and $q_t'$ converge to $r$ as $t \to \infi$ (see Figure \ref{fAlmostEuclideanCylinderMinustPoinst}).
\end{itemize}
There is a path of geodesics  $h_t \,(t \geq 0)$ in $\H^3$ such that 
\begin{itemize}
\item $h_t$ is disjoint from $q_t$ and $q_t'$ for all $t \geq 0$, and  
\item $h_t$ converges to a geodesic $h$ in $\H^3$ sharing an endpoint with $r$ as $t \to \infi$, such that the composition $R(r) R(h)$ is the parabolic holonomy along a boundary geodesic of $P_\infi$.
\end{itemize} 
Indeed one can first find the limit geodesic $h$ which satisfies the second condition, then as $q_t, q_t'$ converges to $r$, one can take a desired path $h_t$.

Let $\xi_t\col \pi_1(P) \to \PSL_2\C$ be such that the holonomy along boundary components are $R(h_t) R(q_t), R(q_t) R(q_t'), R(q_t') R(h_t)$. 
Note that  $R(h_t) R(q_t), R(q_t') R(h_t)$ are hyperbolic elements, as the rotation axes are disjoint, and they converge to the parabolic holonomy along the boundary geodesics of $P_\infi$.

Pick a round disk $D$ on $P_\infi$ containing $p$ such that $\bdr D$ on $\CP^1$ bounds a hyperbolic plane in $\H^3$ orthogonal to $r$.
Then, apply Lemma \ref{EllipticCylinder} to $D$, let $D_t$ be a path of $\CP^1$-structures on an annulus converging to $D \minus \{p\}$, so that it gives the desired path only near the punctured of $P_\infi$. 

Pick a smaller closed regular neighborhood $D'$ of the puncture $p$ of $P_\infi$ such that $\bdr D'$ bounds a hyperbolic plane orthogonal to $r$ and that $D'$ is contained in the interior of $D$. 
 Clearly its complement $K$ in $P_\infi$ and the interior of $D \minus \{p\}$ form an open cover of  $P_\infi$. 
 Then $K$ is topologically a pair of pants. 
Similarly to the proof of \Cref{HolImmersion} using the stability of transversal sections for the Thurston-Ehresmann principle  (\cite{Goldman22GeometricStructuresOnManifolds}), we can prove that there is a path of $\CP^1$-structures on a pair of pants $K_t$ for sufficiently large $t > 0$ such that   $K_t$ converges to $K$ and $e_t$ is the holonomy of $K_t$ around the boundary component corresponding to $\bdr D'$. 
Moreover, by deformation nearly the boundary, we can in addition assume that the boundary of $K_t$ is equivalent to $b_{\ell_t}$.

Then,  since $K$ and $D \minus \{p\}$ form an open cover of $P_\infty$, for sufficiently large $t$.  by gluing  $K_t$ and $A_t$ in the overlapping region, we obtained a desired path of $\CP^1$-structures $P_t$. 
\begin{figure}
\begin{overpic}[scale=.2, 
] {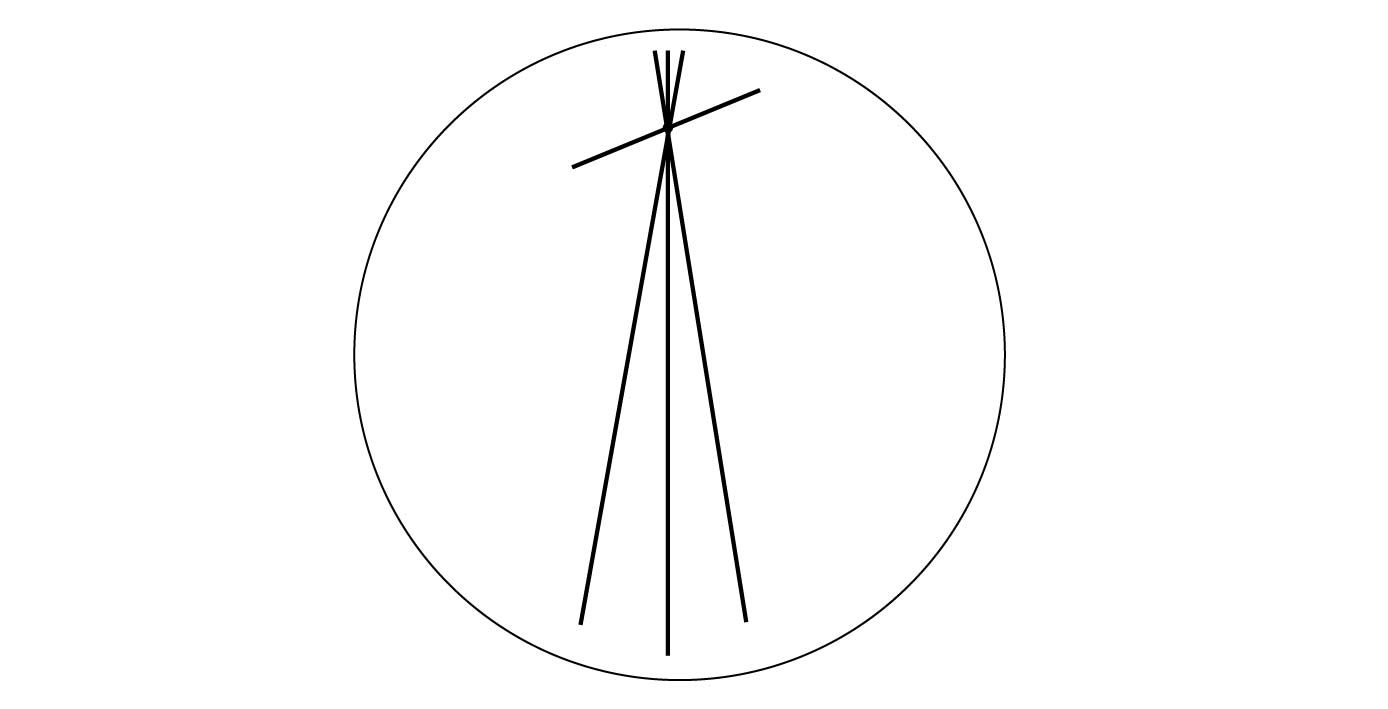} 
   \put(41 , 25 ){$q_t$}  
  \put(51 , 29 ){$q_t'$}
  \put(48.5 ,16 ){\contour{white}{$r$}}  
  \put(32 ,72 ){\textcolor{Blue}{$a_t$}}  
   \put(60 , 20 ){\textcolor{Black}{$\H^3$}}  
      \end{overpic}
\caption{Realize $e_t$ as the compositions of the $\pi$-rotations about $q_t$ and $q_t'$.}\label{fAlmostEuclideanCylinderMinustPoinst}
\end{figure}
\end{proof}
\begin{proposition}\Label{ASingleSmallEllipticPuncture}
Let $P_t = (h_t, \xi_t)$ be a path of $\CP^1$-structures on a pair of pants from \Cref{CylinderWithParabolicPuncture}.
Then, there is a path $\Sigma_t$ of $\CP^1$-structures on a closed surface $F$ minus a point which satisfies the following: 
 \begin{itemize}
\item There is a subsurface $A$ of $F$ whose interior contains $p$, such that  $A$ is homeomorphic to a pair of pants, and  $\Sigma_t|A = P_t$ for all large enough $t > 0$.
\item $\Sigma_t$  converges to a $\CP^1$-structure $\Sigma_\infty$ on $F$ as $t \to \infi$. 
\end{itemize}
\end{proposition}

\begin{proof}
First we construct the limit structure  $\Sigma_\infi$.
Take any complete hyperbolic surface $\tau$ with a single cusp, such that $\tau$ is homeomorphic to a closed surface minus a point, denoted by $F'$. 
Pick a cusp neighborhood $N$ of $\tau$, a horodisk quotient.
The pair of pants $P_\infi$ has two boundary components and one puncture. 
As the two boundary components of $P_\infi$ lift to horocycles, 
we can glue a copy of $\tau \minus N$ along each boundary component of $P_\infi$. 
We thus obtained a $\CP^1$-structure on a closed surface with a single puncture so that $P_\infi$ is its subsurface. 

There are paths $\zeta_{1, t}$ and $\zeta_{2, t}$ of representations $\pi_1(\tau) \to \PSL_2\C$ which converge to the holonomy of $\tau$ as $t \to \infi$, such that their images of the peripheral loop are $R(r_t)R(q_t') $ and $R(q_t)R(r_t)$, respectively, which are hyperbolic elements  (c.f.  \cite{Goldman09FrickeSpaces}).
Let  $\tau_{1, t}, \tau_{2, t}$ be paths of $\CP^1$-structures homeomorphic to $\tau \minus N$ for $t \gg 0$ such that 
$\Hol(\tau_{1, t}) = \zeta_{1, t}$, and $\Hol(\tau_{2, t}) = \zeta_{2, t}$ and  $\tau_{1, t}, \tau_{2, t}$  converge to $\tau \minus N$.
 We may in addition assume that the boundary components of  $\tau_{1, t}, \tau_{2, t}$ are invariant under one-dimensional subgroups of $\PSL_2\C$ containing  $R(r)R(q_t') $ and $R(q_t)R(r)$, respectively.
 
 Then by gluing $\tau_{t, 1}, \tau_{t, 2}, P_t$ along their boundary, we obtain a desired path $\Sigma_t$ of $\CP^1$-structures.
\end{proof}

Let $\Sigma_t$ be the path of $\CP^1$ -structures, obtained from  Proposition \ref{ASingleSmallEllipticPuncture},  on a compact surface with one boundary component.
Let $R_t$ be the $\pi$-rotation of $\H^3$ around the axis $a_t$ of the elliptic $e_t$.
By \Cref{CylinderWithParabolicPuncture}(\ref{iNearAxisAssympotiticallyCylinder}, \ref{iNiceBoundary}), we can glue two copies of $\Sigma_t$ by the involution $R_t$, and we obtain a path of $\CP^1$-structures $C_t$  on a closed surface, so that two copies of  $\Sigma_t$ are embedded in $C_t$ disjointly up to an isotopy.
    Let $m$ be the loop along which the two copies are glued. 
Then, to obtain a marked projective structure, we need to specify the twisting along $m$. 
We glue then so that the Fenchel-Nielsen twisting parameter matches to be $u_t$ so that, by the $\pi$-rotation along $a_t$, the developing maps of adjacent components of $\ti{S} \minus \ti{m}$  are identical.
Let $\Sigma_t^1 = (h_t^1, \rho_t^1), \Sigma^2_t = (h_t^2, \rho_t^2)$ are the subsurfaces of $C_t$ corresponding to $\Sigma_t$.
\begin{theorem}\Label{ExoticConvertenceEllipticNeck}
Let $C_t = (f_t, \rho_t)$ be the path of $\CP^1$-structures as above,  and let $m$ be the loop on $C_t$ corresponding to the boundary components of $\Sigma_t^1$ and $\Sigma_t^2$.
 Let $N$ be the regular neighborhood of $m$. 
 Then, by taking an appropriate isotopy of $S$, $C_t$ satisfies the following.
\begin{enumerate}
\item $\rho_t(m)$ converges to $I$ as $t \to \infi$, and $\rho_t(m)$ is an elliptic element for all  $t > 0$;
\item the axis of $\rho_t(m)$ converges to the point $p$ of $\CP^1$; \Label{iAxisLimit}
\item $f_t\col \til{S} \minus \phi^{-1}(N) \to \CP^1$ converges to a $\rho_\infi$-equivariant continuous map $f_\infi\col \til{S} \minus \phi^{-1}(N)\to \CP^1$, such that $f_\infi$ is a local homeomorphism in the interior; \Label{ConvergenceOfDevElliptic}
\item for each connected component $\ti{N}$ of $ \phi^{-1}(N)$, the boundary components of $\ti{N}$ map to its corresponding limit given by (\ref{iAxisLimit}).\end{enumerate}
\end{theorem}
\begin{proof}
Let $F_1, F_2$ be the connected components of $S \minus N$.
We normalize the developing pair of $C_t$ by a path of $\PSL_2\C$ so that the restriction to $\ti{F}_1$ converges to a developing pair for $\Sigma_\infty$.
Then (1) and (2) clearly hold. 
Moreover, we can take an appropriate isotopy of $S$ so that each boundary component of $\ti{F}_1$ converges to the corresponding limit point of its corresponding axis. 
Since the rotation angle of $e_t^{u_t}$ is $\pi$ by \Cref{EllipticPath}, the restriction of $f_t$  to $\ti{F}_2$ is the same as that to to $\ti{F}_1$ (Figure \ref{fPiRoration}). 
Therefore, the restriction of $f_t$ to $\ti{F}_2$ converges to a developing map of $\Sigma_\infty$ as well.  
Thus we have (\ref{ConvergenceOfDevElliptic}).
Then,  by the equivariant property, we also have (4).
\end{proof}
\begin{figure}
\begin{overpic}[scale=.2
] {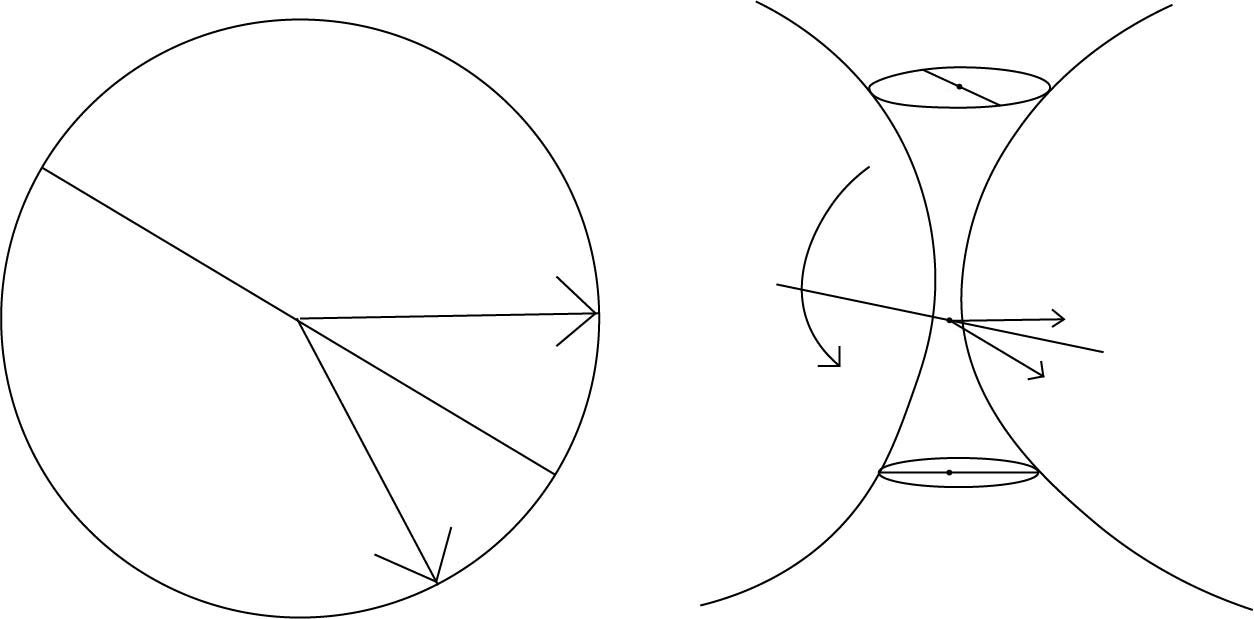} 
  \put(90 , 20){\textcolor{Black}{$a_t$}}  
  \put(58 , 30){    \contour{white}{$e_t^{u_t} = \pi$}    } 
  \put(15 , 30){\textcolor{Black}{$a_t$}}  
\put(75 , 1){\textcolor{Black}{$\Sigma_t^1$}}  
\put(69 , 46){\textcolor{Black}{$e_t^{-u_t}\Sigma_t^2$}}  
      \end{overpic}
\caption{
The left figure is a section of the right figure by a horizontal plane containing $a_t$.
It illustrates the rotation about $a_t$ by $\pi$, and it makes the restriction of $f_t$ on $F_1$ coincide with that to $F_2$ coincide.}\Label{fPiRoration}
\end{figure}
\Qed{ExoticConvertenceEllipticNeck}

\bibliography{DegenerationAlongAloop2024.11.bbl}
\bibliographystyle{alpha}

\end{document}